\documentclass[letterpaper,10pt,reqno]{amsart}
\usepackage[usenames,dvipsnames]{xcolor}

\RequirePackage[OT1]{fontenc}
\usepackage{underscore}
\usepackage[portrait,margin=3.5cm]{geometry}
\usepackage{mathrsfs}
\usepackage[colorlinks,citecolor=blue,urlcolor=blue,linkcolor=blue,linktocpage=true]{hyperref}
\usepackage[foot]{amsaddr}
\usepackage{amssymb,amsthm,amsfonts,amsbsy,latexsym,dsfont}
\usepackage[textsize=small]{todonotes}       
\usepackage{graphicx}
\usepackage[english]{babel}
\usepackage{subfigure}
\usepackage{appendix}
\usepackage{enumerate}

\usepackage{color}              
\usepackage{graphicx}
\usepackage[]{amsmath}
\usepackage{amsthm}
\usepackage{amssymb}
\usepackage{bbm}
\usepackage{hyperref}
\usepackage{comment}
\usepackage{microtype}

\definecolor{MyDarkblue}{rgb}{0,0.08,0.50}
\definecolor{Brickred}{rgb}{0.65,0.08,0}

\hypersetup{
colorlinks=true,       
    linkcolor=MyDarkblue,          
    citecolor=Brickred,        
    filecolor=red,      
    urlcolor=cyan           
}

\newtheorem{theorem}{Theorem}[section]
\newtheorem{lemma}[theorem]{Lemma}
\newtheorem{conjecture}[theorem]{Conjecture}
\newtheorem{proposition}[theorem]{Proposition}
\newtheorem{corollary}[theorem]{Corollary}

\newtheorem{definition}[theorem]{Definition}

\newtheorem{remark}[theorem]{Remark}
\newtheorem{claim}[theorem]{Claim}

\newcommand{\Pv}{\mathbb{P}}

\newcommand{\Ev}{\mathbb{E}}

\newcommand{\CC}{\mathcal{C}}

\newcommand{\bP}{\mathbb{P}}\newcommand{\bR}{\mathbb{R}}

\newcommand{\sss}{\scriptscriptstyle}

\newcommand{\CE}{{\mathcal{E}}}

\newcommand*{\SFPW}{{\mathrm{SFP}}_{W}}
\newcommand*{\SFPWL}{{\mathrm{SFP}}_{W,L}}

\newcommand{\e}{{\mathrm e}}

\numberwithin{equation}{section}

\newcommand{\R}{\mathbb{R}}
\newcommand{\N}{\mathbb{N}}
\newcommand{\Z}{\mathbb{Z}}

\renewcommand{\emptyset}{\varnothing}

\newcommand{\CA}{\mathcal {A}}
\newcommand{\CB}{\mathcal {B}}

\newcommand{\CF}{\mathcal {F}}
\newcommand{\CG}{\mathcal {G}}

\newcommand{\CN}{\mathcal {N}}

\newcommand{\CR}{\mathcal {R}}
\newcommand{\CS}{\mathcal {S}}

\newcommand{\CU}{\mathcal {U}}

\newcommand*{\wih}{\widehat}

\newcommand*{\la}{\lambda}

\newcommand*{\de}{\delta}

\newcommand*{\ve}{\varepsilon}

\newcommand*{\al}{\alpha}
\newcommand*{\ga}{\gamma}

\newcommand*{\be}{\begin{equation}}
\newcommand*{\ee}{\end{equation}}
\newcommand*{\ba}{\begin{aligned}}
\newcommand*{\ea}{\end{aligned}}
\newcommand*{\barr}{\begin{array}{c}}
\newcommand*{\earr}{\end{array}}
\def \toinp    {\buildrel {\Pv}\over{\longrightarrow}}
\def \toindis  {\buildrel {d}\over{\longrightarrow}}
\def \toas     {\buildrel {a.s.}\over{\longrightarrow}}

\newcommand*{\PoiBRW}{\mathrm{PoiBRW}}
\newcommand*{\BerBRW}{\mathrm{BerBRW}}

\newcommand*{\wit}{\widetilde}

\newcommand*{\ind}{\mathbbm{1}}
\def\namedlabel#1#2{\begingroup
    #2%
    \def\@currentlabel{#2}%
    \phantomsection\label{#1}\endgroup
}

\begin{document}
	\title[Explosion in scale-free percolation]{Explosion and distances in  scale-free percolation}

	\date{\today}
	\author[R.v.d. Hofstad]{Remco van der Hofstad}
\author[J. Komj\'athy]{J\'ulia Komj\'athy}
	\address{Department of Mathematics and
	    Computer Science, Eindhoven University of Technology, P.O.\ Box 513,
	    5600 MB Eindhoven, The Netherlands.}
	\email{j.komjathy@tue.nl}
\thanks{Acknowledgements. This work is supported by the Netherlands Organisation for Scientific Research (NWO)
through VICI grant 639.033.806 (RvdH), VENI grant 639.031.447 (JK), the Gravitation Networks grant 024.002.003 (RvdH)}

\begin{abstract}
We study the weighted scale-free percolation ($\SFPWL$) model on $\Z^d$. The vertices of $\Z^d$ are assigned independent and identically distributed (i.i.d.) vertex-weights $(W_x)_{x\in \Z^d}$ from a power-law distribution with exponent $\tau>1$. Conditioned on the vertex-weights, the edges $(x,y)_{x, y \in \Z^d}$ are present independently with probability that a Poisson random variable with parameter $\lambda W_x W_y/(\|x-y\|)^\alpha$ is at least one, for some $\alpha, \lambda>0$, and where $\|\cdot \|$ denotes the Euclidean distance.
After the graph is constructed, we assign i.i.d. random edge-weights from distribution $L$ to all edges present. The focus of the paper is to determine when is the obtained model \emph{explosive}, i.e., when it is possible to reach infinitely many vertices in finite time from a vertex.
We show that explosion happens precisely for those edge-weight distributions that produce explosive branching processes with infinite mean power-law offspring distributions. For non-explosive edge-weight distributions, when $\gamma=\al (\tau-1)/d \in (1,2)$, we characterise the asymptotic behaviour of the time it takes to reach the first vertex that is graph distance $n$ away. For $\gamma>2$, we show that the number of vertices reachable by time $t$ from the origin grows at most exponentially, thus explosion is never possible. 

For non-explosive edge-weight distributions, when $\gamma \in (1,2)$, we determine the first order asymptotics of distances when $\gamma\in (1,2)$. As a corollary we obtain a sharp upper and lower bound for graph distances,  an open problem in \cite{DeiHof13} when $\gamma \in(1,2), \tau>2$.
Distances in the explosive setting turn out to be hard, due to the infinite number of vertices and non-compactness of $\Z^d$. This is in contrast with explosion in similar random graph models such as the configuration model, where typical distances converge to the sum of two explosion times \cite{BarHofKomdist}.  
\end{abstract}

\maketitle

\section{Introduction and results}

\subsection{Introduction} 
How long does it take for a video to go viral? Can we predict when the flu epidemic in the Netherlands will reach Japan?
These questions are emblematic for the study of this paper: spreading processes on spatial complex networks. Spreading processes on complex networks can be found in many aspects of our life. By understanding them, we can understand how infectious diseases spread, or which individuals should be targeted with advertisements to reach the most people at minimal cost. 
 Statistical analyses \cite{AlbBar02, AlbJeoBar99a, Newm03a} suggest that real-life networks tend to be \emph{scale-free}, i.e., their degree sequences follow a \emph{power law}. That is, $\Pv(\deg(v) > x) \approx x^{-\gamma}$, typically for some $\gamma\in(1,2)$, where $\deg(v)$ denotes the degree of vertex $v$. The  \emph{small world} property of acquaintance networks -- often referred to as six degrees of separation -- was popularised by the experiment of Milgram \cite{Milg67, TraMil69}. Generally, real-life networks are small worlds or even \emph{ultrasmall worlds.} That is, typical distances scale as a logarithm or even a double logarithm of the network size \cite{Bacetal12, Milg67, Mont2002, TraMil69}.  Real-life networks also tend to have high \emph{local clustering:} the proportion of triangles versus possible triangles tends to be rather large \cite{WatStr98}. 
 
 Many random graph models have been proposed so far to model real-life complex networks, each capturing some of the important features of these networks. Random graphs without geometry such as the configuration model \cite{Boll80, MolRee95}, variants of the Norros-Reitu model \cite{BJR07, ChuLu02a, NorRei06}, and the preferential attachment model \cite{BA99} mimic the scale-free and the small-world property well, but fail to produce an asymptotically positive \emph{clustering coefficient}. 
 To add the missing feature of positive clustering and to accommodate the rather natural spatial aspect of many real-life networks, spatial versions of the above models have been proposed. Scale-free random graph models with underlying geometry include spatial preferential attachment models (SPA)  \cite{AieBon08, FlaFriVer06, JacMor13}, hyperbolic random graphs \cite{BogPapKri10} and their generalisation, geometric inhomogeneous random graphs (GIRG) \cite{BriKeuLen15},  and the model that we study in this paper, \emph{scale free percolation} (SFP) \cite{DeiHof13}. These spatial models are equipped with an additional \emph{long-range} parameter $\al>0$ that describes how spread out the edges are in space. Let us mention briefly that there are other methods to incorporate clustering, e.g.\ by adding
local communities to an existing model, yielding hierarchical configuration models \cite{HofLeeSte15, SteHofLee16} or scale-free random intersection graphs \cite{BloKur16}. 

Once there is an underlying model that one can use to model real-life networks, the possibility to study the behaviour of information spread on networks opens up. Information spread is a general term that we use to cover a broad range of processes such as information diffusion, spreading of `activity', infection spread, etc., on different types of networks.
A way to model information spread is to allocate random edge-weights to all edges in the model from some underlying distribution. The length on an edge corresponds to the transmission time through that edge.  Distances in the newly obtained \emph{weighted random graph} then correspond to spreading time of information from one vertex to the other one. 

The behaviour of information spread is fairly well understood on random graph models without underlying geometry, such as the Erd{\H o}s-R\'enyi random graph, inhomogeneous random graphs \cite{BHH11, KolKom15}, or the configuration model with finite asymptotic variance of the degrees \cite{BarGes13, BHH10, BHH14, BhaHofKom14}. The infinite variance degree regime, that is, $\gamma \in (1,2)$,  seems to be more challenging. Recently with  Baroni  \cite{BarHofKomdist, BarHofKomtight} and Adriaans \cite{AdrKom17} we have determined the universality classes for weighted distances in the configuration model in this regime. 
Due to the novelty of spatial models, the theoretical study of information spread on them is rather limited: a related process, bootstrap percolation, is studied on hyperbolic random graphs and GIRGs \cite{CanFou16, KocLen16}. The behavior of random walk is studied on SFP in \cite{HeyHulJor17}.
Typical graph distances are studied for SFP in \cite{DeiHof13, DepWut15}, for GIRGs in \cite{BriKeuLen16} for some range of the parameters.

This paper is a step in studying information spread on spatial scale-free models. Our aim is to set up a program to analyse how the combination of the topology of the networks and the transmission time distribution affects the spreading time, and to identify \emph{universality classes}. We believe that the phenomena described in this paper are rather universal and similar results could be proven for other scale-free spatial graph models. 

\subsection{Our contribution}
In this paper we study information diffusion on the SFP model \cite{DeiHof13}. SFP is a random graph with vertex set $\Z^d$ that combines the Norros-Reitu model with \emph{long-range percolation}. 
After the construction of the graph, we allocate i.i.d.\ edge-weights to each existing edge.
We identify the (only) two universality classes of edge-weight distributions for information diffusion, that we call \emph{explosive} and \emph{conservative} class, respectively. 
We show that an edge-weight distribution $F_L(x)=\Pv(L\le x)$ with generalized inverse function $F_L^{(-1)}$ belongs the explosive class precisely when 
\[ \int_K^\infty F_L^{(-1)}\left(\exp\{-\e^x\}\mathrm dx\right)<\infty\] for some $K>0$, where $F_L^{(-1)}$ is the generalized inverse function of $F_L$. 
First we study the \emph{first passage profile} of the origin, i.e., the growth of the cluster of vertices available from the origin within distance $t$ as a function of $t$.
For the explosive class, when $\gamma\in(1,2)$, we show that there is a finite random time $V_0$, such that the number of vertices reachable from the origin within distance $V_0$ is infinite. We call this event \emph{explosion} and $V_0$ the \emph{explosion time} of the origin.
For the conservative class, the number of vertices reachable within any finite distance is a.s. finite, see Theorem \ref{thm:fpp-profile}. We show that explosion is never possible when $\gamma>2$, see Theorem \ref{thm:shape} and Corollary \ref{corr:non-expl}. For the conservative class, and when $\gamma \in (1,2)$, we further study typical (weighted) distances, extending the results on graph distances in \cite{DeiHof13}. For a unit length d-dimensional vector $\underline{e}$, let $\lfloor n \underline{e}\rfloor$ denote the vertex in $\Z^d$ obtained by coordinate-wise taking the integer part of the vector $n\underline e$.
We determine the leading order of the distance between $0$ and the vertex $\lfloor n \underline{e}\rfloor$  in terms of the edge-weight distribution, see Theorem \ref{thm:dist-cons}.
As a corollary (setting the edge-weights to be $1$ with probability $1$)  we obtain that for the graph distance
\[ \lim_{\|x\|\to \infty} \frac{\mathrm d_G(0,x)}{ \log \log \| x\|} \toinp \frac{2}{|\log (\gamma-1)|}, \] 
when $\gamma \in (1,2)$, where $\toinp$ denotes convergence in probability, and $\|x \|$ is the Euclidean norm of $x\in \Z^d$. Whether this convergence holds was an open question in \cite{DeiHof13} when a third parameter  $\tau$, that equals the power-law exponent of the vertex-weight distribution, satisfies $\tau>2$.

Our proofs require a more detailed analysis of the SFP without the edge-weights. An outcome of this analysis is that we are also able to identify the double exponential growth rate of the maximal displacement (that is, the distance of the farthest vertex graph distance $n$ away from the origin)  when $\gamma\in(1,2)$, see Theorem \ref{thm:max-displace}, a result that is interesting in its own right. When $\gamma>2$, we show that the growth rate of the maximal displacement is at most exponential, see Theorem  \ref{thm:shape}. 

Distances in the explosive class turned out to be surprisingly hard, due to the infinite number of vertices and non-compactness of $\Z^d$. This is in sharp contrast with explosion in non-spatial random graph models such as the configuration model, where typical distances converge to the sum of two explosion times \cite{BarHofKomdist, BHH10}. We have a lower bound on asymptotic distances in Theorem \ref{thm:distances} and we state a conjecture about a matching upper bound in Conjecture \ref{conj:distances}. 
Our theorems show that as long as $\gamma \in (1,2)$, the only \emph{relevant} parameter of the model in terms of weighted distances is the power law of the \emph{degree distribution}, $\gamma$. Neither the hidden parameter $\tau$ (coming from the Norros-Reitu model as the `fitness' of vertices) nor the long-range parameter $\al$ seem to play a role.  Similar phenomenon occurs in GIRGs, where the authors \cite{BriKeuLen16} observe that as long as the degrees have infinite variance, the only relevant parameter for typical distances is the power-law exponent of the degree distribution. For a more explicit connection between GIRGs and SFP, see Section \ref{s:discussion}.

\emph{New techniques.} Beyond studying information diffusion in a spatial random graph model with completely general edge-weight distributions, our paper improves existing techniques that are valuable in their own rights.
In Sections \ref{s:explore-BP} and \ref{s:bernoulli}, we describe and study a Branching Random Walk (BRW) in random 
environment that has both 
\emph{(1) infinite mean offspring}, and 
\emph{(2) infinite expected displacement}.
We show that, just like branching processes with infinite mean offspring, the size and maximal displacement of these BRWs grow double-exponentially. The results naturally carry through for the multi-type BRW with the same properties, i.e., when the environment is not fixed in advance but is re-shuffled every time the children of an individual are determined. 
Related work is \cite{BhaHazRoy17}, where the authors study a BRW with property (2). For critical BRWs, the one-arm exponent of BRWs with infinite expected displacement is studied in \cite{Huls15}. 

Another novel technique is the application of the idea of \emph{min-summability} that was invented in \cite{AmiDev13} for determining explosion in infinite-mean BRWs, in the random graph setting. The idea is that along a collection of paths with growing degrees $(d_i)_{i\ge m}$, one can estimate the shortest path by the sum of the typical minimal-edge-weights (that is, $\sum_{i=1}^m F_L^{(-1)}(1/d_i)$), and this estimate can be shown to be sharp when the degrees grow sufficiently fast. 

\emph{Notation.} We write r.v., lhs and rhs for random variable, left-hand side and right-hand side, respectively. For a sequence of random variables $(X_n)_{n \geq 1}$,  $X_n$ converges in probability to a r.v.\ $X$,  shortly $X_n  \toinp X$,  if for all $\ve > 0, \lim_{n \to \infty} \bP(|X_n-X| > \ve ) = 0$. Similarly,  $X_n$ converges in distribution/almost surely to $X$,  shortly $X_n  \toindis X, X_n \toas X$,  if  $\lim_{n \to \infty}\bP(X_n \le x) \to \bP(X \le x)  $ for all $x\in \R$ where $\bP(X \le x) $ is continuous/ $\lim_{n\to \infty} P(\lim_{n\to \infty} X_n=X)=1$.
 For a non-decreasing right-continuous function $F(x)$ the generalised inverse of $F$ is defined as $F^{(-1)}(x) := \inf\{y \in \bR : F(y) \geq x \}$.

\subsection{The model and results}\label{s:results}
In the scale-free percolation model, each vertex $x\in \Z^d$ is assigned an i.i.d.\ vertex-weight $W_x$ from a distribution $W$.
Let us fix the \emph{long-range parameter} $\al>0$ and the \emph{percolation parameter} $\la>0$. Conditioned on the collection of vertex-weights $(W_x)_{x\in \Z^d}$, there is an edge between any pair of non-nearest neighbor vertices $x,y \in \Z^d$ with probability
\be\label{eq:connect} \Pv_W((x,y) \text{ is present} \mid \|x-y\|> 1 ) = 1- \exp\Big\{  - \la \frac{W_x W_y}{\|x-y\|^\al}\Big\},\ee
where $\|x-y\|$ is the Euclidean distance between $x,y$, and $\Pv_W(\cdot):=\Pv(\cdot | (W_x)_{x\in \Z^d})$ is the conditional measure with respect to (wrt) the vertex-weights. Let us denote the resulting random graph by $\SFPW$. After this procedure is done, we assign each existing edge $e$ a random edge-weight, $L_{e}$, i.i.d.\ from some distribution $L$. Since our interests below are not affected by whether nearest neighbor edges are present in the graph, we assume that this is indeed the case, that is, 
\[\Pv((x,y) \text{ is present}\mid \|x-y\|=1) =1.\]
Let us denote the resulting weighted random graph by $\SFPWL$. Here, somewhat confusingly, ``weighted'' refers to the presence of edge-weights and not the vertex-weights $(W_x)_{x\in \Z^d}$. We shall consistently call the variables $L_e$ on edges \emph{edge-weights} and $W_x$ on vertices as \emph{vertex-weights}.

In this paper we are interested in the case when $W$ follows a power-law distribution with exponent $\tau>1$, that is, we assume that 
\be\label{eq:tail-W} \Pv(W>t) = \ell(t)/t^{\tau-1},\ee
for some function $\ell(t)$ that varies slowly at infinity, that is, $\lim_{t\to \infty}\ell(ct)/\ell(t)=1$ for all fixed $c>0$. We assume\footnote{This assumption is without loss of generality for any distribution $W$ with support separated away from $0$ by adjusting $\la$ accordingly.} throughout that that $\Pv(W\ge 1)=1$.
It is shown in \cite[Theorems 2.1, 2.2]{DeiHof13} that the degree distribution then follows a power-law distribution again, with tail exponent 
\be\label{def:gamma}\gamma:=(\tau-1)\al/d.\ee
 That is, for some function $\wit\ell$ varying slowly at infinity, as long as $\al>d$ and $\gamma:=\al(\tau-1)/d>1$,
\be\label{eq:degree-dist} \Pv(D_0>x) = \wit\ell(x)/x^\gamma,\ee
We  assume throughout the paper that $\al>d$ and $\gamma>1$ both hold. The case $\gamma>1$ includes the case where the tails in \eqref{eq:tail-W} are thinner than a power law.
We are interested in weighted distances, that is, for two vertices $x,y \in \Z^d$, let $\Gamma_{x,y}:=\{\pi: \pi \mbox{\ is a path from } x \mbox{ to }y \}$ be the set  of paths between $x,y \in \Z^d$. We define the \emph{$L$-distance} and the \emph{graph distance} as
\be\label{eq:L-distance-def} \mathrm d_L(x,y):=\inf_{\pi\in \Gamma_{x,y}} \sum_{e\in \pi} L_e, \qquad \mathrm d_G(x,y)=\inf\{|\pi|: \pi\in\Gamma_{x,y}\},\ee
where for a path $\pi$, $|\pi|$ denotes the number of edges on $\pi$. We shall write $|\pi|_L:=\sum_{e\in \pi} L_e$ for the $L$-length of a path. 
Let us define three important metric balls around a vertex $x\in \Z^d$, the ones for Euclidean distance, graph distance and weighted distance, respectively, as
\begin{align}\label{eq:eucledian} B^2_r(x)&:=\{y\in \Z^d: \|x-y\|\le r\}, \\
\label{eq:graph}B_r^G(x)&:=\{ y\in \Z^d: \mathrm d_G(x,y)\le r\}, \\
\label{eq:weighted} B^L_r(x)&:=\{y\in \Z^d: \mathrm d_L(x,y)\le r\}.
\end{align}
Note that while $B^2_r(x)$ is deterministic, $B_r^G(x), B_r^L(x)$ are random, depending on the realization of $\SFPWL$. 
Finally, for an integer $n$, we write $\Delta B_n^G(x):=B_n^G(x)\!\setminus\!B_{n-1}^G(x)$.
We shall consider the collection of sets $(B_t^L(x))_{t\in \R^+}$ as indexed by \emph{time} $t\in \R^+$. Note that the set $B_t^L(0)$ grows as $t$ increases. We call this collection the \emph{first passage profile} of $x$. To be able to analyse its characteristics, we introduce two hitting times. The first is  
\be\label{eq:tau-def} \tau_n(x):=\inf\{ t: |B^L_t(x)|=n\}, \ee
the first time when the $L$-metric ball contains $n$ vertices.
We write $\tau_n:=\tau_n(0).$
Next, 
\be \label{eq:mn-def}M_n(x):= \inf\{ t: B_t^{L}(x) \cap \Delta B^G_n(x) \neq \emptyset \}=\inf\{ d_L(x,y): \ y\in \Delta B^G_n(x)\}\ee
is the hitting time of $(B_{n-1}^G(x))^c:=\Z^d \setminus B_{n-1}^G(x)$.
Again, we write $M_n:=M_n(0)$.
In this paper, we write $F_X(\cdot)$ for the distribution function of a random variable $X$. Let
\be\label{def:inverse} F^{(-1)}_X(y):=\inf\{t\in \R: F_X(t)\ge y \}\ee 
the generalised inverse of $F_X$. With $F_L(x):=\Pv(L\le x)$, let us define
\be\label{eq:integral-crit} \mathbf{I}(L):=\int_0^\infty F_L^{(-1)}\Big(\exp\{-\e^{y}\}\Big) \mathrm dy. \ee
Our first theorem characterises the event of ``explosion'' for infinite-variance degrees: 
\begin{theorem}\label{thm:fpp-profile} 
Consider $\SFPWL$ with $W$ satisfying \eqref{eq:tail-W} with a power-law exponent $\tau>1$, $\al>d$ and $\gamma=\al(\tau-1)/d\in(1,2)$, and edge-weight distribution $L$. Let $x\in \Z^d$ be an arbitrary fixed vertex. 

(1) \emph{(Explosive part)}. When $\mathbf{I}(L)<\infty$, the hitting times $\tau_n(x)$ and $M_n(x)$ both converge almost surely, i.e.,
\be\label{eq:taun-limit} \tau_n(x) \toas V_x, \qquad M_n(x)\toas V_x, \ee
where $V_x<\infty$ a.s. We call the random variable $V_x$ the \emph{explosion time} of $x$.

 (2) \emph{(Conservative part)}. When $\mathbf{I}(L)=\infty$,  
\be\label{no-integral-nonexplosion}\tau_n(x)\toas \infty, \qquad M_n(x)\toas \infty.\ee
Further, in this case, 
\be\label{eq:non-explosive-mn}  \frac{M_n(x)}{ \sum_{k=1}^n F_L^{(-1)} \left( \exp\{- {1/(\gamma-1)}^{k} \}\right)} \toas 1. \ee

\end{theorem}

We comment that the (possibly infinite) almost sure limits of $\tau_n(x)$ and $M_n(x)$ are always equal for any graph with almost surely bounded degrees, thus the message of this theorem is that this limit is finite if and only if the integrability condition \eqref{eq:integral-crit} holds, as well as quantifying the rate of growth of $M_n(x)$ in the conservative case.
There is a deep connection to \emph{age-dependent branching processes} (BP). In an age-dependent BP, individuals have i.i.d.\ life lengths, and upon death they produce an i.i.d.\ number of offspring. An age-dependent BP is called explosive if it produces infinitely many individuals within finite time with positive probability, see \cite{Grey74, Sev67}. 
 The criterion in \eqref{eq:integral-crit} is the same criterion that is needed for an BP with offspring distribution given in \eqref{eq:degree-dist} but with $\gamma \in (0,1)$ and birth-time distribution $L$ to be explosive. This statement can be found in \cite[Lemma 5.8, Theorem 6.1]{Komexp} and it first appeared in some form in \cite{AmiDev13}. As a corollary we obtain:

\begin{corollary}
Let $\mathcal W(\mathrm{SFP}):=\{ L : \SFPWL \text{ explosive}\}$  be the set of distributions that produce explosive first passage profiles for the origin in $\SFPWL$ when $\gamma\in (1,2)$. Similarly, let $\mathcal W(\mathrm{BP})$ be the set of life-length distributions that produce explosive age-dependent BPs with power-law offspring distribution for some $\gamma\in(0,1)$. Then $\mathcal W(\mathrm{SFP}) = \mathcal W(\mathrm{BP})$.
\end{corollary}
For the conservative class, we characterise the first order of $L$-distances (see \eqref{eq:L-distance-def}) in $\SFPWL$:
\begin{theorem}[Distances in the conservative case]\label{thm:dist-cons}
Let us consider the same model as in Theorem \ref{thm:fpp-profile}, with $\mathbf{I}(L)=\infty$. Fix an arbitrary unit vector $\underline e$.  
Then, as $m\to \infty$,
\be\label{eq:dist-cons-1111}  \mathrm d_L(0,\lfloor m \underline e \rfloor)/ \left( \sum_{i=1}^{\lfloor \log \log m/|\log (\gamma-1)|\rfloor}2 F_L^{(-1)}\big( \exp\{-1/(\gamma-1)^k\} \big) \right)\toinp 1.\ee 
\end{theorem}
Setting $L\equiv 1$ yields typical \emph{graph distances}, an open question in \cite{DeiHof13} when $\gamma\in(1,2)$ but $\tau>2$:
\begin{corollary}\label{cor:missing-distance}
Consider the scale-free percolation model with $\al>d$, $\gamma\in (1,2)$. Then
\be \label{eq:typical-dist}\lim_{\|x\|\to \infty} \frac{\mathrm d_G(0,x)}{ \log \log \| x\|} \toinp \frac{2}{|\log (\gamma-1)|}. \ee
\end{corollary} 
Related to this result, we study the maximal displacement at graph distance $n$ from $0$. We believe this result is  of independent interest, so we state it as a separate theorem:
\begin{theorem}\label{thm:max-displace}
Let $D^{\max}_n(0):=\max\{ \|y\|: y\in \Delta B^G_n(0)\} $ be the maximal displacement in $n$ steps in $\SFPW$ with $\al>d, \gamma\in (1,2)$. Then for all $\ve$, for some random variables $Z(\ve)\le 2$ and $Y(\ve)$, and almost surely for all $n\ge 1$,
\be\label{eq:max-displacement} \exp\left\{ Z(\ve)\Big(\frac{ 1-\ve}{\gamma-1}\Big)^n\right\} \le
  D^{\max}_n(0) \le \exp\left\{Y(\ve) \Big(\frac{1+\ve}{\gamma-1}\Big)^n\right\}.  \ee  
In other words, almost surely, 
\be\label{eq:embed-12} B_n^{G}(0) \subset B^2_{\exp\{ Y(\ve) (1+\ve)^n/(\gamma-1)^n\}}(0). \ee
Further, $Y(\ve)$ has exponentially decaying tails. For some strictly positive constants $c=c(d,\al, \tau), C_\ve$, 
\be\ba \label{eq:Y-tail}\Pv(Y(\ve)\ge K) &\le C_{\ve} \exp\left\{ -  c\cdot  \ve \cdot K \right\}
	\ea \ee

\end{theorem}

 \begin{remark}\normalfont It would be natural to ask for \emph{almost sure} convergence instead of convergence in probability in the statement of Theorem \ref{thm:dist-cons}. We believe that almost sure convergence is only valid along subsequences that grow sufficiently fast, e.g.\ along the subsequence $([\e^n \underline e])_{n\ge 1}$. The reason for this is that $Y(\ve)$, the random prefactor in the upper bound in \eqref{eq:max-displacement}, decays (only) exponentially. Thus in the sequence $\lfloor n\underline e\rfloor$, it might happen that for infinitely many $n$ the corresponding $Y(\ve)$ in the maximal displacement is of order $\log \log n$. This, in return, makes distances to be shorter than $1-\delta$ times the numerator in \eqref{eq:dist-minus-explosion}, so that the almost sure convergence is lost. We give an argument why our proof cannot be strengthened to a.s. convergence in the proof of Theorem \ref{thm:dist-cons}.
 \end{remark}

 Studying distances in the explosive case turns out to be harder. It is possible to show based on the proof of Theorem  \ref{thm:fpp-profile} that two vertices, arbitrarily far away, can be connected via paths with total length that is a.s.\ \emph{finite}, but we failed to show so far that they can be connected within the sum of the two explosion times.
 The next theorem states that far away vertices have asymptotically independent explosion times, and states the lower bound on the $L$-distance.
\begin{theorem}[Distances in the explosive case]\label{thm:distances}
Consider the same model as in Theorem \ref{thm:fpp-profile}, with $\mathbf{I}(L)<\infty$ in \eqref{eq:integral-crit}. Then, as $\| x\|\to \infty$,
\be\label{eq:indep} (V_0, V_x) \toindis (V^{\sss{(1)}}, V^{\sss{(2)}}),\ee
where $V^{\sss{(1)}}, V^{\sss{(2)}}$ are two \emph{independent} copies of the limiting random variable from \eqref{eq:taun-limit}.
Further, for all $\ve>0$
\be\label{eq:dist-minus-explosion} \lim_{\|x\|\to \infty}  \Pv(\mathrm d_L(0,x)-(V_0+ V_x)<  -\ve) = 0.\ee
\end{theorem}
The theorem could have been stated in a similar formulation than the one in Theorem \ref{thm:dist-cons}:
\[ \lim_{m\to \infty}\Pv(\mathrm{d}_L(0,\lfloor m\underline e\rfloor ) -(V_0+V_{\lfloor m\underline e\rfloor}) < -\ve) =0.  \]

Next we study $\SFPWL$ with finite-variance degrees, i.e., when $\gamma>2$. 
Recall the notation for different metric balls from \eqref{eq:eucledian}-\eqref{eq:weighted}. The next theorem bounds $B_t^L(0)$ in the Euclidean space.
\begin{theorem}\label{thm:shape}
In $\SFPWL$ with $\al>d$, $\gamma>2$ and $F_L(0)=0$, there exist \emph{deterministic} constants $C, \wit C> 0$ (depending only on the distribution of $W$ and $L$ but not on the realization of the graph), and random constants $n_0(x)\in \N, t_0(x)\in \R$ such that, almost surely for all $n\ge n_0(x), t\ge t_0(x)$, 
\be\label{eq:shape-1} B_n^G(x) \subset B^2_{\exp\{Cn\}} (x), \qquad  B_t^L(x) \subset B^2_{\exp\{\widetilde Ct\}} (x).  \ee 
The value of $n_0(x), t_0(x)$ depend on the random realisation of the graph as well as on $x\in \Z^d$.
\end{theorem}
Note that since $|B^2_{\exp\{\widetilde Ct\}} (x)|=O(\exp\{d\widetilde C t \})$, the first statement in \eqref{eq:shape-1} implies that $|B_t^{L}(x)| \le \exp\{ d C' t \}$. Thus, $|B_t^L(x)|$ grows at most exponentially, and thus explosion is impossible:
\begin{corollary}\label{corr:non-expl}
In $\SFPWL$ with $\al>d$,$\gamma>2$ and $F_L(0)=0$, 
 \be\label{eq:exp-growth-2} \tau_n(x) \ge O(\log n),\ee
and thus $\tau_n(x)\to \infty$ almost surely. That is, explosion is never possible when $\gamma>0$.
\end{corollary}
Let us write  $\mathrm{Vol}_d$ for the volume of the unit ball in $\R^d$, and define the constant
\be\label{eq:cdla-def}  c_{d,\la}:=2( d + \la^{d/\al} (\mathrm{Vol}_d+ \pi_d)),\ee
with $\pi_d<\infty$ be defined as  $\pi_d:= \sup_{x\ge 0}\{ \int_{y:\|y\|\ge x} 1/\|y\|^\al \mathrm d^dy/ x^{d-\al}\}$.
\begin{remark}\normalfont
It is possible to give a more quantitative bound on the exponential growth of $B_t^L(0)$. By bounding the Malthusian parameter of the BRW in Section \ref{s:explore-BP}, one obtains that
 \be\label{eq:exp-growth-1} \limsup_{t\to \infty}\e^{-\beta t} |B_t^L(0)|=0 \ee
 holds for any $\beta>\beta^+$, where $\beta^+$ is the solution to $\Ev[\e^{-\beta^+ L}] \Ev[W^{2d/\al}] c_{d,\la}=1$. The proof requires cca 8 pages, and it does not provide a spatial embedding, so we decided to cut it from the paper.  
\end{remark}
\subsection{Organization of the paper}
We discuss some open questions in Section \ref{s:discussion}, and relate our model to GIRGs. Then, we start proving our results. In Section \ref{s:explore-BP} we define two BRWs that both provide a coupling upper bound on the exploration of the cluster of a vertex in $\SFPWL$. Theorem \ref{thm:thinned-BRW} --  the coupling of the exploration to these BRWs -- is a crucial ingredient for proving lower bounds for the results in Section \ref{s:results}. 
In Section \ref{s:gamma>2} we prove Theorem \ref{thm:shape}. Then, in Section \ref{s:prelim} we state some basic notions and introduce a boxing method that structures the high-weight vertices in the graph. In Section \ref{s:exp-proofs} we prove the explosive part of Theorem \ref{thm:fpp-profile}, and Theorem \ref{thm:distances}. These are the easier proofs of the paper.
Then, in Section \ref{s:bernoulli} we prepare to prove the conservative theorems, by analysing the size of and maximal displacement in generation $n$ of the upper bounding BRW from Section \ref{s:explore-BP} when $\gamma \in(1,2)$. This section is novel from the BRW point of view since it studies a general BRW  in a random environment that has \emph{infinite mean offspring}.  We bound the double-exponential growth rate of the generation sizes and the maximal displacement of this BRW. Section \ref{s:gamma12} then makes use of the results in Section \ref{s:bernoulli} and proves the  conservative part of Theorem \ref{thm:fpp-profile}, and Theorems \ref{thm:dist-cons}, \ref{thm:max-displace}. 
Finally, Section \ref{s:explore-proof} is devoted to describing an exploration process of the cluster of the origin (Section \ref{s:explore}), which results in the proof of  Theorem \ref{thm:thinned-BRW} (Section \ref{s:coupling-proof}).

\subsection{Discussion and open problems}\label{s:discussion} In this section we list some open questions. First of all, 
we conjecture that the statement about explosive $L$-distances, \eqref{eq:dist-minus-explosion} in Theorem \ref{thm:distances}, can be strengthened:
\begin{conjecture}\label{conj:distances}Under the condition of Theorem \ref{thm:distances}, for any two vertices $x,y\in \Z^d$
\be\label{eq:conj-upper-bound} d_L(x,y) - V_x + V_y \le 0\ee
almost surely. 
Thus, as $\|x\|\to \infty$, $\mathrm d_L(0, x) - (V_0+V_x)\toas 0.$
\end{conjecture}
The fact that the distance between two vertices tends to the sum of two explosion times is observed in the configuration model with infinite-variance degrees \cite[Theorem 1.2]{BarHofKomdist}, a non-spatial model with finitely many vertices.
A related model to SFP, with a finite number of vertices is the geometric inhomogeneous random graph model, (GIRG) \cite{BriKeuLen15}. This model lives on a compact space, - the $d$-dimensional unit torus $\mathbb T_d=\R^d/\Z^d$ or unit cube $[0,1]^d$ -  the location of $n$ vertices is sampled independently and uniformly at random according to the Lebesque measure and then an i.i.d.\ power-law vertex-weight $W_x$ is assigned to each vertex $x$. The long-range parameter of the model, $\wit \al$ is assumed to be larger than $1$. Two vertices, given their locations $x,y\in \mathbb T_d$ and vertex-weights $W_x, W_y$ are then connected with probability $p_n(x,y,W_x, W_y)$ that satisfies
\[ c_1 \le p_n(x,y, W_x, W_y) \min\left\{1, (W_x W_y/n)^{\wit\al}/ \|x-y \|^{\wit \al d} \right \} \le C_1 \]
for some $0<c_1<C_1<\infty$. Applying the transformation $T_nx:=n^{1/d} x$ maps the unit cube onto $[0,n^{1/d}]^d$, i.e., to a box with volume $n$. Considering the images of the points $T_n x$ as vertices, the connection probabilities in the model after this transformation become
\be\label{eq:p-GIRG} p(T_nx, T_ny, W_x, W_y)= \Theta\left( \min \left\{  1, \|T_nx-T_ny \|^{-\alpha d} (W_x W_y)^\alpha\right\}\right),\ee
that is, the factor $1/n$ cancels. Note that the edge connection probabilities of SFP in \eqref{eq:connect} also satisfy \eqref{eq:p-GIRG}. A Poissonised version of the number of points in GIRGs could provide a natural  extension of the model to $\mathbb R^d$. 
 If we assume that the vertex-weights follow a power-law distribution with exponent $\wit \tau>2$ in GIRG, then it is not hard to see that setting $\al=d\wit \al$ and $\tau-1=(\wit \tau-1)/\wit \al$ in the SFP model produces $\gamma=\wit \tau-1$. So, the GIRG can be looked at as a modification of SFP where (a) the $n$th graph is restricted to a box of volume $n$, (b) the locations of vertices are randomised (instead of the grid $\Z^d$ in SFP) (c) the edge probabilities are somewhat more general. 
 Here we state the corresponding conjecture about GIRGs.
\begin{conjecture}
Add i.i.d.\ edge-weights from distribution $L$ with $\mathbf{I}(L)<\infty$ to the edges of the GIRG. When the vertex-weights follow a power law distribution with exponent $\wit \tau\in (2,3)$,  the $L$-distance between two uniformly chosen vertices  from the giant component  converges to the sum of two i.i.d.\ random variables. In particular, the average $L$-distance does not grow with $n$.
\end{conjecture} 
To prove Conjecture \ref{conj:distances} on the SFP model on $\Z^d$, one could try  to show that the two \emph{shortest} explosive paths $\gamma^\star_0$ and $\gamma^{\star}_x$, having length $V_0, V_x$ respectively, can be connected by a short path (length less than $\ve$). This turns out to be a hard problem, since the space $\Z^d$ is non-compact and in principle it is possible that these paths wander off in different directions, avoiding very high degree vertices, so that a short connection cannot be established.
Further note that shortest exploding paths are special, in particular, both the degrees of their vertices and the edge-weights on the edges attached to them fail to be i.i.d. and we do not know how `densely' they cover the vertices of $\Z^d$. A shortest path that  goes to infinity slowly  can cause strong dependencies on the neighboring edge-weights, which makes probabilistic bounds 
based on independence impossible to apply. Based on the analogy to the conservative setting\footnote{Where Conjecture \ref{conj:growth} is not that hard to prove for the path realising $M_n(0)$.}, we believe that the shortest exploding path jumps off towards infinity double-exponentially:

\begin{conjecture}\label{conj:growth}
Let $\gamma_{0}^\star=(0, v_1^\star, v_2^\star, \dots):= \arg \min \{   |\gamma_L| \}$ be the shortest exploding path starting at $0$.
 Then there exists some $b>1$ such that $\| v_n^\star \| \ge \exp\{ b^n\}$.
\end{conjecture}
Double-exponential growth would be fast enough so that the shortest path cannot cover too many vertices in one region, and hence some independence on the edge-weights can be maintained. This conjecture is strong enough to imply Conjecture \ref{conj:distances}:
\begin{claim}\label{claim:conj1-2}
Conjecture \ref{conj:growth} implies Conjecture \ref{conj:distances}.
\end{claim}
This implication is far from obvious, we provide a sketch proof below the proof of Theorem \ref{thm:distances}.
Moving away from the $\gamma\in (1,2)$ setting, there is a lot to discover when $\gamma>2$. Typical graph distances in this regime are widely open: we only know a general logarithmic lower bound from \cite[Theorem 5.5]{DeiHof13},  and that distances are  linear when $\gamma>2, \al>2d$, see \cite[Theorem 8 (b2)]{DepWut15}. When $\al\in (d, 2d)$, the precise order is unknown, and so is the constant prefactor when $\al>2d$.  It would be interesting to see whether the BRW upper bound is a close approximation in the $\al\in (d,2d)$ regime and whether distances are truly logarithmic. 
These latter questions are similar in flavour to \emph{first passage percolation on long-range percolation}, when edges of the long-range percolation model have exponential lengths \cite{ChaDey16}. Due to the lack of power-law degrees, explosion is impossible in that model. Graph distances in supercritical long-range percolation are investigated in \cite{Bisk11, BisLin17}.

\section{Dominating branching random walks}\label{s:explore-BP}
In this short section we set the preliminaries to the lower bound of the proof of Theorem \ref{thm:shape} and Theorem \ref{thm:dist-cons},  \eqref{no-integral-nonexplosion} and the lower bound for \eqref{eq:non-explosive-mn} in Theorem \ref{thm:fpp-profile}.
The crucial ingredient of these proofs is the exploration of the neighborhood of a vertex $x$ in the order that corresponds to the $L$-distance, that is an interpretation of Dijkstra's algorithm, and we call it \emph{first passage exploration} (FPE). Importantly, Theorem \ref{thm:thinned-BRW} below describes a three-process coupling that couples the exploration on the $\SFPWL$ to two dominating  (BRW) in $\Z^d$. First we describe the two BRWs, one with Poissonian and one with sum of Bernoulli offspring distributions, denoted by $\PoiBRW$ and $\BerBRW$, respectively.

 In the skeleton\footnote{The skeleton of any BRW is just a random branching process tree, without being embedded in $\Z^d$.} of the BRWs, we name the individuals in the Harris-Ulam manner. That is, we order the children of the same individual in a way that corresponds their birth-order and call the root $\emptyset$, the children of the root $1, 2, \dots$, the second generation as $11,12, \dots, 21, 22, \dots$ In general, $i_1i_2\cdots i_k$ stands for the individual that is the $i_k$th child of the $i_{k-1}$th child $\dots$ of the $i_1$st child of the root. We refer to this coding as the \emph{name} of the individuals. 
We write $p(z)$ for the parent of an individual $z$. Naturally, $p(i_1\cdots i_k)=i_1\cdots i_{k-1}$. We set $p(\emptyset):=\emptyset$. In both BRWs, each individual $x$ has a vertex-weight $W_x$ (a type) that depends on her location\footnote{Following Jagers and Nerman \cite{JagNer84, JagNer96}, we use the female pronoun to refer to individuals.}. Note that the environment $(W_z)_{z\in \Z^d}$ stays the same. In this respect the BRWs are BRWs in random environment.

To initialise, we draw a collection of i.i.d. vertex-weights $(W_z)_{z\in \Z^d}$ from distribution $W$ that we call the environment. We set the location of the root at $0$, its name $\emptyset$ and give her vertex-weight $W_0$.  Conditioned on their vertex-weights and location, individuals reproduce \emph{independently}. Consider an individual with name $a$, located at $M_a\in \Z^d$ with vertex-weight $W_{M_a}$.
Conditionally on $M_a$ and $(W_z)_{z\in \Z^d}$,  her children in $\PoiBRW$ are described as follows. First, for all $y\in \Z^d$ draw a conditionally independent variable 
\be\label{eq:poi-vector} 
\quad 
N_y^{\mathrm{P}}(a) \sim \mathrm{Poi}\left( \la W_{M_a} \frac{W_y}{\|y-M_a\|^\al}\right), \ee
and allocate $N_y^{\mathrm{P}}(a)$ many children to the location $y\in \Z^d$. Each of these children have vertex-weight $W_y$ but note that each of them reproduces independently again.
 Since the sum of Poisson rvs is Poisson again, the total number of children via this method is distributed as
\be\label{eq:degree-dist-1} D_a^{\mathrm{P}} = \mathrm{Poi}\Bigg( \la W_{M_a}\!\cdot\!\!\! \sum_{y\in \Z^d\setminus \{M_a\}} \frac{W_y}{\|y-M_a\|^\al}\Bigg). \ee
We additionally add $2d$ children, each of them located at one of the $2d$ nearest-neighbors of $M_a$. 

We describe $\BerBRW$ coupled to $\PoiBRW$: here, we allow at most one edge to each location $y\in \Z^d$, i.e., for the individual $a$ located at $M_a$ in $\BerBRW$, we allocate $N_y^{\mathrm{B}}(a):=\ind\{N_y^{\mathrm{P}}(a)\ge 1\}$ many children to location $y\in \Z^d$. Again, we additionally add $2d$ children, each of them located at one of the $2d$ nearest-neighbors of $M_a$. When there are multiple edges to a nearest-neighbor vertex, we keep the added one. Note that 
\begin{equation}\label{eq:Ber-prob}
\Pv\left(N_y^{\mathrm{B}}(a)=1 \mid M_a, (W_z)_{z\in \Z^d}\right)= 1-\exp\left\{ \la W_{M_a} \frac{W_y}{\|y-M_a\|^\al}\right\},	
\end{equation} so the number of children of $a$ to $y\in \Z^d$ equals the probability that  the edge $(M_a,y)$ is present in SFP, see \eqref{eq:connect}. 
 Another way to realize the vector in \eqref{eq:poi-vector} is to first draw $D_a^{\mathrm{P}}$ in \eqref{eq:degree-dist-1} and then apply a multinomial thinning, that is, we choose $D_a^{\mathrm{P}}$ many  \emph{locations} (or marks) $M_{a1}, M_{a2}, \dots, M_{aD_a^{\mathrm{P}}} \in \Z^d\setminus \{M_a\}$ in an i.i.d.\ fashion, 
each of them having the conditional distribution
\be\label{eq:mark-choose} \Pv\left( M_{a1} =y \mid M_a, (W_z)_{z\in \Z^d}\right)= \frac{W_y/\|y-M_a\|^\al}{\sum_{z\in \Z^d\setminus\{M_a\}} W_z/\|z-M_a\|^\al }.\ee
Since a multinomial thinning of a Poisson variable yields \emph{independent} Poisson variables, we obtain \eqref{eq:poi-vector} back again.   

 Once the two BRWs are generated in a coupled manner,  we assign i.i.d.\ edge-weights from distribution $L$ to each existing edge\footnote{Edge here means parent-child relationship.} in the two BRWs in a coupled way, i.e., whenever $N^{\mathrm{B}}_y(a)=1$ for an individual $a\in\BerBRW$, we use an edge-weight chosen uniformly from the $N^{\mathrm{P}}_y(a)\ge1$ many edge-weights between $a$ and its children at $y$ in $\PoiBRW$. 
 
 Let us denote by $\CB_t^{L,\mathrm{P}}(0), \CB_t^{L,\mathrm{B}}(0)$ the graph spanned by the vertices reachable within $L$-distance $t$ in $\PoiBRW, \BerBRW$, respectively, including the edge-weights $L$ on the edges and the vertex-weights. 
Let us write $\CB_t^{L,S}(0)\subseteq \CB_t^{L}(0)$ for the tree \emph{formed by the edges that are on shortest-length paths to $0$} in $\SFPWL$, including edge-weights, vertices and their vertex-weights. The vertex set of $\CB_t^{L,S}(0),\CB_t^{L}(0)$ are the same, only edges that are not on shortest paths toward zero are not present in $\CB_t^{L,S}(0)$.

\begin{theorem}\label{thm:thinned-BRW}
Consider the shortest-path tree $\CB^{L, \mathrm{S}}_t(0)$
 in $\SFPWL$. There is a three-process coupling of the above described $\PoiBRW, \BerBRW$  to $\SFPWL$ such that for any $t\ge 0$, $\CB^{L, \mathrm{S}}_t(0)$ in $\SFPWL$ can be obtained as a subgraph of   $\CB_t^{L,\mathrm{B}}(0)$. 
More precisely, under the coupling,
\[ \CB_t^{L, \mathrm{S}}(0) \subseteq \CB_t^{L,\mathrm{B}}(0) \subseteq   \CB_t^{L,\mathrm{P}}(0).\]
\end{theorem}

We provide the three-process coupling and the proof of Theorem \ref{thm:thinned-BRW} in Section \ref{s:explore-proof}. 
Let us make a remark. In the Norros-Reitu (a similar, non-spatial) model the exploration of the cluster of a vertex can be coupled to a \emph{multitype} branching process  instead of a BP in a \emph{random environment}. In other words, the whole collection of vertex-weights can be newly drawn to determine the children of a newly explored individual. In the scale-free percolation model, this is not possible for the following reason. The information that a vertex close to the origin is not explored for many steps of the exploration reveals some information on the vertex-weight of the given vertex (i.e., it is most likely very small). This phenomenon is not present in the non-spatial model.
\section{Non-explosion for $\gamma>2$}\label{s:gamma>2}
In this section we prove that the scale-free percolation model can never explode when $\gamma>2$ and $F_L(0)=0$. That is, we prove Theorem \ref{thm:shape}.
Recall that we write $m_x=\Ev[W^x]$ for $x\in \R$. The proof of Theorem \ref{thm:shape} is based on the following lemma: 
 \begin{lemma}\label{lem:age-dep}
 Consider $\BerBRW$ as descibed in Section \ref{s:explore-BP}. 
Then, the expected generation sizes $Z_n^{\mathrm{B}}$ of  $\mathrm{BerBRW}$ grow exponentially. Namely, for all $n$,  
\be \label{lem:exp-growth-2} m_+^{-n}\Ev[ Z_n^{\mathrm{B}} \mid W_{\emptyset}=w] \le   w^{d/\al} m_{d/\al}/m_{2d/\al}.  \ee
for $m_+ =  c_{d,\la} \Ev[W^{2d/\al}]$, with $c_{d,\la}$ as in \eqref{eq:cdla-def}. \end{lemma} 
 \begin{proof}
 Recall the definition of the edge probabilities in $\BerBRW$ from \eqref{eq:Ber-prob}, and that nearest neighbor edges are, additionally, always present. So, let us write $\mathcal{N}:=\{ y: \| y\|=1 \}$ for the nearest neighbors of the origin.
For a set $\CA\subset \R^+$, $N_{w}^{\mathrm{B}}(\CA)$, called the \emph{reproduction kernel}, denotes the number of children with vertex-weight in the set $\CA$ of an individual with vertex-weight $w$ located at the origin. In case of the $\BerBRW$, the distributional identity holds: 
\begin{equation}\label{eq:kernel-1}
	N_{w}^{\mathrm{B}}(\mathrm d v) \ {\buildrel d \over =}   \sum_{y\in \Z^d} \ind_{\{W_y\in (v, v+\mathrm dv)\}}  \left( \ind_{\{y \in \mathcal{N}\}} + \ind_{\{y\not\in \mathcal{N}\}} \mathrm{Ber}(1-\e^{-\la w v/\|y\|^\al}) \right).
\end{equation}
The \emph{expected reproduction kernel} (cf.\ \cite[Section 5]{Jagers89}) is defined as $\mu(w, \mathrm dv):=  \Ev[N_w^{\mathrm{B}}(\mathrm dv)]$. Using that $1-\exp\{-\la w v/\|y\|^\al \}\le \min\{1, \la w v/\|y\|^\al\}$, it can be bounded from above as
 \be\label{eq:mu-ber} \ba \mu(w,\mathrm dv )
 &\le F_W(\mathrm dv) \Bigg(\sum_{
\substack{y: \|y\|\le (\la wv)^{1/\al} \\ \text{or } y \in \CN}}  1 + \la vw \sum_{y: \|y\|\ge (\la wv)^{1/\al}} \|y\|^{-\al}\Bigg)\\
 &=: c_{d,\la}(w,v) w^{d/\al} v^{d/\al} F_W(\mathrm dv),\ea\ee
 where $c_{d, \la}(w,v)\le c_{d,\la}=2(d+\la \mathrm{Vol}_d+ \la^{1/\al} \pi_d)$  in \eqref{eq:cdla-def}. To obtain the second line, we bound the sum on the rhs in \eqref{eq:mu-ber} by approximating the two sums by two integrals.  Then, $\mathrm{Vol}_d$ denotes the volume of the unit ball in $\Z^d$,  while $\pi_d$, used when bounding the second integral, is defined as $\pi_d:= \sup_{x\ge 0}\{ \int_{y:\|y\|\ge x} 1/\|y\|^\al \mathrm d^dy/ x^{d-\al}\}$. The factor $2$ is a crude upper bound on the approximation between the integrals and the  sums.
 Finally, the term $2d$ comes from the contribution of nearest neighbors, where $w^{d/\al}v^{d/\al}>1$ holds by the assumption that $\Pv(W\ge 1)=1$, see after \eqref{eq:tail-W}.  
Defining 
\be\label{eq:mupm}\mu_{+}(w, \mathrm dv):= c_{d,\la} w^{d/\al} v^{d/\al} F_W(\mathrm dv),\ee we obtain an upper bound 
$  \mu(w,\mathrm dv) \le \mu_+(w, \mathrm dv)$ \emph{uniform in} $v,w$.
Observe that $\BerBRW$ without the edge-weights is a branching process (BP), where each individuals have a `type' being its vertex-weight. It is not a multi-type BP in the usual sense, since the environment causes dependencies between the types of children of different individuals, while in a multi-type BP those types should be independent. Nevertheless, we can bound the \emph{expected} number of individuals with given type in generation $n$ by simply applying  the composition operator $\ast$ acting on the type space, that is, define 
\[ \mu^{\ast n}(w, \CA)=\int_{\R^+} \mu(v, \CA) \mu^{\ast (n-1)}(w, \mathrm dv), \qquad  \mu^{\ast1}:=\mu. \]
For instance, $\mu^{\ast2}$ counts the expected number of individuals with type in $\CA$ in the second generation, averaged over the environment.  
The rank-1 nature of the kernel $\mu_+$ implies that its composition powers factorize. An elementary calculation using \eqref{eq:mupm} shows that 
\be\label{eq:mu-convolv}\mu_+^{\ast n}(w, \mathrm dv) =w^{d/\al} \cdot  v^{\mathrm d/\al }F_W(\mathrm dv) \cdot (c_{d,\la})^n \Ev[W^{2d/\al}]^{n-1}. \ee
Thus,
\[ \ba \Ev[Z_n^{\mathrm{B}} \mid W_\emptyset=w]&=\mu^{\ast n}(w, \R^+) \le \mu^{\ast n}_+(w, \R^+) \\
&= w^{d/\al} \int_{v\in \R^+}v^{\mathrm d/\al }F_W(\mathrm dv) \cdot (c_{d,\la})^n \Ev[W^{2d/\al}]^{n-1}\\
&=w^{d/\al} (m_+)^{n} m_{d/\al} / m_{2d/\al}.  \ea\]
\end{proof}
We are ready to prove Theorem \ref{thm:shape}:
\begin{proof}[Proof of Theorem \ref{thm:shape}]
We know from \cite[Theorem 5.5]{DeiHof13} that when $\gamma>2$, for a fixed $x\in \Z^d$ with $\|x\|$ large enough, $\lim_{\|x\|\to \infty}\Pv(d_G(0,x)\ge \eta \log \|x\|)=1$ for some $\eta>0$. We improve this theorem and show that this holds \emph{for all} $x\in \Z^d$ with the same norm simultaneously. In other words, 
\be\label{eq:graph-contain} B^G_n(0) \subset B^2_{\exp\{ C n\}}(0), \ee for some $C>0$ for all sufficiently large $n$.  
Indeed, quoting the second formula after \cite[(5.16)]{DeiHof13}, for some constants $C', \kappa>0$, and $x\wedge y:=\min(x,y)$,
\[ \Pv(d_G(0, x)=i) \le i \Big( C' \la^{1\wedge(\tau-1)/2}\Big)^i (\log \| x\|)^\kappa \| x\|^{-\al (1\wedge(\tau-1)/2)}. \]
First we sum this formula in $i$, for $i=1,\dots, n$. Since $\sum_{i=1}^n i q^i \le  n q^{n+2}/(q-1)^2$ for $q>1$, 
 \[ \Pv(d_G(0, x)\le n) \le C_\la  n \Big( C' \la^{(\tau-1)/2\wedge 1}\Big)^n (\log \| x\|)^\kappa \| x\|^{-\al (1\wedge (\tau-1)/2)}. \]
Summing now for all $x$ with norm  $\| x\| \ge \e^{Cn}$ for some $C>0$, we obtain that
\[ \ba \Pv(\exists x\in \Z^d: \| x\|\ge \e^{Cn} &\text{ and }  d_G(0, x)\le n) \\
&\le C_\la n \big( C' \la^{1\wedge(\tau-1)/2}\big)^n \sum_{x: \| x\|\ge \e^{Cn} } (\log \| x\|)^\kappa \| x\|^{-\al (1\wedge(\tau-1)/2 )}\\
&\le  \big( C' \la^{1\wedge(\tau-1)/2}\big)^n C_\la n  C_d (\log \e^{Cn})^\kappa (\e^{Cn})^{d-\al (1\wedge(\tau-1)/2)}\\
 \ea \]
for some $d$-dependent constant $C_d$, where we have used \eqref{eq:karamata} to obtain the last line. Note that $d -\al\big(\frac{\tau-1}{2}\wedge 1\big)=d(1-\frac{\gamma}{2}\wedge \frac{\al}{d})<0$ for $\gamma>2$ and $\al>d$.
Thus we can choose $C$ large enough so that the product of the first and last factor on the rhs is $<1/2^n$, say, and then we arrive at
\[   \Pv(A_n):=\Pv(\exists x\in \Z^d: \| x\| \ge  \e^{Cn} \text{ and } d_G(0, x)\le n) \le C_d' n^{\kappa+1} \frac1{2^n},  \]
where we have combined all constants into $C_d'>0$.
Since the rhs is summable in $n$, by the Borel-Cantelli lemma, almost surely only a finitely many $A_n$s occur. Let $n_1$ denote the first index so that  $A_n^c$ holds for all $n>n_1$. That is, 
\[ A_n^c=\{  \forall x\in \Z^d \text{ with } \|x\|\ge \e^{Cn}: d_G(0, x) >n \} = \{ B_n^{G}(0)\subset B^2_{\exp\{Cn\}}\}.\]
This shows \eqref{eq:graph-contain}. In what follows we show that  the following implication is true for some $\wit C, t_1\ge 0$:
\be\label{eq:implication} \forall n\ge n_1: B^G_n(0) \subset B^2_{\exp\{ C n\}}(0) \quad \Longrightarrow \quad \forall t\ge t_1: B^L_t(0) \subset B^2_{\exp\{ \wit C t\}}(0). \ee
Recall that $\Delta B_n^{G,L}(0)$ denotes those vertices in $\SFPWL$ from which the shortest path to the origin uses $n$ edges, and that Proposition \ref{lemma:thinned-BRW} states that the shortest path tree $\CB^{L, \mathrm{S}}_t(0)$ is also present in BerBRW. 
This implies that $|\Delta B_n^{G,L}(0)| \le 
Z_n^{\mathrm{B}}$, the size of generation $n$ in $\BerBRW$.
Let $m_+$ be as in Lemma \ref{lem:age-dep}, and let $t_0$ be so small that $F_L(t_0)< 1/(16 m_+^2) $. Set $n:=\lfloor 2t/t_0 \rfloor+1$ and $\wih C=2C/t_0+C$. By \eqref{eq:graph-contain}, for all $n\ge n_1$,   $B^G_n(0) \subseteq B^2_{\exp\{C n\}}(0) \subseteq B^2_{\exp\{ \wih C t\}}(0)$ since $Cn \le   \wih C t$. So,  $\{B_t(0)\not \subset B^2_{\exp\{ \wih C t\}}(0)\}\subseteq\{B_t(0) \not \subset B^G_{n}(0)\}$, where the latter event is
\[ \exists y \in \Z^d: d_G(0, y)> n \ \& \ d_L(0, y)\le t.  \]
Note that $d_G(0,y)> n$ means that  the $L$-shortest path from $y$ to $0$ uses more than $n$ edges. 
Following the shortest path from any such $y$ to $0$ we can also find a $y'\in \Delta B_n^{G,L}(0)$ with  $d_L(0, y')\le t$. 
By Proposition \ref{lemma:thinned-BRW}, $\Delta B_n^{G,L}(0) \subseteq \CG_n^{\mathrm{B}}$, and thus a union bound results in
\be\label{eq:estimate-1}
\Pv(B_t^L(0)\not \subset B^2_{\exp\{ \wih C t\}}(0)) \le \Pv(\exists y' \in \Delta B_n^{G,L}(0): d_L(0, y')\le t )
\le \Ev[Z_n^{\mathrm B}] F_L^{\ast n}(t),
\ee
where $F_L^{\ast n}(t)$ is the distribution function of the $n$-fold convolution of $L$ with itself, 
since the edge-weights are i.i.d.\ copies of $L$.
Since the vertex-weight of the root is i.i.d.\ $W$ from \eqref{eq:tail-W}, by Lemma \ref{lem:age-dep},
\be\label{eq:znb-333} \Ev[Z_n^{\mathrm{B}}]=\Ev[ \Ev[ Z_n^{\mathrm{B}}\mid W_\emptyset=W] ]\le m_{+}^n\Ev[W^{d/\al}] m_{d/\al}/ m_{2d/\al}=m_+^n (m_{d/\al})^2/m_{2d/\al}.\ee
Now we bound  
 $F_L^{\ast,n}(t)$. If more than $n/2$ variables in the sum $L_1+\dots +L_n$ would have length at least $t_0$, then the sum exceeds $n t_0 /2\ge t$. Hence at least $n/2$ variables have value at most $t_0$.
 Thus, 
\be\label{eq:FL-convolution}  F_L^{\ast,n}(t)\le F_L^{\ast, n}(n t_0 /2) \le {n \choose n/2}(F_L(t_0))^{n/2} \le 2^n \left( \frac{1}{16m_+^2}\right)^{n/2}=\frac1{2^n m_+^n}.  \ee
Combining \eqref{eq:znb-333} with \eqref{eq:FL-convolution}, the rhs of \eqref{eq:estimate-1} is at most 
\be\label{eq:pak}
\Pv(B_t^L(0)\not \subset B^2_{\exp\{ \wih C t\}}(0)) \le m_{d/\al} m_+^n /(2^n m_+^n)=m_{d/\al} 2^{-n}.
\ee
Define the event $\wit A_k:=\{ B_k^L(0) \not \subseteq B^2_{\exp\{ \wih C k\}}(0)\}$. By \eqref{eq:pak}, $\Pv(\wit A_k)$ is summable in $k$, thus, by  the Borel-Cantelli lemma, there exists a random $k_1\ge n_1$, such that $A_k^c$ holds a.s.\ for all $k>k_1$. Consider now a $t\in (k-1, k)$ for some $k\ge k_1+1$.
Then, due to the monotonicity of $B_t^L(0)$ in $t$, on the event $A_k^c=\{  B_k^L(0) \subseteq B^2_{\exp\{ \wih C k\}}(0)\}$, 
\[ B_t^{L}(0) \subseteq B_{k}^L(0)\subseteq B^2_{\exp\{\wih C k\}}. \]
Clearly $\e^{\wih C k} \le \e^{\wih C} \e^{\wih C t }\le \e^{2\wih C t}$ for all $t\in (k-1, k)$ and $t\ge 1$. Thus, \eqref{eq:implication} follows with $\wit C=2\wih C=4C/t_0$ and $t_1=k_1$.
This finishes the proof of \eqref{eq:shape-1}. 
\end{proof}

\section{Boxing and greedy paths}\label{s:prelim}
In this section we develop the main ingredient for proving upper bounds -- a boxing method combined with greedy minimum-length paths.  We begin with some definitions.
\begin{definition}[Power-law tail behavior] We say that the random variable has regularly-varying tail with power-law exponent $\al\in (0,1)$ if there exists a $K>0$ and a function $\ell(\cdot)$ that varies slowly
 at infinity such that for all $x\ge K$
\be\label{eq:power-law-alpha} \Pv\left(X \ge x \right)=\ell(x)/ x^\al.\ee
\end{definition}
We shall often use Potter's theorem about slowly-varying functions. For all $\ve>0$,
\be \label{eq:slowly-1}\lim_{x\to \infty} x^\ve \ell(x) = \infty, \qquad
\lim_{x\to\infty} x^{-\ve} \ell(x) = 0.
	\ee
We shall use the following Karamata-type theorem \cite[Propositions 1.5.8, 1.5.10]{BinGol}: For any $\beta> 0$, and $a>0$,
\begin{equation}\label{eq:karamata}
	\lim_{u\to\infty}\frac{1}{ u^{-\beta} \ell(u)} \int_{u}^\infty \ell(x) x^{-\beta-1} \mathrm dx = \lim_{u\to\infty}\frac{1}{ u^{\beta} \ell(u)}\int_{a}^u \ell(x) x^{\beta-1} \mathrm dx =\frac{1}{\beta}. 
\end{equation}

We continue with the notion of \emph{min-summability}, a powerful tool to analyse the explosion of age-dependent BPs, a criterion developed in \cite{AmiDev13}. We cite \cite[Corollary 4.3]{AmiDev13} after this definition:
\begin{definition}Let $(a_k)_{k\in \N}$ be a sequence of integers with $\lim_{k\to\infty}a_k= \infty$ and $(L_{i,j})_{i,j\in \N}$ be i.i.d.\ copies of $L$.The distribution $F_L$ is called $a_k$-summable if almost surely
\[\sum_{k=1}^\infty \min\{L_{k,1}, L_{k,2}, \dots, L_{k, a_k}  \} <\infty.\]
\end{definition}
The following lemma is a rewrite of \cite[Corollary 4.3]{AmiDev13}.
\begin{lemma}[Min-summability criterion, \cite{AmiDev13}]\label{thm:amidev2}
Given a sequence $(a_k)_{k\in \N}$ with $a_k \ge c a_{k-1}$ for all sufficiently large $k$ and some $c>1$, the distribution $F_L$ is $a_k$-summable  if and only if
\be\label{eq:explosive-crit-1} \sum_{k=1}^\infty F_L^{(-1)}(1/a_k)<\infty.\ee
\end{lemma}
The proof is based on Kolmogorov three series theorem, we refer the reader to \cite[Corollary 4.3]{AmiDev13} for details. A consequence of this powerful lemma is that if two sequences $(a_k)_{k\ge 1}, (b_k)_{k\ge 1}$ satisfy   $a_k\le b_k$ for all sufficiently large $k\in \N$, and $F_L$ is $a_k$-summable then it is also $b_k$-summable. 
\begin{definition}[Double-exponentially growing sequence]\label{def:double-exp}
We say that a sequence $(a_k)_{k\in \N}$ grows double-exponentially if there exist constants 
$b_1, b_2>0$ and $c_2\ge c_1 >1$ such that 
\be\label{eq:double-exp}\exp\{b_1c_1^k\} \le a_k\le \exp\{b_2 c_2^k\}. \ee
\end{definition}
\begin{claim}\label{lemma:integral} For a double-exponentially growing sequence $(a_k)_{k\in \N}$, and an edge-weight distribution $F_L$, the criterion in \eqref{eq:explosive-crit-1} is equivalent to the integral criterion \eqref{eq:integral-crit}.
\end{claim}

\begin{proof}  By definition, $a_n \ge \exp\{ b_1 c_1^n\}$ for some $b_1>0, c_1>1$.
By monotonicity of $F_L^{(-1)}$,
 \be\label{eq:sum-1234} \sum_{k=1}^\infty F_L^{(-1)}(1/a_k) \le \sum_{k=1}^\infty F_L^{(-1)}\big(\exp\{ - b_1 c_1^k \}\big).\ee
Choose $K$ large enough so that $K\log c_1+ \log b_1>0$. The rhs of \eqref{eq:sum-1234} is equi-convergent with  
\[ \int_K^\infty F_L^{(-1)}\big(\exp\{ - b_1 c_1^x \}\big) \mathrm dx = \frac{1}{\log c_1}\int_{K\log c_1+\log b_1}^\infty  F_L^{(-1)}\big(\exp\{ - \e^y \}\big) \mathrm dy. \]
Thus the convergence of \eqref{eq:integral-crit} implies \eqref{eq:explosive-crit-1}.
 The other direction is established in a similar manner using that $a_k \le \exp\{ b_2 c_2^k\}$ for $b_2>0, c_2>1$.
\end{proof}
\subsection{Boxing around the origin and greedy paths}\label{s:boxing-greedy}
The upper bound in many of the proofs uses a boxing procedure that we describe now.
We surround vertex $0$ by an infinite sequence of box-shaped annuli $\Gamma_n$, and we divide each annulus into equal-size subboxes, $B_n^{\sss{(i)}}$, $i=1, \dots, b_n$. More precisely, let us fix $C, D>1$ to be chosen later, and define
\be \label{eq:rn}R_n:=\exp\{ C^n\},\qquad
  D_n:= \exp\{ D C^n\}\ee
and  set
\be\label{eq:layering-def}\ba \mathrm{Box}_n:=&\{ y\in \Z^d: \|y\|_\infty<D_n/2\}, \\
\Gamma_n:=&\mathrm{Box}_n\setminus \mathrm{Box}_{n-1}, \\
B_n(z):=&\{ y\in \Z^d: \|y-z\|_\infty<R_n/2\}.
\ea\ee
The construction is as follows: for each $n\in \N$, $\Gamma_n$ is an annulus with outer and inner radius $D_{n}/2, D_{n-1}/2$, respectively (we use $\ell_\infty$ norm to get a box-shape). We divide each $\Gamma_n$ into disjoint subboxes of radius $R_n/2$ as in \eqref{eq:layering-def}.
Then the number of boxes in $\Gamma_n$, denoted by $b_n$, is given by
\be\label{eq:bn} b_n=\frac{\mathrm{Vol}(\Gamma_n)}{\mathrm{Vol}(B_n(0))}(1+o(1)) = \frac{\exp\{ dD C^n\}-\exp\{d D C^{n-1}\} }{ \exp\{d C^n \}}(1+o(1)). \ee
Evidently, for all large enough $n$,
\be\label{eq:an}  \exp\{ d(D-1) C^n\}/2 \le b_n \le \exp\{d (D-1) C^n \}, \ee 
thus $b_n$ grows double-exponentially when $C,D>1$. We order the subboxes in $\Gamma_n$  in an arbitrary way and denote subbox $i$ by $B_n^{\sss{(i)}}$. Let us call the maximal vertex-weight vertex in subbox $B_n^{\sss{(i)}}$ the \emph{center} of $B_n^{\sss{(i)}}$, and denote it by $c_n^{\sss{(i)}}$.
Let us fix a small $\ve>0$, and for $\delta:=\delta(\ve)=\ve (\gamma-1)/(2\gamma)$ define
\be \label{eq:C-epsz} C(\ve):=(1-\ve)/(\gamma-1), \quad\quad
  D(\ve):=\frac{(1- \delta)(1-\ve/\gamma)}{1-\ve} -  \de/2 > 1. \ee
The next lemma is a \emph{quenched}, i.e., it holds for almost all realisation of the vertex-weights and the edges in $\SFPW$:
\begin{lemma}\label{lem:layers-complete} 
For any $\de>0$, there is a random $n_w(\de)$ such that for all $n>n_w(\delta)$, all centers of boxes $(c_n^{\sss{(i)}})_{i\le b_n}$ satisfy that
\be\label{eq:weight-of-center} W_{c_n^{(i)}} \ge \exp\left\{C^n (1-\de)\al/\gamma\right\}.\ee
More importantly, for $C:=C(\ve), D:=D(\ve)$ as in \eqref{eq:C-epsz}, and $\delta:=\ve(\gamma-1)/(2\gamma)$ as in \eqref{eq:C-epsz}, there is a random $n_0:=n_0(\ve)\ge n_w(\delta)$ such that for all $n\ge n_0$, 
the centers $(c_n^{\sss{(i)}})_{i\le b_n}$ of the subboxes within $\Gamma_n$ form a complete graph on $b_n$ vertices, and the centers $(c_{n-1}^{\sss{(j)}})_{j\le b_{n-1}}$ in $\Gamma_{n-1}$ and $(c_{n}^{\sss{(i)}})_{i\le b_{n}} $ in $\Gamma_{n}$ form a complete bipartite graph on $b_{n-1}$ and $b_n$ vertices in each bipartition, respectively.

Further, for some constants $c_1, c_2$ that do not depend on $\ve$, and for all large enough $K$,
\be\label{eq:n0} \Pv(n_0(\ve) \ge K) \le \exp\{ - c_1 \exp\{ \ve c_2 C(\ve)^K\} \}. \ee

\end{lemma}
\begin{proof}
We start by showing \eqref{eq:weight-of-center}. Note that $ W_{c_n^{(i)}}=\max_{x\in B_n^{(i)}}W_x $ per definition.
Using that $\al/\gamma=d/(\tau-1)$,  and that $1-x\le \e^{-x}$ as well as \eqref{eq:tail-W}, 
\be\label{eq:weight-bound-111} \ba \Pv\Big(\max_{x\in B_n^{(i)}}W_x \le   \exp\Big\{\frac{d C^n(1-\de)}{\tau-1}\Big\}\Big)&= F_W\Big(  \exp\Big\{\frac{d C^n(1-\de)}{\tau-1}\Big\}\Big)^{\exp\{ d C^n\}}\\
& \le \exp\left\{ -\e^{- d C^n (1-\de)}\ell(\e^{- d C^n (1-\de)/(\tau-1)} ) \e^{dC^n} \right\}\\
&=\exp\{ - \e^{dC^n \de/2}\},
\ea \ee
where we have used \eqref{eq:slowly-1} to establish that there is an $n_1(\de)$ such that for all $n\ge n_1(\de)$, $\e^{- d C^n \de/2}\ell(\e^{- d C^n (1-\de)/(\tau-1)} )>1$.
Thus, using \eqref{eq:an} and a union bound, the probability that $\Gamma_n$ contains at least one center that has smaller vertex-weight than the rhs in \eqref{eq:weight-of-center} is at most
\be\label{eq:n0-1}  b_n \exp\{ - \e^{dC^n \de/2}\} \le \exp\{dD C^n -  \e^{dC^n \de/2} \} \ee
which is summable in $n$. Thus by the Borel-Cantelli lemma, there is a random  $n_w(\de)>n_1(\de)$ such  that the statement in \eqref{eq:weight-of-center}  holds for all $n\ge n_w(\de)$ and $i\le b_n$.

 Next we show that the centers in $\Gamma_{n-1}$ are connected to the centers in $\Gamma_{n}$. Note that the $\ell_2$ distance between any two subboxes in $\mathrm{Box}_n$ is at most $\sqrt{d}D_n$.
Here we set $\delta:=\ve (\gamma-1)/ (2\gamma)$, and assume that $n-1>n_w(\delta)$. Using the connection probability in \eqref{eq:connect}, the estimate in \eqref{eq:weight-of-center} and writing $d/(\tau-1)=\al/\gamma$, we bound the probability that we can find two centers $c_n^{\sss{(i)}}$ and $c_{n-1}^{\sss{(j)}}$ for some $i\le b_n, j\le b_{n-1}$ that are not connected by an edge by:
 \be\label{eq:box-conn} \Pv( \exists i\le b_n, j\le b_{n-1}: \ c_n^{\sss{(i)}}\not \leftrightarrow c_{n-1}^{\sss{(j)}}) \le  b_n b_{n-1}\exp\Big\{  -\la \frac{\exp\{C^n\frac{(1-\de)\al}{\gamma}\} \exp\{C^{n-1}\frac{(1-\de)\al}{\gamma}\}}{(\sqrt{d} D_n)^\al}  \Big\}. \ee
  Using \eqref{eq:rn}, the exponent on the rhs becomes $-\la d^{-\al/2}$ times
  \be\label{eq:exponent-1} \exp\left\{   C^{n-1} \al \left( (1-\de) (1+C)/\gamma - DC \right) \right\}. \ee
For a decay with $n$ in \eqref{eq:box-conn}, the coefficient of $C^{n-1}$ in the exponent in \eqref{eq:exponent-1} must be positive,  i.e.,
\be\label{eq:exp-pos}  (1-\de)(1+C)/\gamma - DC >0. \ee
This is satisfied with the choices $C(\ve), D(\ve)$ in \eqref{eq:C-epsz}, \eqref{eq:exp-pos} holds (the lhs of \eqref{eq:exp-pos} equals $\ve (1-\ve)/(4\gamma)>0$).
Note also that this is the best possible double-exponential growth rate achievable, since for any $C\ge 1/(\gamma-1)$, $D>1$ cannot be satisfied. 
  By \eqref{eq:an}, $b_n b_{n-1}\le \exp \{ C^n 2(D-1) \}$, and hence the rhs of  \eqref{eq:box-conn} is at most
\be\label{eq:box-conn-2}  \exp \{ C^n 2(D-1) \} \exp\left\{   -\la d^{-\al/2} \exp\big\{ C^{n-1} \al \ve (1-\ve)/(4\gamma)\big\} \right\}, \ee
which is summable in $n$. The Borel-Cantelli lemma ensures that the event 
$\{ \exists i\le b_n, j\le b_{n-1}: \ c_n^{\sss{(i)}}\not \leftrightarrow c_{n-1}^{\sss{(j)}}\}$ happens only finitely many times. We set $n_0(\ve)>n_w(\delta(\ve))$ to be the random index after which the complement of the event on the lhs of \eqref{eq:box-conn} holds for all $n\ge n_w(\ve)$.

Collecting the terms on the rhs of \eqref{eq:n0-1} and \eqref{eq:box-conn-2}, and summing them from $K+1$ to infinity, we can notice that the first term is dominant. We obtain that
\[ \Pv(n_0(\ve)>K)\le \exp\left\{ - c_1 \exp\big\{ \ve c_2  C(\ve)^K \big\}\right\}, \] 
where $c_1=\min\{\la d^{-\al/d}, 1\}/2, \ c_2=\min\{\al/(8\gamma),  (\gamma-1)/(4\gamma)\}$ and we assumed $1-\ve>1/2$ to obtain $c_2$. The factor $1/2$ in $c_1$ compensates the prefactors $dDC^n$ in \eqref{eq:n0-1} and $C^n 2(D-1)$ in \eqref{eq:box-conn-2}. 
This shows \eqref{eq:n0}. The proof of the statement that the centers within $\Gamma_n$ form a complete graph is the same, only $b_n b_{n-1}$ should be replaced by $b_n^2$ and $C^{n-1}$ by $C^n$ in \eqref{eq:box-conn}.
\end{proof}
The theorems with conservative edge-weights require the extension of  the boxing to connect two vertices, $0 $ and $x\in \Z^d$, where $x:=\lfloor m\underline e\rfloor$ in the proof later. For this extension, we define two infinite  sequences of annuli $(\Gamma^{\sss{(0)}}_n, \Gamma^{\sss{(x)}}_n)_{n\ge 1}$ so that $\Gamma^{\sss{(0)}}_n=\Gamma^{\sss{(x)}}_n$ for all large enough $n$: Set 
\be \label{def:nx} n(x):=\max_n \{D_n\le \|x\|/2\} = \left\lfloor \frac{ \log(\log \| x\| -\log 2) - \log D(\ve) }{\log C(\ve)}  \right\rfloor
 \ee 
where $\lfloor z\rfloor= \max\{y\in \Z: y\le z\}$.  
 For $i\le n(x)$, define $(\Gamma^{\sss{(0)}}_i, \Gamma^{\sss{(x)}}_i)_{i\le n(x)}$, centered around $0$ and  $x$, respectively, given by \eqref{eq:layering-def}.  Let us then define 
\be \wit D_{n(x)+1}:=2 \exp\{ D C^{n(x)+1}\}, \quad \mathrm{\wit{Box}}_{n(x)+1}:=\{ y\in \Z^d: \| y\| \le \wit D_{n(x)+1} \} \ee
and $\Gamma_{n(x)+1}^{\sss{(0)}} = \Gamma_{n(x)+1}^{\sss{(x)}} :=\mathrm{\wit{Box}}_{n(x)+1} \setminus ( \mathrm{Box}_{n(x)}^{\sss{(0)}} \cup \mathrm{Box}_{n(x)}^{\sss{(x)}})$ to ensure that it contains both $\Gamma^{\sss{(0)}}_{n(x)}$ and  $\Gamma^{\sss{(x)}}_{n(x)}$. Let $\Gamma_{n(x)+2}^{\sss{(q)}}:=\mathrm{Box}_{n(x)+2}\setminus \wit{\mathrm{Box}}_{n(x)+1} $ and finally, for $n\ge n(x)+2$ let
\[  \Gamma_{n}^{\sss{(0)}}=\Gamma_{n}^{\sss{(x)}}:=\mathrm{Box}_{n}\setminus \mathrm{Box}_{n-1} \] as defined in \eqref{eq:rn} (returning to the `usual' sizes). Let us call this merging system of annuli \emph{merging annuli connecting $0,x$.}
The proof of Lemma \ref{lem:layers-complete} implies the following corollary:

\begin{corollary}\label{corr:merging}
Consider the \emph{merging annuli connecting $0,x$} as described before.  With this definition, Lemma \ref{lem:layers-complete} stays valid for $(\Gamma_i^{\sss{(0)}})_{i\ge 0}$. Further, there is an $n_x(\ve)\le n(x)+1$  such that for all $n\in[n_x(\ve), n(x)]$ all the centers of boxes within $\Gamma_n^{\sss{(x)}}$ are connected to all the centers of boxes in $\Gamma_{n+1}^{\sss{(x)}}$ \emph{and} all centers of boxes within $\Gamma_{n(x)}^{\sss{(x)}}$ are connected to all centers of boxes within $ \wit \Gamma_{n(x)+1}^{\sss{(0)}}$. Finally, with the same $c_1,c_2$ as in Lemma \ref{lem:layers-complete},
\be\label{eq:nx-ve} \Pv\big(n_x(\ve) \in [K, n(x)]\mid  n_0(\ve)<n(x)\big) \le \exp\left\{ -c_1 \exp\{ \ve c_2 C(\ve)^K\}
\right\}. \ee

\end{corollary}
\begin{proof}
Note that the vertex and edge set of $\mathrm{Box}^{\sss{(0)}}_{n(x)}, \mathrm{Box}^{\sss{(x)}}_{n(x)}$ are independent since the boxes are disjoint. Thus, the error estimates in the proof of Lemma \ref{lem:layers-complete} can be applied to the two systems of annuli separately for all $n\le n(x)$. The random variable $n_0(\ve)$ can be defined as the index of the last annulus that contains a center with too small vertex-weight or a center that is not connected to all the centers in the next annuli in the infinite system $(\Gamma_k^{\sss{(0)}})_{k\le 1}$. 
On the event that $n_0(\ve)<n(x)$, all the centers of boxes  in $\Gamma_{n(x)+1}^{\sss{(0)}} = \Gamma_{n(x)+1}^{\sss{(x)}}$ have large enough vertex-weights, thus  $n_x(\ve)$ can be defined as $1$ plus the largest $k\in [0, n(x)]$ such that the annulus $\Gamma_{k}^{\sss{(x)}}$ either contains a center with too small vertex-weight or does not have all the connections to centers in $\Gamma_{k+1}^{\sss{(x)}}$. Note also that $n_x(\ve) = n(x)$ might occur, in which case the second statement of Corollary \ref{corr:merging} is empty. The error bound in \eqref{eq:nx-ve} is obtained in the same way as \eqref{eq:n0}, since conditioning on the event $n_0(\ve)<n(x)$ implies that all the centers in all the centers of boxes  in $\Gamma_{n(x)+1}^{\sss{(0)}} = \Gamma_{n(x)+1}^{\sss{(x)}}$ have large enough vertex-weights.
\end{proof}

We finish this section with two definitions. We define a unique \emph{greedy path} based on the boxing starting from the origin. This greedy path will be used in many of the proofs later on.

\begin{definition}[Greedy path]\label{def:greedy-path}
We define the infinite greedy path $\gamma^{\mathrm{gr}}_0$ starting at $0$ recursively as follows.
Fix $\ve>0$ and take $n_0(\ve)$ from Lemma \ref{lem:layers-complete}. Take the shortest path from $0$ to the set  $c_{n_0}^{\sss{(i)}}, i \le b_n$, that stays within\footnote{The set of this paths is non-empty since nearest-neighbor edges are always present.} $\mathrm{Box}_{n_0}$. Let us denote the end-vertex of this path segment by $c_{n_0}^{\mathrm{gr}}$.
Suppose we have already added $c_{n-1}^{\mathrm{gr}}$ to $\gamma_0^{\mathrm{gr}}$. By Lemma \ref{lem:layers-complete}, for $n\ge n_0(\ve)$, $c_{n-1}^{\mathrm{gr}}$ is connected by an edge to all the $b_n$ many centers in  $\Gamma_n$. Then, let us choose the edge with minimal length among these edges, and set $c_{n}^{\mathrm{gr}}$ to be its end vertex in $\Gamma_n$. That is, let
\[ i^{\mathrm{gr}}_{n}:= \arg\min_{i\le b_{n}}\{ L_{(c_{n-1}^{\mathrm{gr}}, c_{n}^{(i)}) } \}, \qquad c_{n}^{\mathrm{gr}}:= c_{n}^{(i_{n}^{\mathrm{gr}})}.\]
We denote the resulting infinite path by $\gamma^{\mathrm{gr}}_0$.
\end{definition} 

Using Corollary \ref{corr:merging}, we immediately extend this definition to construct two \emph{merging greedy paths}, one from vertex $0$ and one from vertex $x$.  
Assuming  that $n_0(\ve), n_x(\ve) \le n(x)$, let us  apply Definition \ref{def:greedy-path} twice, using the separate annuli $(\Gamma_i^{\sss{(0)}})_{i\le n(x)}, (\Gamma_i^{\sss{(x)
}})_{i\le n(x)}$ and $\Gamma_{n(x)+1}^{\sss{(0)}}=\Gamma_{n(x)+1}^{\sss{(x)}}$,  to define two greedy path segments, started at $0$ and $x$, respectively, 
\be\label{eq:merging-greedy} \gamma_0^{\mathrm {gr}} [0, c_{n(x)+1}^{\mathrm gr}(0) ] \quad \mbox{and} \quad \gamma_x^{\mathrm{gr}} [0, c_{n(x)+1}^{\mathrm {gr}}(x) ].\ee
We merge the greedy paths. When $c_{n(x)+1}^{\mathrm{gr}}(0)=c_{n(x)+1}^{\mathrm {gr}}(x)$, they are merged. Suppose $\gamma_0^{\text{gr}}$ uses $c_{n(x)+1}^{\text{gr}}(0) \neq c_{n(x)+1}^\text{gr}(x)$. Let the connecting vertex be defined as
 \be\label{eq:merging-greedy-2}  c_{n(x)+2}^{\text{gr}}:=\arg\min_{i\le b_{n(x)+1}-2}  (L_{ c_{n(x)+1}^{\text{gr}}(0), c_{n(x)+1}^{(i)} } + L_{ c_{n(x)+1}^{\text{gr}}(x),  c_{n(x)+1}^{(i)} } ),\ee i.e., we use the minimal-length path of two edges via centers to connect  $c_{n(x)+1}^{\text{gr}}(0)$ to $c_{n(x)+1}^\text{gr}(x)$.   The greedy algorithm in Definition \ref{def:greedy-path} continues from $c_{n(x)+2}^{\text{gr}}$ on the merged annuli $(\Gamma^{\sss{(0)}}_n)_{n\ge n(x)+1}$ and the two merged paths follow the same edges to infinity.
 
 \section{`Explosive' proofs}\label{s:exp-proofs}
In this section we prove the explosive part of Theorem \ref{thm:fpp-profile} and Theorem \ref{thm:distances}.
For a possibly infinite path $\gamma$ and two vertices $v,w\in \gamma$, let $\gamma[v, w], |\gamma[v,w]|_L, |\gamma[v,w]|_G$ denote the segment of the path between the two vertices $v, w$, its length and its number of edges, respectively.

\begin{proof}[Proof of Theorem \ref{thm:fpp-profile}: Explosive part]
We show that that the length of the greedy path in Definition \ref{def:greedy-path} is a.s. finite.
Since the centers of subboxes depend only on vertex-weights, but not on edge-weights, when determining $c_n^{\mathrm{gr}}$, the chosen edge-weight is the  minimum of $b_{n}$ i.i.d.\ edge-weights from distribution $L$. The length of the constructed path thus can be written as, with i.i.d.\ $(L_{k,i})$,
\be\label{eq:greedy-ray}   |\gamma_0^{\mathrm{gr}}[0, c_{n+n_0}^{\text{gr}}]|_L = |\gamma_0^{\mathrm{gr}}[ 0, c_{n_0}^{\text{gr}}]|_L + \sum_{k=n_0}^{n+n_0-1} \min\{ L_{k, 1}, L_{k, 2}, \dots, L_{k, b_{k+1}} \}.\ee
By \eqref{eq:an} and \eqref{eq:C-epsz}, $(b_k)_{k\in \N}$ grows double-exponentially. By Claim \ref{lemma:integral}, for double-exponentially growing sequences, the convergence of $\mathbf{I}(L)$ in \eqref{eq:integral-crit}  is equivalent to the summability criterion \eqref{eq:explosive-crit-1} in
Lemma \ref{thm:amidev2}. So, \eqref{eq:explosive-crit-1} is satisfied by $(b_k)_{k\in \N}$,  and this in turn implies by Lemma \ref{thm:amidev2} that the sum in
\eqref{eq:greedy-ray} is a.s.\ finite.  
Let us denote the a.s.\ limit of \eqref{eq:greedy-ray} - that is, the total length of $\gamma^{\text{gr}}$ by
\[ \lim_{n\to \infty} | \gamma^{\text{gr}}[0, c_n^{\mathrm{gr}}]|_L:=V_0^{\text{gr}}. \]
Clearly, $M_{n}(0)\le |\gamma^{\text{gr}}[0, c_{n+n_0}^{\text{gr}}]|_L \le V_0^{\text{gr}}$ since $\gamma^{\text{gr}}[0, c_{n+n_0}^{\text{gr}}]$ has at least $n$ edges 
and it might not be optimal. 
Thus, $(M_n(0))_{n\ge 1}$ is a uniformly bounded increasing sequence, hence it converges a.s.
\end{proof}

\begin{proof}[Proof of Theorem \ref{thm:distances}]
We start using \cite[Theorem 5.3]{DeiHof13},  in particular the following statement:
Let $\kappa:=\gamma-1$ when $\tau\in (1,2]$ and $\kappa :=\al/d-1=\gamma/(\tau-1)-1$ when $\tau>2$, and let $\ve>0$ arbitrarily small. Let us denote the event
\be\label{eq:A} A_{0,x}:=\left\{ \mathrm{d}_G(0,x) \ge (1-\de) 2\log \log \|x\|/|\log \kappa|\right\}. \ee
Then, \cite[Theorem 5.3]{DeiHof13} says that $\lim_{\|x\|\to \infty}\Pv\left( A_{0,x}\right)=1$.
Set $r_x:=\lfloor (1-2\delta)\log \log \|x\|/|\log \kappa| \rfloor$, where $\lfloor y \rfloor$ is the smallest integer that is at most $y$.
On the event $A_{0,x}$, the \emph{graph distance balls}  $B_{r_x}^G(0)$  and $B_{r_x}^G(x)$ are disjoint. 
Since the balls  $B_{r_x}^G(0)$  and $B_{r_x}^G(x)$ are defined in terms of \emph{graph distance} and are disjoint,  any path that connects $0$ and $x$ must intersect their boundary. Hence
\be\label{eq:boundary-bound} d_L(0, x) \ge \mathrm{d}_L(0,  \Delta B_{r_x}^G(0)) + \mathrm{d}_L (x, \Delta B_{r_x}^G(x)) = M_{r_x}(0) + M_{r_x}(x), \ee
where $M_n(y)$ was defined in \eqref{eq:mn-def}.
Since $B_{r_x}^G(0)\cap B_{r_x}^G(x)=\emptyset$ on $A_{0,x}$, all the edge-weights are independent in $B_{r_x}^G(0)$  and $B_{r_x}^G(x)$, hence
$M_{r_x}(0)$ is independent of $M_{r_x}(x)$.  In Theorem \ref{thm:fpp-profile} we have established that $M_n(y)\toas V_y$ as $n\to \infty$. Thus $M_{r_x}(0)$ converges a.s., thus also in distribution, to $V_0:=V^{\sss{(1)}}$, while $M_n(x)$ with $x$ fixed converges to $V_x$ a.s. Since $r_x\to \infty$ as $\|x\|\to \infty$, the sequence $M_{r_x}(x)$ converges \emph{in distribution} to $V^{\sss{(2)}}$ that has the same distribution as $V^{\sss{(1)}}$ under the annealed measure of the model by translation invariance. The independence of $M_{r_x}(0), M_{r_x}(x)$ implies that the limit variables $V^{\sss{(1)}}, V^{\sss{(2)}}$ are independent (while $V_0, V_x$ are not). Rewriting the inequality in \eqref{eq:boundary-bound} results in   
\[ d_L(0, x)-(V_0-V_x) \ge  (M_{r_x}(0)-V_0) + (M_{r_x}(x)-V_x). \]
The lhs converges to zero in probability under the event $A_{0,x}$, establishing \eqref{eq:dist-minus-explosion}. 
\end{proof}
\begin{proof}[Sketch of proof of Claim \ref{claim:conj1-2}]
It is possible to show that the vertex vertex-weights along the shortest path tend to infinity. Thus, we can follow the shortest paths $\gamma^\star_0$ and $\gamma^\star_x$ until we reach vertices with sufficiently large vertex-weights (and sufficiently large degree), say $v_0^\star$ and $v_x^\star$. Then, we can connect $v_0^\star$ to $c_n^{\text{gr}}$ and $v_x$ to $c_m^{\text{gr}}$ for some $n,m$ within length $\delta/3$, where $c_n^{\text{gr}}, c_m^{\text{gr}}$ are centers of boxes used on the \emph{greedy path} $\gamma^{\text{gr}}$  constructed in the proof of Theorem \ref{thm:fpp-profile}. 
Then,  $\gamma^{\text{gr}}[c_n^{\text{gr}}, c_m^{\text{gr}}]$ connects  $c_n^{\text{gr}}$ to $c_m^{\text{gr}}$ within length $\ve/3$, when $\min\{n,m\}$ is sufficiently large. The problem however is that the edge-weights fail to be i.i.d.\ around the shortest exploding paths $\gamma^\star_0$ and $\gamma^\star_x$. Thus, the second step - connecting $v_0^\star$ to $c_n^{\text{gr}}$ and $v_x$ to $c_m^{\text{gr}}$ via short paths - is only possible if we have some guarantee that the shortest paths $\gamma_0^\star, \gamma_x^\star$ do not spend too much time in the neightborhoods of $v_0^\star, c_n^\text{gr}, v_x^\star, c_m^{\text{gr}}$, respectively. A double-exponential speed like the one in Conjecture \eqref{conj:growth} is sufficient, since in this case $\gamma^\star$ spends in average a tight number of steps in each annulus $\Gamma_n$.
\end{proof}
\section{The Bernoulli BRW when $\gamma\in(1,2)$}\label{s:bernoulli}
In the next section (Section \ref{s:gamma12}) we prove the theorems concerning the $\gamma\in(1,2)$ case: the Conservative Part of Theorem \ref{thm:fpp-profile}, and Theorems \ref{thm:dist-cons}, \ref{thm:max-displace}. 
A key ingredient to the proofs is a bound on the double-exponential growth rate of the generation sizes and the maximal displacement in the upper bounding $\BerBRW$ (without the edge-weights $L$ added). This section is devoted to the analysis of $\BerBRW$.
The main result of this section is the following proposition.
 \begin{proposition}[Double-exponential growth and maximal displacement of $\BerBRW$]  \label{prop:ber-gensizes}
Let $\CG_k^{\mathrm{B}},Z_k^{\mathrm{B}}$ denote the set and number of vertices in generation $k$ in $\BerBRW$, defined in Section \ref{s:explore-BP}, with parameters $\al>d, \tau$ and $\gamma=\al(\tau-1)/d\in(1,2)$. Let 
\be\label{eq:ck} c_k(\ve, i):=  2\exp\left\{ i(1+\ve) \frac{d}{\al}\Big(\frac{1+\ve}{\gamma-1}\Big)^k\right\}, \qquad
S_k(\ve, i):=\exp\left\{ i z \Big(\frac{1+\ve}{\gamma-1}\Big)^k \right\},
\ee
for a constant $z\le 2\gamma/(\al(\gamma-1))+(2+d/\al)/(d-\al)$ defined in \eqref{def:z} below.
Then, for every $\ve>0$ there exists an a.s. finite random variable $Y=Y(\ve)<\infty$ such that   the event 
\be\label{eq:ck2}\ba  &\forall k\ge 0:  Z_k^{\mathrm{B}} \le  c_k(\ve, Y) \  \mathrm{ and }\  \CG_k^{\mathrm{B}} \cap \left( B^{2}_{S_k(\ve, Y)}\right)^c  =\emptyset   \ea \ee
holds. Further, $Y(\ve)$ has exponentially decaying tails with 
	\be\ba \label{eq:Y-tail}\Pv(Y(\ve)\ge K) &\le 2C_{\ve} \exp\left\{ - K \ve (d/\al)(1+\ve)/(4(\gamma-1)) \right\}+ \wit  C_{\ve} \exp\{ -K \ve\}\\
	&=:2p_{\ve, K}+\wit p_{\ve, K}.
	\ea \ee
where $C_{\ve}, \wit C_\ve$ are constants that do not depend on $K$.
 \end{proposition}
Consequently, both the size, and the maximal displacement in generation $k$,  of $\BerBRW$ grow double-exponentially with rate at most $(1+\ve)/(\gamma-1)$. To be able to prove Proposition \ref{prop:ber-gensizes}, we need to bound not only the generation sizes but also the number of individuals with vertex-weight in a given (generation-dependent) interval, which is the content of the Lemma \ref{lem:types} below.
For  fixed $\ve,i>0$ and all $k\ge 0$ define 
\begin{equation}\label{eq:mk}
m_k:=m_k(\ve, i)=\exp\left\{i\Big(\frac{1+\ve}{\ga-1}\Big)^k\right\},	
\end{equation}
and for some fixed integer $n>1$ and all $1\le j\le n$ define the intervals
\be \label{eq:ikj}I_k^j:=I_k^j(\ve, i):=\left[ m_k^{1-j/n}, m_k^{1-(j-1)/n}\right), \qquad I_k^0:=[ m_k, \infty).\ee
Let $W(\CG_k^{\mathrm{B}}):=\{W_x\}_{x\in \CG_k^{\mathrm{B}}}$ denote the list of vertex-weights of individuals in $\CG_k^{\mathrm{B}}$.  
Roughly speaking, we would like to track the tail distribution of vertex-weights of individuals in generation $\CG_k^{\mathrm{B}}$. For this we use a discretisation technique, and set some $1/h$ that is large but not too large, and then count the number of individuals with vertex-weight in an interval that can be heuristically described as follows: the expected maximal vertex-weight in generation $k-1$ \eqref{eq:mk}, raised to two powers that form a $h$-wide interval, \eqref{eq:ikj}. It turns out this is the right scaling, i.e., we have to group individuals with vertex-weight that are roughly a given power away from the maximal vertex-weight.
Further, let us set 
\be \label{def:z} z:=\max\left\{\frac{(\frac{\gamma}{\gamma-1}+ 1/n)}{\al} + \frac{2\ve}{1+\ve}, \frac{1+\ve + d(\gamma-1)/(\al n) }{d-\al}    \right\},\ee  and define $S_k:=S_k(\ve, i):=\left( m_{k}(\ve, i)\right)^{z}$
as in \eqref{eq:ck}.
Then, Proposition \ref{prop:ber-gensizes} is a consequence of the following two lemmas:
\begin{lemma}[Weights and displacement in $\CG_k^{\mathrm B}$]\label{lem:types}
Consider $\BerBRW$ as described in Section \ref{s:explore-BP} and let $\gamma\in(1,2)$. Fix $\ve>0$ and some integer $n\in [2(\gamma-1)/\ve, 3(\gamma-1)/\ve]$. Let 
\begin{equation}\label{eq:induction-a}\ba 
	E_k&:=E_k(\ve, i)=\left \{\forall \ 0 \le j \le n:  |W(\CG_k^{\mathrm{B}})\cap I_{k+1}^j| <  m_k^{(1+\ve)(d/\al)(j/n)}. \right\}\\
	F_k&:=F_k(\ve, i)=\left \{  \CG_k^{\mathrm{B}}\cap (B^2_{S_{k} }(0))^c = \emptyset \right\}
\ea\end{equation} 
\begin{enumerate}
\item  
Then, for all sufficiently large $i$, and with $p_{\ve, i}, \wit p_{\ve,i}$ defined in \eqref{eq:Y-tail},
\begin{equation}\label{eq:induction-b}
	\Pv\Big(\bigcap_{k=0}^\infty E_k\Big) 
   \ge 1-p_{\ve,i}, \qquad 
	\Pv\Big(\bigcap_{k=0}^\infty (E_k\cap F_k)\Big) 
   \ge 1- 2p_{\ve,i}-\wit p_{\ve,i},
  \end{equation}
\item  $Z_k^{\mathrm{B}} \le 2 m_k^{(1+\ve)d/\alpha}=c_k(\ve, i)$ as defined in \eqref{eq:ck} holds on $E_k$.
\item On $E_k$, $W_{\max}^{(k)}:=\max_{x\in \CG_k^{\mathrm B}} W_x$ satisfies that $W_{\max}^{(k)} \le m_{k+1}$.

\end{enumerate}

\end{lemma}
We first prove Proposition \ref{prop:ber-gensizes} subject to Lemma \ref{lem:types}. 
\begin{proof}[Proof of Proposition \ref{prop:ber-gensizes} subject to Lemma \ref{lem:types}]
Let 
\be \label{eq:wit-ek}\wit E_k(\ve, i):=\{ Z_k^{\mathrm{B}} \le c_k(\ve, i)\}\cap \Big\{ \CG_k^{\mathrm{B}} \cap \left( B^{2}_{S_k(\ve, i)}\right)^c  =\emptyset \Big\}\subset (E_k\cap F_k).\ee With a Borel-Cantelli type argument, we strengthen the result of Lemma \ref{lem:types}. For all integers $i\ge 1$, set  $H_i:=\bigcap_{k=0}^{\infty} \wit E_k(\ve, i)$. 
Then using the second inequality in \eqref{eq:induction-b}, 
\be\label{eq:hi-estimate}\sum_{i=1}^\infty \Pv(H_i^c)\le \sum_{i=1}^\infty 1-\Pv\left( \bigcap_{k=0}^\infty (E_k(\ve,i) \cap F_k(\ve, i))\right) \le \sum_{i=1}^\infty (2p_{\ve, i} + \wit p_{\ve, i}) <\infty, \ee
since the last sums are geometric series in $i$.	
Thus, by the Borel-Cantelli lemma, only finitely many $H_i^c$-s occur. Let $Y(\ve)  $ be the (random) first index at and after which no $H_i^c$ occur anymore. This implies that $H_{Y(\ve)}=\bigcap_{k\ge 0} \wit E_k(\ve,Y(\ve))$ holds, showing \eqref{eq:ck2}.
To obtain the tail behaviour of $Y(\ve)$, note that for all $i,j\ge 0$ $H_{i+j}^c\subset H_i^c$,
(since  $c_k(\ve, i)$ and $S_k(\ve, i)$ are increasing in $i$). Thus 
\be\label{eq:y-tail-111}\Pv(Y(\ve)\ge K)=\Pv\Big(\bigcup_{i\ge K} H_i^c\Big) = \Pv(H_K^c)\le  2p_{\ve, K}+\wit p_{\ve, K},\ee
finishing the proof.  
\end{proof}
Before proceeding to the proof of Lemma \ref{lem:types}, we state and prove a claim about the expected number of children of an individual with certain vertex-weights and norms. 
Consider an individual with vertex-weight $w$ located at the origin. For a number $S>0$, let 
$N_{w}(\mathrm d v, \ge S)$ denote the number of its children with vertex-weight in the interval $(v, v+\mathrm dv)$ and displacement with norm at least $S$ in $\BerBRW$. Let us further write $N_{w}(\ge u, \ge 0), N_{w}(\ge 1, \ge S)$  for the number of its children with vertex-weight at least $u$, and displacement with norm at least $S$ in $\BerBRW$, respectively.  
\begin{claim}\label{claim:expected-offspring}
Consider an individual located at $0\in \Z^d$ with vertex-weight $w$ in $\BerBRW$ and recall $\ell(\cdot)$ from \eqref{eq:tail-W}.
Then for all $\gamma>1$, there exists a constant $0<M_{d,\la}<\infty$ such that
\be\label{eq:bound-on-single} \Ev[N_{w}(\ge u, \ge 0) ] \le M_{d,\la} w^{d/\alpha} \ell(u)u^{-(\gamma-1)d/\al}. \ee 
Further, 
\be \label{eq:bound-on-norm-2} \Ev[N_{w}(\ge 1, \ge S) ] \le M_{d,\la} w S^{d-\al}+ M_{d,\la} w^{\tau-1} S^{-d(\gamma-1)}  \ell(S^\al/(\la w)).   \ee

\end{claim}
\begin{proof}
 Recall the definition of the edge probabilities in $\BerBRW$ from \eqref{eq:Ber-prob}, and that we write $\mathcal{N}=\{ y: \| y\|=1 \}$ for the nearest neighbors of the origin.
 Similarly to \eqref{eq:kernel-1}, the distributional identity 
\begin{equation}\label{eq:neighbors-dist}
	N_{w}(\mathrm d v, \ge S) \ {\buildrel d \over =}   \sum_{y\in \Z^d: \| y\|\ge S
	} \ind_{\{W_y\in (v, v+\mathrm dv)\}} \left( \ind_{\{y \in \mathcal{N}\}} + \ind_{\{y\not\in \mathcal{N}\}} \mathrm{Ber}(1-\exp\{-\la w v/\|y\|^\al\}) \right)
\end{equation}
  holds
where $W_y$ is i.i.d.\ from $W$ and the Bernoulli r.v.s are (conditionally on $u$) are independent of $W_y$.
 First we aim to show \eqref{eq:bound-on-single}. By the same argument as in \eqref{eq:mu-ber},  and with $c_{d,\la}$ as in \eqref{eq:cdla-def},
 \be\label{eq:bound-on-single-2}\ba \Ev[N_{w}(\mathrm dv, \ge 0) ] &\le \Pv(W \in (v, v+\mathrm dv)) v^{d/\al}w^{d/\al} c_{d,\la}.\ea
\ee Integrating with respect to $v$ for all $v\ge u$ yields
\begin{equation}\label{eq:children-single}\ba
	\Ev[N_{w}(\ge u, \ge 0)] &\le  c_{d,\la} w^{d/\al} \Ev[W^{d/\al} \ind_{\{W\ge u\}}]=M_{d,\la} w^{d/\alpha} \ell(u) u^{d/\alpha-\tau+1}\\
	&=M_{d,\la} w^{d/\alpha} \ell(u)u^{-(\tau-1)(\gamma-1)/\gamma}, 
	\ea
\end{equation}
where we have used \eqref{eq:tail-W} and the Karamata-type theorem in \eqref{eq:karamata}  to establish that $\Ev[W^{d/\al} \ind_{\{W\ge u\}}]\le c \ell(u) u^{d/\alpha-\tau+1}$ holds for all $\al>d$ and then set $M_{d,\la}:=c \,c_{d,\la} $, and finally that $d/\alpha=(\tau-1)/\gamma$. 

Next we prove \eqref{eq:bound-on-norm-2}. Let us assume that $S>2$ so nearest-neighbor edges do not play a role. To bound $\Ev[N_{w}(\mathrm d v, \ge S)]$ we can use any of the two bounds $\la w v/\|y\|^{\al}$ or $1$ on the expectation of the Bernoulli r.v.s in \eqref{eq:neighbors-dist}.  We distinguish three cases depending on whether $v\le S^\al/wv$ or not, and whether $y\ge (\la wv)^{1/\al}$ in the latter case.
With $F_W(\mathrm dv):= \Pv(W\in (v, v+\mathrm dv))$, separating the three different cases and taking expectations,
\be \label{eq:s-norm-1}\ba
\Ev[N_w(\ge 1, \ge S)] &\le \int_{1\le v \le S^\al/(\la w)} F_W(\mathrm dv) \sum_{ y\in \Z^d: \|y \|\ge S } \la w v/\| y\|^\al \\
&+ \int_{v\ge  S^\al/(\la w)} F_W(\mathrm dv) \Bigg( \sum_{\substack{ y\in \Z^d\\  \|y \|\le (\la wv)^{1/\al} }}  1  + \sum_{\substack{ y\in \Z^d\\  \|y \|\ge (\la wv)^{1/\al} }} \frac{\la w v }{\| y\|^\al}\Bigg).
\ea \ee
Let us denote the first term on the rhs of \eqref{eq:s-norm-1} by $T_1$. The inner sum in $T_1$ is at most $\la w v 2\pi_d S^{d-\al}$, by the definition of $\pi_d$ after \eqref{eq:cdla-def}. Thus the first term is at most 
\be\label{eq:first-term} T_1\le (\la w) 2 \pi_d S^{d-\al} \Ev[W \ind_{\{W\le S^\al/(\la w)\}}].\ee
When $\tau\in (1,2)$, $\Ev[W]=\infty$ and by Karamata's theorem in \eqref{eq:karamata}, 
\[ \Ev[W\ind_{\{W\le S^\al/(\la w)\}}]= \int_1^{S^\al/(\la w)} \ell(t) t^{\tau-1} \mathrm dt \le C_{\tau} \left(S^{\al} / (\la w)\right)^{2-\tau} \ell(S^\al/(\la w)). \]
Combining this estimate with \eqref{eq:first-term}, for $\tau\in (1,2)$, 
 \be\label{eq:first-term-12} T_1 \le S^{d-\al(\tau-1)} (\la w)^{\tau-1} 2 \pi_d C_{\tau} \ell(S^\al/(\la w)), \ee
contributing to the second term in \eqref{eq:bound-on-norm-2}.
When $\tau>2$, $\Ev[W]=m_1<\infty$, thus in this case
\be\label{eq:first-term->2} T_1 \le 2 \pi_d m_1 (\la w)  S^{d-\al},  \ee
yielding the first term in \eqref{eq:bound-on-norm-2}.
Let us write $T_{21}$ and $T_{22}$ for the two integrals arising when distributing the sum in the second term in \eqref{eq:s-norm-1}. Since $S\ge 2$, the inner sum in $T_{21}$ is at most
\[ \sum_{y:\|y\|\le(\la wv)^{1/\al} } 1 \le 2 C_d\int_{0}^{(\la wv)^{1/\al}} y^{d-1} \mathrm dy = 2 C_d (\la wv)^{d/\al}.  \]
Thus, for all $\tau>1$, by Karamata's theorem in \eqref{eq:karamata}, and \eqref{eq:tail-W}
\be\label{eq:t21} \ba  T_{21} &\le  2 C_d (\la w)^{d/\al} \Ev[ W^{d/\al} \ind_{\{ W \ge S^{\al}/(\la w)\}}] \\
&\le 2 C_d  (\la w)^{d/\al} \wit C_\tau (S^{\al}/(\la w))^{d/\al -\tau +1} \frac{d}{\al} \ell(S^\al/(\la w))\\
&= 2 C_d \wit C_\tau S^{d-\al(\tau-1)} (\la w)^{\tau-1} \ell(S^\al/(\la w)),
\ea\ee
as required in \eqref{eq:bound-on-norm-2}.
We continue estimating $T_{22}$ in \eqref{eq:s-norm-1}. Here, similarly as in \eqref{eq:bound-on-single-2}, the inner sum is at most $2\pi_d (\la w v)^{d/\al}$, so, 
\be\label{eq:t22} T_{22} \le 2 \pi_d (\la w)^{d/\al} \Ev[W^{d/\al} \ind_{\{W \ge S^\al/\la w\}}] 
\le 2 \pi_d \wit C_\tau (\la w)^{\tau-1} S^{d-\al(\tau-1)}   \ell(S^\al/(\la w)),
\ee
where we have obtained the second line in a similar way as in \eqref{eq:t21}.
Combining  \eqref{eq:first-term-12}, \eqref{eq:first-term->2}, \eqref{eq:t21} and \eqref{eq:t22} yields \eqref{eq:bound-on-norm-2} by setting $M_{d,\la}$ to be the sum of all constants involved.
\end{proof}

We are ready to prove Lemma \ref{lem:types}.
\begin{proof}[Proof of Lemma \ref{lem:types}]
Note that Parts (b) and (c) are direct consequences of Part (a).
Namely, assuming (a), using the sum of a geometric series we obtain that
\[ Z_k^{\mathrm{B}} = \sum_{j=0}^{n} |W(\CG_k^{\mathrm{B}})\cap I_{k+1}^j| < \sum_{j=0}^{n}  m_k^{(1+\ve)(d/\al)(j/n) } \le 2 m_k^{(1+\ve)d/\al},   \]
where we have assumed that $i$ is so large in \eqref{eq:mk} that $m_k^{(1+\ve)(d/\al)(1/n)}-1>m_k^{(1+\ve)(d/\al)(1/n}/2$ holds for all $k\ge 1$. 
Next, $W_{\max}^{\sss{(k)}} \le m_{k+1}$ is a rewrite of Part (a) for $j=0$,  stating that $|W(\CG_k^{\mathrm{B}})\cap I_{k+1}^0| = 0$.

Next we prove Part (a).
Let us write $\Pv_W(\cdot):=\Pv(\cdot| (W_y)_{y\in \Z^d})$, with $\Ev_W[\cdot]$ the corresponding conditional expectation.
Given the vertex-weights $(W_y)_{y\in \Z^d}$, the number and vertex-weight of individuals in $\CG_k^\mathrm{B}$ only depends on individuals in $\CG_{k-1}^\mathrm{B}$, hence
 $\Pv_W(E_k|\cap_{s\le k-1}E_{s})=\Pv_W(E_k\mid E_{k-1})$. So 
\be\ba \label{eq:pa-W} \Pv_W\Big(\bigcap\limits_{s= 0}^k E_s\Big)&=\Pv_W(E_0)\prod_{s=1}^k \Pv_W\Big(E_s\mid \bigcap\limits_{j\le s-1}E_{j}\Big)\\
&=\prod_{s=0}^k (1-\Pv_W(E_s^c\mid E_{s-1})) \ge 1-\sum_{s=0}^{k} \Pv_W(E_s^c\mid E_{s-1}), \ea\ee
where we set $E_{-1}:=\Omega$ the full probability space in the second line. 
Taking expectation of both sides with respect to $(W_y)_{y\in \Z^d}$ yields
\be\ba \label{eq:pa2} \Pv\Big(\bigcap\limits_{s= 0}^k E_s\Big) \ge 1-\sum_{s=0}^{k} \Pv(E_s^c\mid E_{s-1}). \ea\ee
In what follows we inspect $\Pv(E_s^c\mid E_{s-1})$. Let us start with $s=0$, that is, $\Pv(E_0^c)$. Since the only individual in $\CG_0^{\mathrm{B}}$ is the root, $E_0$ is satisfied for all $j\le n$ once $m_0^{h(1+\ve)d/\al} >1$ and the vertex-weight of the root $W_0 < m_1$.  
Thus, for $i$ large enough such that $m_0^{h(1+\ve)d/\al} >1$, using \eqref{eq:tail-W},
\be \label{eq:a0} \Pv(E_0^c) = \Pv( W_0 > m_1) = m_1^{ - (\tau-1) } \ell(m_1 ) 
\le \big(m_0^{ (1+\ve)}\big)^{-(\tau-1)/(2(\gamma-1)) }, \ee
where we got rid of $\ell(\cdot)$ on the rhs of the first line using Potter's theorem in \eqref{eq:slowly-1}.
Let us write $E_{k,j}$ for the event that the statement in Part (a) holds for a specific $j$. Then $E_s=\cap_{j\le 1/h} E_{s,j}$ and thus $\Pv(E_s^c \mid E_{s-1})\le \sum_{j\le 1/h} \Pv( E_{s,j}^c\mid E_{s-1})$.
   We estimate $\Pv(E_{s,j}^c\mid E_{s-1})$ using Markov's inequality as
\begin{equation}\label{eq:Markov-on-term-i}
\Pv(E_{s,j}^c\mid E_{s-1})\le \Ev\left[| W(\CG_s^{\mathrm{B}})\cap I_{s+1}^j| \mid E_{s-1}\right] / m_s^{(1+\ve)(d/\al)(j/n)}
\end{equation}
To bound the numerator, note that each individual in $W(\CG_s^{\mathrm{B}})\cap I_{s+1}^j$ is a child of some individual in $W(\CG_{s-1}^{\mathrm{B}})\cap I_{s}^{h}$ for \emph{some} $h\le n$. By the inductive assumption, on $E_{s-1}$, the bound  \eqref{eq:induction-a}) holds on the number of  individuals in  $W(\CG_{s-1}^{\mathrm{B}})\cap I_{s}^{h}$, and we can use \eqref{eq:bound-on-single} from Claim \ref{claim:expected-offspring} to bound the expected number of children of these individuals. Note that the number of children of different individuals depend through the environment $(W_y)_{y\in \Z^d}$,  the expectation is linear and hence dependencies do not cause a problem. 
Since the bound in  \eqref{eq:bound-on-single} is increasing in $w$, we use the upper end of the interval $I_{s}^{h}$, that is, $m_s^{1- (h-1)/n}$, for the vertex-weight of the individuals in $W(\CG_{s-1}^{\mathrm{B}})\cap I_{s}^{h}$.
 Thus we bound
 \be \label{eq:children-multi-1}
\Ev\left[|W(\CG_s^{\mathrm{B}})\cap I_{s+1}^j|\mid E_{s-1}\right]\le \sum_{h=1}^n |W(\CG_{s-1}^{\mathrm{B}})\cap I_{s}^h|\cdot\Ev[ N_{m_s^{1-(h-1)/n}} (\ge m_{s+1}^{1-j/n}, \ge 0)], \ee
where the requirement  that the child in $\CG_s^{\mathrm{B}}$ has vertex-weight at least $m_{s+1}^{1-j/n}$, is an upper bound. Using the bound on $|W(\CG_{s-1}^{\mathrm{B}})\cap I_{s}^h|$ on the event  $E_{s-1}$, (see \eqref{eq:induction-a}), as well as \eqref{eq:bound-on-single},
 \be\ba \label{eq:children-multi-12}
\Ev&\left[|W(\CG_s^{\mathrm{B}})\cap I_{s+1}^j|\mid E_{s-1}\right] \\
&\le \sum_{h=1}^n m_{s-1}^{(h/n) (1+\ve)d/\al} \cdot M_{d,\la}\, m_s^{(1-(h-1)/n)d/\al} \ell(m_{s+1}^{1-j/n})m_{s+1}^{-(1-j/n) (\gamma-1)d/\al}.
 \ea\ee
We carry out some analysis to bound the rhs \eqref{eq:children-multi-12}. By using the recursion $m_{s}=m_{s-1}^{(1+\ve)/(\gamma-1)}$ on the middle factor, we collect factors of $m_{s-1}$ containing the exponent $h-1$ to form a geometric series:
 \[ \ba \le  M_{d,\la} m_{s-1}^{(d/\al)(1+\ve) \left(1/(\gamma-1)+1/n\right)} \ell(m_{s+1}^{1-j/n}) &m_{s+1}^{-(1-j/n) (\gamma-1)d/\al} \\
 &\cdot \sum_{h=1}^n \left(m_{s-1}^{(d/\al) (1+\ve) ( 1/(\gamma-1) -1)/n}\right)^{-(h-1)}. \ea \]
The geometric sum is at most $2$ since its parameter is at most $1/2$ for $i$ large enough.
Using the recursion on $m_{s+1}$ and  collecting exponents containing $j/n$,  the rhs of \eqref{eq:children-multi-12} is at most 
 \[ \ba &\le 2 M_{d,\la} m_{s}^{(j/n)(1+\ve)d/\al} \cdot m_s^{-(1+\ve)d/\al} m_{s-1}^{ (d/\al)(1+\ve) \left(1/(\gamma-1)+1/n\right)} \ell(m_{s+1}^{1-j/n})\\
 &\le2 M_{d,\la} m_{s}^{(j/n)(1+\ve)d/\al} \cdot \left( m_{s-1}^{ (1+\ve)d/\al }\right)^{ 1/n  -  \ve/(\gamma-1)} \ell(m_{s+1}^{1-j/n}), \ea   \]
 where we have used the recursion on $m_s$  on the factor $m_s^{-(1+\ve)d/\al}$ in the first line to obtain the second line. 
The  exponent of $m_{s-1}^{ (1+\ve)d/\al }$ is $1/n-\ve/(\gamma-1)<-\ve/(2(\gamma-1))$ by the lower bound on $n$ in the statement of Lemma \ref{lem:types}. We get rid of the slowly varying function by using Potter's bound in \eqref{eq:slowly-1} and obtain that for some constant $M_{\ve}$ independent of $i$, the following bound holds:
 \[ 
 \left( m_{s-1}^{(1+\ve)d/\al }\right)^{-\ve/(4 (\gamma-1))} \ell(m_{s+1}^{1-j/n}) < \sup_{x\ge 0}\{x^{-\xi}\ell(x)\}\le   M_{\ve},
 \]
 with $\xi:=(\tau-1)\ve(\gamma-1)/(4\gamma(1+\ve))$. The same bound holds for  all $j\le n$ since we took the smallest exponent $\xi$ that is obtained at $j=0$.
  Thus we have arrived at
  \[ 
  \Ev\left[|W(\CG_s^{\mathrm{B}})\cap I_{s+1}^j|\mid E_{i-1}\right]\le 2 M_{d,\la} M_{\ve}  m_{s}^{(j/n)(1+\ve)d/\al} \cdot \left( m_{s-1}^{ (1+\ve)d/\al }\right)^{-\ve/(4(\ga-1))}.
  \]
  Using this bound in \eqref{eq:Markov-on-term-i}, we see that the factor $m_{s}^{(j/n)(1+\ve)d/\al}$ cancels. Then, summing the rhs of \eqref{eq:Markov-on-term-i} in $j\le n$ yields that 
 \be\label{eq:esc-cond-es} \Pv(E_s^c \mid E_{s-1})\le \sum_{j=0}^n \Pv(E_{s,j}^c\mid E_{s-1}) \le (n+1) 2 M_{d,\la} M_{\ve}  \left( m_{s-1}^{ (1+\ve)d/\al }\right)^{-\ve/(4(\gamma-1))}.
 \ee 
 Note that $n+1\le 3 (\gamma-1)/\ve+1$ is a constant. More importantly, since $m_i$ grows double-exponentially, the expression on the rhs is summable in $s$ and its sum from zero to infinity is dominated by the sum of a geometric series. Choose $i$ so large that $m_0^{(d/\al)(1+\ve)\ve/(4(\gamma-1))}>2$ and then the sum is at most twice its first term (Otherwise, the constant prefactor changes only). Thus, combining this with \eqref{eq:pa2}, we obtain that 
   \eqref{eq:pa2} turns into
   \be\label{eq:esc-sum} \ba
   \Pv\Big(\bigcap_{s=0}^k E_s\Big) &\ge 1- 4 (n+1) M_{d,\la} M_{\ve} m_0^{-(d/\al)(1+\ve)\ve/(4(\gamma-1))}-\Pv(E_0^c)\\
   &\ge 1- 16 (\gamma-1)\ve^{-1} M_{d,\la} M_{\ve} \exp\left\{ - i(d/\al)(1+\ve)\ve/(4(\gamma-1)) \right\}=1-p_{\ve,i},\ea
   \ee
   where we have used \eqref{eq:a0} to estimate $\Pv(E_0^c)$, and one has to choose $\ve$ so small that $\ve/4\le \tau-1/2$ so that that term is swallowed by the second term (when increasing the constant prefactor to $16$, say). Let us set $C_\ve:=16 (\gamma-1) M_{d,\la} M_{\ve}/\ve$ in the definition of $p_{\ve, i}$ in \eqref{eq:Y-tail}.
  Note that the rhs does not depends on $k$. Taking $k$ to infinity, we obtain that 
  \be
  \Pv\Big(\bigcap_{s=0}^\infty E_s\Big)=\lim_{k\to \infty}\Pv\Big(\bigcap_{s=0}^k E_s\Big) \ge 1- p_{\ve,i}.
  \ee
finishing the proof of the first inequality in \eqref{eq:induction-b}. 
Next we prove the second inequality in \eqref{eq:induction-b}.
Let us write $\CS_k:=B^2_{S_k}(0)$.
Similarly as in \eqref{eq:pa2}, we use the Markov branching property of $\BerBRW$ across generations, which  implies that $\Pv_W(E_k\cap F_k|\cap_{s\le k-1}E_{s}\cap F_s)=\Pv_W(E_k \cap F_k\mid E_{k-1}\cap F_{k-1})$. With an analogous rewrite as in \eqref{eq:pa2}, we obtain that  
\be\ba \label{eq:pa34} \Pv_W\Big(\bigcap\limits_{s= 0}^k (E_s\cap F_s)\Big) &\ge 1-\sum_{s=0}^{k} \Pv_W((E_s\cap F_s)^c\mid E_{s-1}\cap F_{s-1})\\
&\ge 1- \sum_{s=0}^{k} \Pv_W(E_s^c\mid E_{s-1}\cap F_{s-1}) -\sum_{s=0}^{k} \Pv_W(F_s^c\mid E_{s-1}\cap F_{s-1}),  \ea\ee
where we have set $F_{-1}:=\Omega$ the full probability space again. Taking expectations on both sides with respect to the environment $(W_y)_{y\in \Z^d}$ results in
\be\label{eq:pa3} \Pv\Big(\bigcap\limits_{s= 0}^k (E_s\cap F_s)\Big) \ge 1- \sum_{s=0}^{k} \Pv(E_s^c\mid E_{s-1}\cap F_{s-1}) -\sum_{s=0}^{k} \Pv(F_s^c\mid E_{s-1}\cap F_{s-1}).\ee
First we give an upper bound on the first sum on the rhs. 
We shall inductively use that $\Pv(E_{s-1}\cap F_{s-1})\ge 1/2$ (in a bootstrap-type argument). This trivially holds for $s=0$ since $E_{s-1}\cap F_{s-1}=\Omega$ by definition. Assuming that $\Pv(E_{s-1}\cap F_{s-1})\ge 1/2$ holds for all $s\le k-1$, by the definition of conditional probability, dropping $F_{s-1}$ from the numerator and dividing by $\Pv(E_{s-1})$ yields
\[ \Pv(E_s^c\mid E_{s-1}\cap F_{s-1}) \le \frac{\Pv(E_s^c \cap E_{s-1})}{\Pv(E_{s-1}\cap F_{s-1})} \le  \frac{\Pv(E_s^c | E_{s-1})}{\Pv(E_{s-1}\cap F_{s-1})} \le 2 \Pv(E_s^c | E_{s-1}). \]
We can now use \eqref{eq:esc-cond-es} in the proof of Lemma \ref{lem:types}  to  bound the rhs, and the argument  between \eqref{eq:esc-cond-es} and \eqref{eq:esc-sum} to estimate its sum over $s$.  We arrive that the first sum in \eqref{eq:pa3} is at most $2p_{\ve,i}$. 

It remains to estimate the second sum on the rhs of \eqref{eq:pa3}. 
We start with $s=0$. In generation 0 the only individual is the root, located at $0$. By \eqref{eq:ck}, $S_0>0$, so $F_0=\{\CG_0^{\mathrm B} \cap \CS_0^c=\emptyset\} $ is always satisfied. We continue bounding the $s$th term using Markov's inequality by calculating the expected number of individuals in $\CG_s^{\mathrm{B}}\cap \CS_s^c$. Similarly as in \eqref{eq:children-multi-1}, each of these individuals is a child of some individual in $W(\CG_{s-1}^{\mathrm{B}}) \cap I_{s}^h$ for some $h\le n$. Since we condition on $E_{s-1}\cap F_{s-1}$, all these parents have norm at most $S_{s-1}$ and we can use the bound on their number from \eqref{eq:induction-a}. Note that each child in $\CS_{s}^c$ thus have to has replacement at least $S_s-S_{s-1}$. Thus, analogously to \eqref{eq:children-multi-1}, but now focusing on the \emph{location} of the children rather than on their vertex-weight,
\be \Ev[|\CG_s^{\mathrm{B}}\cap \CS_s^c| \mid E_{s-1}\cap F_{s-1}] \le \sum_{h=1}^n |W(\CG_{s-1}^{\mathrm{B}}) \cap I_{s}^h|\cdot \Ev[ N_{m_s^{1-(h-1)/n}} (\ge 1, \ge S_s-S_{s-1})], 
 \ee
where we  have again used that the expected number of children is monotone increasing in the vertex-weight of the parent. We use that $S_s-S_{s-1}\ge S_s/2$ when $i$  large enough, in \eqref{eq:mk}, and then the bound \eqref{eq:bound-on-norm-2} in Claim \ref{claim:expected-offspring}, as well as \eqref{eq:induction-a}, to obtain
\be\ba\label{eq:two-terms}  \Ev&[|\CG_s^{\mathrm{B}}\cap \CS_s^c| \mid E_{s-1}\cap F_{s-1}] \le \sum_{h=1}^n  m_{s-1}^{(h/n)(1+\ve)d/\al}   \cdot M_{d,\la}  m_s^{(1-(h-1)/n)}(S_s/2)^{d-\al}\\
&+ \sum_{h=1}^n  m_{s-1}^{(h/n)(1+\ve)d/\al} M_{d,\la} m_s^{(\tau-1)(1-(h-1)/n)} (S_s/2)^{-d(\gamma-1)}  \ell(S_s^\al/(\la m_s^{1-(h-1)/n})).\\
 \ea\ee
Denoted by $\wit T_1$ and $\wit T_2$ the first and second sum on the rhs. We start with $\wit T_1$. For simplicity of formulas let us express $S_s:=m_{s-1}^{Z/(\al-d)}$ with $Z$ obtainable from \eqref{def:z}, \eqref{eq:ck} and using the recursion $m_s=m_{s-1}^{(1+\ve)/(\gamma-1)}$. Using this recursion also on the middle factor and collecting the factors that contain the exponent $h-1$, we obtain that 
\be \wit T_{1} \le M_{d,\al} 2^{\al-d} m_{s-1}^{(1+\ve)/(\gamma-1)+(1+\ve) d/(\al n) -Z} \cdot \sum_{h=1}^n\big( m_{s-1}^{(d/\al-1/(\gamma-1))(1+\ve)/n}\big)^{(h-1)}.
\ee 
The sum on the rhs is a geometric sum with base less than $1$ since $d/\al-1/(\gamma-1)<0$. Its base is $<1/2$ for all $s$ (including $s-1=0$) for  $i$ large enough (otherwise, $2$ should be replaced by another generic constant). Thus the sum on the rhs is at most $2$.
Using the second term from \eqref{def:z}, we see that  $Z\ge (1+\ve)((1+\ve)/(\gamma-1) +d/(n\al))$, thus
\be \label{eq:t1} \wit T_1\le 2^{1+\al-d} M_{d,\al} m_{s-1}^{(1+\ve)(1/(\gamma-1) +d/(n\al))-Z}\le M_1 m_{s-1}^{-\ve/(\gamma-1)}, \ee
with $M_1:=2^{1+\al-d} M_{d,\al}$.
We continue bounding $\wit T_2$ from \eqref{eq:two-terms} analogously. Using the recursion $m_s=m_{s-1}^{(1+\ve)/(\gamma-1)}$ in its middle factor and with  $\wit Z:=z d (1+\ve)$ we write $S_s=m_{s-1}^{\wit Z/(d(\gamma-1))}$.
Then 
\[\ba  \wit T_2 \le M_{d,\la} &2^{d(\gamma-1)} m_{s-1}^{(1+\ve) ( (\tau-1)/(\gamma-1) + d/(n\al) ) -\wit Z}  \\
&\cdot \sum_{h=1}^n  \left(m_{s-1}^{(1+\ve)(d/\al - (\tau-1)/(\gamma -1))/n}\right)^{h-1} \ell\big(m_s^{z\al+1/n -(h-1)/n}/\la^{1-(h-1)/n} \big).\ea\]
Since $\tau-1=\gamma d/\al$, the exponent of $m_{s-1}$ inside the sum is $(1+\ve)d/\al (1-1/(\gamma-1)) < 0$, so, the sum would be a geometric series with parameter $<1$ if the slowly-varying function would not be present, so next we get rid of that.
Let us define 
\[ \wit M_\ell:=\max_{h\le n} \sup_{x\in \R^+} \left\{ x^{-\ve(\gamma-1)/(1+\ve)} \ell(x^{z\al +1/n -(h-1)/n} / \la^{1-(h-1)/n} ) \right\}.\]
Note that $\wit M_\ell<\infty$ by Potter's theorem in \eqref{eq:slowly-1} and since there are only  finitely many values of $h$. Importantly, $\wit M_\ell$ does not depend on $i$.
Using this bound, at the expense of an additional exponent $+\ve$ of $m_{s-1}$ outside the sum, the $\ell(\cdot)$ factor disappears from the geometric sum. This sum is then at most $2$ for $i$ large enough, and we obtain that
\[  \wit T_2 \le M_{d,\la} 2^{d(\gamma-1)+1} \wit M_\ell \cdot m_{s-1}^{(1+\ve) (d/\al) ( \gamma/(\gamma-1) + 1/n ) -\wit Z+\ve}.\]
Note that when $\wit Z \ge 2\ve + (1+\ve) (d/\al) ( \gamma/(\gamma-1) + 1/n )$, the exponent of $m_{s-1}$ is at most $-\ve$. 
Using the first term in \eqref{def:z} and that $\wit Z=d (1+\ve) z$, it is elementary to see that this is indeed the case.  Hence, setting $M_2:=M_{d,\la} 2^{d(\gamma-1)+1} \wit M_\ell$,
\be \label{eq:t2} \wit T_2 \le M_2 \cdot m_{s-1}^{-\ve}.\ee
Returning to \eqref{eq:pa3}, we can now bound $\Pv(F_s^c\mid E_{s-1}\cap F_{s-1})$ using Markov's inequality and the bounds \eqref{eq:t1} and \eqref{eq:t2} on \eqref{eq:two-terms}.
Namely, since $\Pv(F_0^c)=0$,
\[ \sum_{s=0}^k \Pv(F_s^c\mid E_{s-1}\cap F_{s-1}) \le \wit C_{\ve} \sum_{s=1}^k \left(m_{s-1}^{-\ve} + m_{s-1}^{-\ve/(\gamma-1)}\right) \le \wit C_{\ve} m_0^{-\ve}=:\wit p_{\ve, i},\]
with $\wit C_{\ve}:=\max\{M_1, M_2\}$. Finally we advance the induction hypothesis on $\Pv(E_k\cap F_k)\ge 1/2$. Note that
 $\Pv(E_k\cap F_k)\ge \Pv(\cap_{s=0}^{k} (E_s \cap F_s))$. We can then estimate $\Pv(E_k\cap F_k)$  following \eqref{eq:pa3}, which is at least $1-2p_{\ve,i}-\wit p_{\ve, i}\ge 1/2$ when $i$ is sufficiently large, finishing the proof.  
\end{proof}

  \section{`Conservative' proofs}\label{s:gamma12}
In this section we prove the theorems related to the conservative case when $\gamma\in (1,2)$. That is, we prove Theorem \ref{thm:max-displace},  the Conservative Part of Theorem \ref{thm:fpp-profile}, and finally Theorem \ref{thm:dist-cons}. The proofs of these theorems all have the similarity that their upper bound part uses the boxing method described in Section \ref{s:prelim} while their lower bound uses the coupling described in Theorem \ref{thm:thinned-BRW} to a  $\BerBRW$ and  the upper bound on the growth of this $\BerBRW$ established in Proposition \ref{prop:ber-gensizes}.
\begin{proof}[Proof of Theorem \ref{thm:max-displace} subject to Theorem \ref{thm:thinned-BRW}]
 The upper bound in \eqref{eq:max-displacement} as well as \eqref{eq:embed-12} follows directly from  Theorem \ref{thm:thinned-BRW} and Proposition \ref{prop:ber-gensizes}. Namely, by the thinning, the set of vertices graph distance $n$ away from the root in the $\SFPWL$ is a subset of the vertices in generation $n$ in $\BerBRW$.\footnote{Strictly speaking, Theorem \ref{thm:thinned-BRW} couples the $\BerBRW$ to $\SFPWL$ not in breadth-first-search way but according to the edge-weights $L$. However, it is not hard to modify the exploration algorithm and the proof of Theorem \ref{thm:thinned-BRW} to accommodate thinning in breadth-first-search order, and then this statement is true.}
Thus, the maximal displacement in $\Delta B_n^G(0)$ is dominated by that of generation $n$ of the $\BerBRW$. Proposition \ref{prop:ber-gensizes} finishes the proof. 
Recall the greedy boxing and the connectivity between the centers of subboxes in Lemma \ref{lem:layers-complete}. For the lower bound, recall that all centers in $\Gamma_{k}$ are connected to all centers in $\Gamma_{k+1}$ whenever $k\ge n_0(\ve)$. Since nearest-neighbor edges are always present, let $H(n_0(\ve)):=\min_{i\le b_{n_0(\ve)}}d_G(0, c_{n_0(\ve)}^{(i)})$ be the graph distance between $0$ and the centers in $\Gamma_{n_0(\ve)}$, and let $c_{n_0(\ve)}^\star$ be the vertex where this is attained. There is a path of  $k-n_0(\ve)$ edges between $c_{n_0(\ve)}^\star$ and any center $c_k^{(i)}$ for $k\ge n_0(\ve)$. Recall $\|c_k^{(i)}\|\ge D_{k-1}=\exp\{ D C^{k-1}\}$ with $C=(1-\ve)/(\gamma-1)$. So, the  the following lower bound on  $D_{n}^{\max}$ holds: 
\[ D_{n}^{\max}\ {\buildrel a.s.\over  \ge  }\ \exp\left\{ D \Big(\frac{1-\ve}{\gamma-1} \Big)^{n+n_0(\ve)-H(n_0(\ve))-1} \right\}.  \]
Since $n_0(\ve)\ge 0, H(n_0(\ve))\ge 0$, and $D(\ve)$ is from \eqref{eq:C-epsz},
set $Z(\ve):=D  ((\gamma-1)/(1-\ve))^{1+H(n_0(\ve))}\le 2$ for all $\ve<1/4$. An estimate of $Z(\ve)$ on $\ve$ can be obtained using \eqref{eq:n0} and the bound $H(n_0(\ve))\le \exp\{ dD C^{n_0(\ve)}\}$, but we believe that that is far off the truth. 
 \end{proof}
\begin{proof}[Proof of Theorem \ref{thm:fpp-profile}: Conservative part, upper bound]
We shall first show that, for $\ve>0$ arbitrarily small, almost surely, 
\be\label{eq:greedy-conservative} \limsup_{n\to \infty} \frac{M_n(0)}{ \sum_{i=1}^n F_L^{-1} \left(  1/b_k\right)} \le  1,\ee
where $b_k$ is as in \eqref{eq:bn}. Then we show that the denominator is at most a factor $1+\ve$ times the denominator of \eqref{eq:non-explosive-mn}.
To show  \eqref{eq:greedy-conservative}, we analyse the length of the greedy path $\gamma^{\mathrm{gr}}$ constructed in Definition \ref{def:greedy-path}. Recall that the greedy path when at a center in $\Gamma_k$, chooses the minimum edge-weight from $b_{k+1}$ (\eqref{eq:bn}) edges leading to a center $\Gamma_{k+1}$. Hence, exactly as in \eqref{eq:greedy-ray}, almost surely,
\be\label{eq:cons-decomp} M_n(0) \le |\gamma[0, c_{n+n_0(\ve)}^{\text{gr}}]|_L= |\gamma[ 0, c_{n_0(\ve)}^{\text{gr}}]|_L + \sum_{k=n_0}^{n+n_0(\ve)-1} \min\{ L_{k, 1}, L_{k, 2}, \dots, L_{k, b_{k+1}} \}.\ee
We analyse the behavior of the summands on the rhs. Note that
\be\label{eq:summands} \Pv\left( \min_{j\le b_{k+1}}\{ L_{k, j}\} > F_L^{(-1)}\left(1/b_{k}\right)\right)\le \Big(1-F_L\big(F_L^{(-1)}\left(1/b_{k}\right)\big)\Big)^{b_{k+1}} 
 \le \exp\{ -b_{k+1}/b_{k} \}. \ee
Since $b_k$ grows double-exponentially by \eqref{eq:an}, the rhs is summable in $k$. Thus, by the Borel-Cantelli lemma, there is a random $n_1\ge n_0(\ve)$ such that for all $k\ge n_1$, each term in the sum in \eqref{eq:cons-decomp} is at most $F_L^{(-1)}\left(1/b_{k}\right)$.
Hence for $n\ge n_1$, almost surely,
\be\ba \label{eq:cons-decomp-2} M_n(0)&\ {\buildrel a.s. \over \le} \ |\gamma[ 0, c_{n_0(\ve)}^{\text{gr}}]|_L + \sum_{k=n_0(\ve)}^{n_1-1} \min_{j\le b_{k+1}}\{ L_{k, j}\} + 
\sum_{k=n_1}^{n-1} F_L^{(-1)} (1/b_{k})\\
&\ {\buildrel a.s. \over \le}  \sum_{j=1}^{d D_{n_0}+n_1} L_j + \sum_{k=n_1}^{n-1} F_L^{(-1)} (1/b_{k})\ea
\ee
We explain the second line. Observe that $0$ can be connected to any vertex within $\mathrm{Box}_{n_0}$ by a nearest-neighbor path of length at most $d D_{n_0}$, where $D_{n_0}$ is the side-length of $\mathrm{Box}_{n_0}$, see \eqref{eq:rn}. Thus, the first term on the rhs of the first line of \eqref{eq:cons-decomp-2} is at most the sum of $d D_{n_0}$ many i.i.d. copies of $L$. The extra $n_1$ copies of $L$ in the second line can be chosen to be one of the variables within each  minima in the second sum in the first line, hence the inequality holds almost surely.

By the assumption that the integral in \eqref{eq:integral-crit} diverges and that $b_k$ grows double exponentially, by Claim  \ref{lemma:integral}, the sum $\sum_{k=1}^\infty F_L^{(-1)}(1/b_k)$ diverges. 
Hence, for any \emph{fixed} realization of $n_0, n_1$, and the variables $L_j$ in the first term in \eqref{eq:cons-decomp-2}, 
\[ \lim_{n\to \infty} \frac{  \sum_{j=1}^{2 d D_{n_0}+n_1} L_j + \sum_{k=n_1}^{n-1} F_L^{(-1)} (1/b_{k})}{  \sum_{k=1}^{n} F_L^{(-1)} (1/b_{k})}=1\quad a.s.\]
Combining this with \eqref{eq:cons-decomp-2} yields \eqref{eq:greedy-conservative}.
It is left to show that 
\be\label{eq:goal} \lim_{n\to \infty}\frac{\sum_{k=1}^{n} F_L^{(-1)} (1/b_{k})}{\sum_{k=1}^{n} F_L^{(-1)} (\exp\{- (\gamma-1)^{-k}\})} \le 1+\ve',\ee
where we recall that $b_k$, from \eqref{eq:bn}, grows double-exponentially with rate $C=(1-\ve)/(\gamma-1)$ where $\ve>0$ can be chosen arbitrarily small. Note that $b_k,F_L^{(-1)}$ are both monotonically  increasing  and hence $F_L^{(-1)} (1/b_{k})$ is decreasing in $k$.
Note also that the worse lower bound $b_k\ge \exp\{C^k d(D-1)/2\}$ is also valid in \eqref{eq:an}.  
As a result, with $b:=d(D -1)/2$, 
\be\label{eq:sum-1} \sum_{k=1}^{n} F_L^{(-1)} (1/b_{k}) \le F_L^{(-1)}(1/b_1)+\int_{1}^{n} F_L^{(-1)} \left( \exp\{- b C^{x}\} \right)\mathrm dx. \ee
Let us write $\xi:=1/(\gamma-1)$ and change variables so that $\exp\{-bC^x\}=\exp\{-\xi^y\}$.  Thus, \eqref{eq:sum-1} can be bounded from above by
\be\label{eq:sum-2}  F_L^{(-1)}(1/b_1)+\frac{\log \xi}{\log C}\int_{(\log C +\log b)/ \log \xi}^{(n\log C +\log b)/  \log \xi } F_L^{(-1)} \left( \exp\{- \xi^{y}\} \right)\mathrm dy. \ee
Note that by definition, $C=(1-\ve)\xi$, hence small enough $\ve>0$, $\log \xi/\log C \le 1+ 2\ve/\log \xi$.
Since $\log C/\log \xi <1$,  the integration boundary in \eqref{eq:sum-2} is $\le n$ for all $n> \log b/\log (1-\ve)$.
Thus,
\be\label{eq:sum-3}\sum_{k=1}^{n} F_L^{(-1)} (1/b_{k}) \le F_L^{(-1)}(1/b_1)+ \left(1+\frac{2\ve}{\log \xi}\right) \int_1^{n}F_L^{(-1)} \left( \exp\{ -\xi^{y}\} \right) \mathrm dy. \ee
Finally, we turn the integral  back to the sum in the denominator in \eqref{eq:goal}.
Clearly 
\be\label{eq:sum-5}  \int_1^n F_L^{-1} (\exp\{ -\xi^y\})\mathrm dy \le \sum_{k=1}^{n} F_L^{(-1)} (\exp\{ -\xi^k\}) 
\le  F_L^{(-1)}(\exp\{ -\xi\} )+\int_1^n F_L^{-1} (\exp\{ -\xi^y\})\mathrm dy \ee
Since both the integral and the sum diverge, the ratio of the sum in the middle and the integral tends to $1$. Combining \eqref{eq:sum-3} with this establishes \eqref{eq:goal} with $\ve'= 2\ve/|\log(\gamma-1)|$. 
Since $\ve'$ is arbitrarily small, combining \eqref{eq:greedy-conservative} and \eqref{eq:goal} results in  
\be\label{eq:limsup-mn} \limsup_{n\to \infty} \frac{M_n(0)}{\sum_{k=1}^{n} F_L^{(-1)} (\exp\{ -1/(\gamma-1)^k\})  }\le 1,\ee
finishing the proof of the upper bound. \end{proof}
\begin{proof}[Proof of Theorem \ref{thm:fpp-profile}, Conservative Part, lower bound, subject to Theorem \ref{thm:thinned-BRW}]
Theorem  \ref{thm:thinned-BRW} establishes a coupling between the exploration on $\SFPWL$ and the $\BerBRW$, with the important feature that 
 from each vertex in the explored cluster, the shortest paths to the origin in the \emph{thinned} $\BerBRW$ has the same distribution as the shortest path to the origin in the $\SFPWL$. 
 Recall $M_n(0)$ from \eqref{eq:mn-def}, and let us denote by $M_n^{\mathrm{B}}$ the time to reach generation $n$ in $\BerBRW$.
Under the coupling in Theorem \ref{thm:thinned-BRW}, $M_n^{\mathrm{B}} \le M_n(0)$ almost surely. Trivially, $M_n^{\mathrm{B}}$ is almost surely larger than the sum of the minimum edge-weights in each generation of the $\BerBRW$. Thus we obtain
\be\label{eq:mn-lower-1} M_n(0)\    \ {\buildrel a.s. \over \ge}  \ M_n^{\mathrm{B}}\ {\buildrel a.s. \over \ge} \ \sum_{k=1}^n \min\{L_{k,1}, \dots, L_{k, Z_k^{\text{B}}} \} \  {\buildrel a.s.\over \ge} \ \sum_{k=1}^n \min\{L_{k,1}, \dots, L_{k, c_k(\ve, Y)}\}\ee
where $Z_k^{\mathrm{B}}$ stands for the size of generation $k$  in $\BerBRW$, and we have used \eqref{eq:ck2} from Proposition \ref{prop:ber-gensizes} for an a.s. upper bound on $Z_k^{\text{B}}$, and the monotonicity of the minimum.
Next we analyse the expression on the rhs. We abbreviate $c_k:=c_k(\ve,Y)$. Similarly to \eqref{eq:summands},
\be\label{eq:min-too-small} \Pv\left(\min\{L_{k,1}, \dots, L_{k, c_k}\} < F_L^{(-1)} (1/c_{k+1}) \right)= 1- \left(1- 1/c_{k+1}\right)^{c_k}\le c_k/c_{k+1}.\ee
Since $c_k$ grows double-exponentially, the rhs is summable in $k$. Thus, by the Borel-Cantelli Lemma, there is a random $k_1$ s.t.\ for all $k\ge k_1$, the $k$th term on the rhs of \eqref{eq:mn-lower-1} is at least $F_L^{(-1)} (1/c_{k+1})$.
Hence, for $n>k_1$,
\be\label{eq:mn-lower-2}  M_n(0) \ { \buildrel a.s. \over \ge }\ \sum_{k=k_1}^n  F_L^{(-1)}(1/ c_{k+1}).\ee
Next we relate the rhs to the denominator in \eqref{eq:non-explosive-mn} in Theorem \ref{thm:fpp-profile}  using the same method as in the proof of the upper bound. We can lower bound the sum by an integral as in \eqref{eq:sum-5} and then change variables,  now using lower bounds: Here, with $C:=(1+\ve)\xi$, and $b:=Y(1+\ve) d/\al$, 
\be \label{eq:change-var-2}  \log \xi/\log C \ge 1 - \ve/\log \xi,\ee
and in this case, since $\log C/\log \xi>1$, the upper integration boundary after change of variables in \eqref{eq:sum-2} equals $n \cdot \log C/\log \xi + \log b/\log \xi$, which is \emph{larger} than $n$ for all sufficiently large $n$. 
See the proof of \eqref{eq:goal} between \eqref{eq:sum-1}-\eqref{eq:sum-5} for more details.  Ultimately, when $\ve>0$ in \eqref{eq:ck} is arbitrarily small, then 
\be\label{eq:ratio-2} \lim_{n\to \infty} \frac{\sum_{k=k_1}^n  F_L^{(-1)}(1/ c_{k+1})}{\sum_{k=1}^n  F_L^{(-1)}(\exp\{ - (\gamma-1)^{-k}\})}\ge 1-\ve'',  \ee
with $\ve''$ also arbitrarily small. Combining \eqref{eq:mn-lower-2} with \eqref{eq:ratio-2}, we obtain that 
\[ \liminf_{n\to \infty} \frac{M_n(0)}{\sum_{k=1}^n  F_L^{(-1)}(\exp\{ - 1/(\gamma-1)^k\})} =1.\]
This finishes the proof of the lower bound in \eqref{eq:non-explosive-mn}, and, with the upper bound in \eqref{eq:limsup-mn},  \eqref{eq:non-explosive-mn} is now proved. Finally, \eqref{no-integral-nonexplosion} follows by noting that this lower bound tends to infinity as $n\to \infty$, due to the equi-convergence of the sum in the denominator and the integral in \eqref{eq:integral-crit}.
\end{proof}
\begin{proof}[Proof of Theorem \ref{thm:dist-cons} subject to Theorem \ref{thm:thinned-BRW}]
We start by showing the upper bound, that uses the boxing technique again. Set $x:=m\underline e$. Recall the deterministic $n(x)$ from \eqref{def:nx} and the random  $n_0(\ve), n_x(\ve)$ from Lemma \ref{lem:layers-complete} and Corollary \ref{corr:merging}, respectively.  
In this case, we use the \emph{merging greedy paths} $\gamma_0^{\mathrm{gr}}, \gamma_x^{\mathrm{gr}}$ defined in \eqref{eq:merging-greedy} and \eqref{eq:merging-greedy-2} after Definition \ref{def:greedy-path}. 
On the event that $\{n_0(\ve), n_x(\ve)\le n(x)\}$, by Definition \ref{def:greedy-path}, for $q\in \{0,x\}$ $\gamma_q^{\mathrm{gr}}$ leaves $q$ by using the shortest path between $q$ and the centers of boxes $(c_{n_q(\ve)}^{\sss{(i)}})_{i\le b_{n_q(\ve)}}$ (see between \eqref{eq:rn}-\eqref{eq:C-epsz} for notation) that stays within $\mathrm{Box}_{n_q(\ve)}^{\sss{(q)}}$.  The center where the shortest $L$-distance path is attained is denoted by $c_{n_q(\ve)}^{\mathrm{gr}}(q)$.  From $c_{n_q(\ve)}^{\mathrm{gr}}(q), q\in\{0,x\}$, respectively, the two greedy paths follow the minimal-edge-weight towards centers of boxes in the next annulus until they reach annuli $\Gamma_{n(x)+1}^{(q)}$ at respective centers of boxes  $c_{n(x)+1}^{\mathrm{gr}}(q)$. Finally, the two paths merge  by connecting both of these last two vertices via a vertex $c_{n(x)+2}^{\mathrm{gr}}$ within $\wit{\mathrm{Box}}_{n(x)+1}$ as described in \eqref{eq:merging-greedy-2}. Thus, the path 
$\gamma_0^{\mathrm{gr}}[0, c_{n(x)+2}^{\mathrm{gr}}] \cup \gamma_x^{\mathrm{gr}}[0, c_{n(x)+2}^{\mathrm{gr}}]$ connects $0$ and $x$ and its $L$-length provides an upper bound on $\mathrm d_L(0,x)$. This is what we analyse now.

Recall  the tail behavior of  $n_q(\ve)$, for $q\in\{0,x\}$ from \eqref{eq:n0} in Lemma \ref{lem:layers-complete} and from \eqref{eq:nx-ve} in Corollary \ref{corr:merging}:
\be\label{eq:n0-tail1} \Pv(n_q(\ve) \ge K ) \le  2 \exp\big\{ -c_1 \exp\{ \ve c_2 C(\ve)^{K} \}  \big\}. \ee
Then, for $q\in\{0,x\}$,
\be\label{eq:min-333} | \gamma_q^{\mathrm{gr}}[q, c_{n_q(\ve)+1}^{\mathrm{gr}}(q)] |_L \le  |\gamma[ q, c_{n_q(\ve) }^{\text{gr}}(q)]|_L + \sum_{k=1}^{n(x)} \min\{ L_{k, 1}^{\sss{(q)}}, L_{k, 2}^{\sss{(q)}}, \dots, L_{k, b_{k+1}}^{\sss{(q)}} \} \ee
where we fill the sum on the rhs with newly drawn  i.i.d.\ $L_{k,j}^{\sss{(q)}}$ for all $k\le n_q(x)$. Let us define the random variables $k_x(q)\le n(x), q\in\{0,x\}$ as the last index in the sum on the rhs of \eqref{eq:min-333}  that is larger  than $F_L^{(-1)}(1/b_k)$.
That is,
\be\label{eq:kxq} k_x(q):= \max\{k: k \le n(x):  \min_{j\le b_{k+1}}\{ L_{k,j}^{\sss{(q)}}  \} \ge F_L^{(-1)}(1/b_k) \}.\ee
Such a $k_x(q)$ exists by \eqref{eq:summands} and the Borel-Cantelli lemma. 
For $k\le k_x(q)$ we can use one of the $L_{k,j}^{\sss{(q)}}$ inside each mimina, so that combining this with \eqref{eq:min-333}, we have the upper bound 
\be \label{eq:dl-as-1} d_L(0,x) \  {\buildrel a.s. \over \le} \sum_{j=1}^{d D_{n_0(\ve)}+k_x(0)+1}  L_j^{(0)}  +  \sum_{j=1}^{d D_{n_x(\ve)} + k_x(x)+1}   L_j^{(x)}   
 +2\sum_{k=1}^{n(x)} F_L^{(-1)} (1/b_{k}) ,  \ee
where the $L_j^{(q)}$ are independent collections of i.i.d.\ variables, and we took the worst possible case for connecting $q$ to $c_{n_q(\ve)}^{\mathrm{gr}}(q)$  via nearest-neighbor edges.
 By adding one edge-weight to each of the first two terms on the rhs of \eqref{eq:dl-as-1}, we also take into account the last two edges connecting $c_{n(x)+1}^{\mathrm{gr}}(0)$ to $c_{n(x)+1}^{\mathrm{gr}}(x)$.
We would like to show that, for $q\in \{0,x\}$, 
\be\label{eq:first-neglig} \Pv\Bigg(\sum_{j=1}^{d D_{n_q(\ve)} +k_x(q)+1} L_j^{(q)}  \ge \delta/2 \sum_{k=1}^{n(x)} F_L^{(-1)} (1/b_{k}) \Bigg)\to 1\ee 
as $m \to \infty$ and $x=\lfloor m\underline e \rfloor$. To show this, we argue as follows. 
As $m \to \infty$, the sequences $n_0(\ve), k_x(0)$ converge to their unrestricted limits (dropping the restriction $k\le n(x)$ in \eqref{eq:kxq}), thus the lhs within the probability sign converges to a proper random variable. Then, \eqref{eq:first-neglig} for $q=0$ follows directly from this since the rhs within the probability sign tends to $0$.  Further, $n_x(\ve), k_x(x)$ come from boxing around vertices $x=\lfloor m \underline e\rfloor$, which  is different for each $x$, thus, these are tight sequences of random variables, with respective uniform tail bounds given by \eqref{eq:k1-tail} and \eqref{eq:n0-tail1}. The tightness of the sequences $n_x(\ve), k_x(x)$ implies the tightness of $dD_{n_x(\ve)}+ k_x(x)+1$ and thus the tightness of the sum on the lhs within the probability sign of \eqref{eq:first-neglig}. Since the rhs tends to infinity as $n(x)\to \infty$, \eqref{eq:first-neglig} follows  by the definition of tightness.
Combining \eqref{eq:first-neglig} with \eqref{eq:dl-as-1}, for any $\de>0$, as $m\to \infty$ and with $x=\lfloor m \underline e \rfloor$, 
\be\label{eq:conv-inp-cons} \Pv\Bigg( d_L(0,x)\  {\buildrel \over \le} \ (1+\de)  2\sum_{k=1}^{n(x)} F_L^{(-1)} (1/b_{k})\Bigg) \to 1.\ee
Next, we change this to \eqref{eq:dist-cons-1111}.
Recall from \eqref{eq:an} that $b_k\le \exp\{ b C^n\}$ with $C=(1-\ve)/(\gamma-1)$ and for $b=d(D-1)>0$. 
Following \eqref{eq:sum-1} and \eqref{eq:sum-2}, as well as  $\log \xi/\log C<1+2\ve/\log\xi$, we obtain
\be\label{eq:sum-44}  \sum_{k=1}^{n(x)} F_L^{(-1)} (1/b_{k}) \le F_L^{(-1)} (1/b_{1}) +\Big(1+\frac{2\ve}{\log \xi}\Big) \int_{(\log C +\log b)/ \log \xi}^{(n(x)\log C +\log b)/  \log \xi } F_L^{(-1)} \left( \exp\{- \xi^{y}\} \right)\mathrm dy. \ee
Finally, using the bound on the  integral in \eqref{eq:sum-5}, 
\be\label{eq:sum-45} \sum_{k=1}^{n(x)} F_L^{(-1)} (1/b_{k}) \le  F_L^{(-1)} (1/b_{1})+ \Big(1+\frac{2\ve}{\log \xi}\Big) \sum_{k=1}^{\lceil (n(x)\log C +\log b)/  \log \xi\rceil  } F_L^{(-1)}(\exp\{ -\xi^k\}),  \ee
where $\lceil z\rceil=\min\{y\in \Z: y\ge z\}$. 
Using that $\lfloor z\rfloor\le z, \lceil z\rceil \le z+1$, as well as  \eqref{def:nx}, the summation boundary on the rhs is, for $x=\lfloor m\underline e \rfloor$, at most 
\be \label{eq:sum-46} \lceil (n(x)\log C +\log b)/  \log \xi\rceil   \le \log \log m/\log \xi + (- \log D  + \log (d(D-1))/\log \xi +1.\ee
Since the summands tend to zero, as $m \to \infty$, a constant deviation from   $\log \log m/\log \xi$ in terms of the number of summands is negligible in the limit. Thus, combining  \eqref{eq:sum-45} with \eqref{eq:sum-46} and setting $1+\de':=(1+\de)(1+ 2\ve/\log \xi)$, yields that 
\be\label{eq:upper-dist-cons}
(1+\de) \sum_{k=1}^{n(x)} F_L^{(-1)} (1/b_{k}) \le (1+\de') \sum_{k=1}^{\lceil \log \log m/\log \xi\rceil } F_L^{(-1)} (\exp\{-\xi^k\})  \ee
for arbitrary small $\delta'>0$, for all $m$ sufficiently large. This inequality, combined with \eqref{eq:conv-inp-cons} finishes the proof of the upper bound of Theorem \ref{thm:dist-cons}.

We turn to prove the lower bound.
We use the upper bounding $\BerBRW$ from Theorem \ref{thm:thinned-BRW} and  Proposition \ref{prop:ber-gensizes} on its maximal displacement and generation sizes. 
Let us first set the vertex vertex-weights $(W_z)_{z\in \Z^d}$, and then use two explorations on $\SFPWL$, one started from $0$ and one from $x$, with two dominating  $\BerBRW$s, denoted by $\mathrm{BRW}^{(0)}, \mathrm{BRW}^{(x)}$, \emph{independent} of each other  conditioned on $(W_z)_{z\in \Z^d}$. We add $(0)$ and $(x)$ either as a superscript or as an argument to quantities related to the two explorations started from a root individual located at $0$ and $x$, respectively.
Edges are present independently conditioned on the vertex-weights, so the two explorations are also independent on $\SFPWL$ as long as we guarantee that they \emph{stay in disjoint boxes}.  For $q\in\{0,x\}$, let $B_n^{G,L}(q), \Delta B_n^{G,L}(q)$ denote the vertices $y\in \Z^d$ from which the $L$-shortest path to $q$  contains at most/precisely $n$ edges in $\SFPWL$. Suppose for some $N_x(0), N_x(x)\ge  0$ we can guarantee that  $\Delta B_{N_x(0)}^{G,L}(0)$ is disjoint of  $\Delta B_{N_x(x)}^{G,L}(x)$. Then, any shortest path from $0$ to $x$ must intersect these sets and thus 
\be\label{eq:intersect} \mathrm{d}_L(0,x)\  {\buildrel a.s. \over \ge}\ \min\{ \mathrm{d}_L(0, y): y \in \Delta B_{N_x(0)}^{G,L}(0)\} + \min\{ \mathrm{d}_L(0, z): z \in \Delta B_{N_x(x)}^{G,L}(x)\}. \ee
Let us consider the  two disjoint boxes \[ \mathrm{B'}(0):=[-\|x \|/2, \| x\|/2]^d, \qquad \mathrm B'(x):=[x-\|x \|/2, x+ \|x\| /2]^d.\]   We will find below   $N_x(0), N_x(x)\ge  0$ that satisfies 
\[ \bigcup_{i\le N_x(0)}\CG_i^{\mathrm B}(0) \subseteq \mathrm{B'}(0)\quad  \text{and} \quad \bigcup_{i\le N_x(x) }\CG_i^{\mathrm B}(x)  \subseteq \mathrm{B'}(x). \] 
The edge sets within $\mathrm{B'}(0), \mathrm{B'}(x)$ in $\SFPWL$ are independent conditioned on $(W_z)_{z\in \Z^d}$, hence, it is possible to couple $B_{N_x(q)}^{G,L}(q)$ to $\mathrm{BRW}^{(q)}$ stopped at generation $N_x(q)$. Heuristically this is true since the difference between an exploration on $\SFPWL$ and a $\BerBRW$ is that the $\BerBRW$ is allowed to have new offspring upon returning to an already visited vertex, while the exploration is not. In other words,
Theorem \ref{thm:thinned-BRW}  can be extended to hold for both explorations jointly in this case. 
Therefore we obtain that, for all $i\le N_x(q), q\in \{0,x\}$,
\be \label{eq:delta-subset}\Delta B_i^{G,L}(q) \subseteq \CG_i^{\mathrm B}(q).\ee
The minimum decreases if we increase the set, thus, by \eqref{eq:intersect} and \eqref{eq:delta-subset},
\be\ba \label{eq:dl-lower} \mathrm{d}_L(0,x)\  &{\buildrel a.s. \over \ge}\ \min\{ d_L(0, y): y \in \CG_{N_x(0)}^{\mathrm B}(0)\} + \min\{ d_L(0, z): z \in \CG_{N_x(x)}^{\mathrm B}(0)(x)\} \\&= M_{N_x(0)}^{\mathrm B}(0) + M_{N_x(x)}^{\mathrm B}(x), \ea \ee
where the two variables on the rhs are independent.   
We now modify \eqref{eq:ck2} in Proposition \ref{prop:ber-gensizes} to hold for the two explorations to determine $N_x(0), N_x(x)$. 
The two explorations become dependent after leaving their respective boxes $\mathrm{B}'(0), \mathrm{B}'(x)$ and thus the definition of $Y(\ve)$ as the one after \eqref{eq:hi-estimate} needs to be modified to maintain independence. 
In order to do so, similarly as in \eqref{eq:wit-ek}, for $q\in\{0,x\}$,  we define the event 
\be\label{eq:ekq} \wit E_k^{\sss{(q)}}(\ve, i):= \{ Z_k^{\mathrm{B}}(q) \le c_k(\ve, i)\}\cap \Big\{ \CG_k^{\mathrm{B}}(q) \cap \left( B^{2}_{S_k(\ve, i)}(q)\right)^c  =\emptyset \Big\},\ee Recall $S_k(\ve,i)$ from \eqref{eq:ck} and define the deterministic number
\be\label{eq:nxe} N_x(\ve, i):=\max\{k: S_k(\ve,i)\le \|x\|/2\}=\left \lfloor \frac{\log\log(\| x\|/2)- \log (iz)}{\log ((1+\ve)/(\gamma-1))}\right\rfloor. \ee 
Then, we set, for $q\in\{0,x\}$, 
\[ \wit H_i^{\sss{(q)}} := \bigcap_{k=0}^{N_x(\ve, i)} \wit E_k^{\sss{(q)}}(\ve, i).\] 
Heuristically speaking, $\wit H_i^{\sss{(q)}}$ is the event that the $\mathrm{BRW}^{(q)}$ grows double exponentially with rate $(1+\ve)/(\gamma-1)$ and prefactor at most  $i$ \emph{before it leaves the box $B'(q)$}. Since $H_i\subset \wit H_i^{\sss{(q)}}$, \eqref{eq:hi-estimate} remains valid for $(\wit H_i^{\sss{(q)}}), q\in\{0,x\}$ as well, namely,
\be\label{eq:hi-estimate-2}\sum_{i=1}^\infty \Pv\big((\wit H_i^{\sss{(q)}})^c\big)\le \sum_{i=1}^\infty\Big( 1-\Pv\Big( \bigcap_{k=0}^{N_{x}(\ve, i)} E_k^{\sss{(q)}}(\ve,i) \cap F_k^{\sss{(q)}}(\ve, i)\Big)\Big) \le \sum_{i=1}^\infty (2p_{\ve, i} + \wit p_{\ve, i}) <\infty. \ee 
By the Borel Cantelli lemma only finitely many $(\wit H_i^{\sss{(q)}})^c$ events occur and hence 
\be\label{eq:witye} \wit Y_q(\ve):=\min\{k:  \forall i\ge k, \wit H_i^{\sss{(q)}}\} \ee 
 is a.s. finite and $\wit Y_0(\ve), \wit Y_x(\ve)$ are independent, since they are determined on a disjoint vertex and edge-set of the graph. The meaning of $\wit Y_q(\ve)$ is again that $\mathrm{BRW}^{(q)}$ grows double exponentially with rate $(1+\ve)/(\gamma-1)$ and prefactor $\wit Y_q(\ve)$ \emph{before it leaves the box $B'(q)$}.
Finally, the tail estimate \eqref{eq:Y-tail} remains valid for $\wit Y_q(\ve)$ as well by \eqref{eq:hi-estimate-2}. 
 With $\wit Y_q(\ve)$ at hand, the last generation where $\mathrm{BRW}^{(q)}$ is still inside box $B'(q)$ is precisely $N_x(\ve, \wit Y_q(\ve))=:N_x(q)$, with $N_x(0), N_x(x)$ \emph{independent}. We can use these in \eqref{eq:dl-lower}, and also that $Z_k^{\mathrm{B}}(q)\le c_k(\ve, \wit Y_q(\ve))$ by \eqref{eq:ekq} and the definition of $\wit Y_q(\ve)$. Since $M_n^{\mathrm{B}}(q)$ is a.s. larger than the sum of the minimum edge-weights in each generation (see the inequality \eqref{eq:mn-lower-1}), \eqref{eq:dl-lower} can be further bounded from below by
\be \label{eq:dl-lower-2} \mathrm{d}_L(0,x)\  { \buildrel a.s. \over \ge}\sum_{q=0,x} \sum_{k=1}^{N_x(q)} \min\{ L_{k,1}^{\sss{(q)}} , \dots, L_{k, c_k(\ve, \wit Y_q(\ve))}^{\sss{(q)}} \}. \ee
By \eqref{eq:min-too-small}, the $k$th term on the rhs is at least $F_L^{(-1)}(1/c_k(\ve, \wit Y_q(\ve)))$ with probability at least $1-c_{k}(\ve, \wit Y_q(\ve)))/c_{k+1}(\ve, \wit Y_q(\ve)))$.
Using  \eqref{eq:ck}, the error probabilities are summable and thus defining 
\be\label{eq:kxq-def} \wit k_x(q):=\max\Big\{k: k\le N_x(q) \text{ and } \min\{ L_{k,1}^{\sss{(q)}} , \dots, L_{k, c_k(\ve, \wit Y_q(\ve))}^{\sss{(q)}} \} \le F_L^{(-1)}(1/c_k(\ve, \wit Y_q(\ve)))  \Big\},\ee
with  $\wit k_x(0), \wit k_x(x)$ being independent. The tail of $\wit k_x(q)$ can be estimated using \eqref{eq:ck} as follows:
\be\label{eq:kxx-tail} \Pv\left(\wit k_x(q) \in [K, N_x(q)]\mid\! N_x(q), \wit Y_q(\ve)\right) \le \sum_{k\ge K} \frac{c_{k}(\ve, \wit Y_q(\ve)))}{c_{k+1}(\ve, \wit Y_q(\ve)))} \le  2 \exp\Big\{\!-  \wit Y_q(\ve) \wit c \Big(\frac{1+\ve}{\gamma-1}\Big)^{K}  \Big\},\ee
with $\wit c=(1+\ve)(2+\ve-\gamma) d / (\al (\gamma -1))$. Continuing to bound \eqref{eq:dl-lower-2}, we arrive at
\be\label{eq:dl-lower-3} d_L(0,x) \ {\buildrel a.s. \over \ge} \   \sum_{k=\wit k_x(0)}^{N_x(0)}  F_L^{(-1)}(1/ c_{k+1}(\ve, \wit Y_0(\ve))) + \sum_{k=\wit k_x(x) }^{N_x(x)}  F_L^{(-1)}(1/ c_{k+1}(\ve, \wit Y_x(\ve))). \ee
To proceed we use a change of variables.
We estimate the sum using the lower bound in \eqref{eq:sum-5}, with $C:=(1+\ve)/(\gamma-1)$, $b:=\wit Y_q(\ve) (1+\ve) (d/\al) C$, 
\be\ba\label{eq:sum-estimate-2}  \sum_{k=\wit k_x(q)}^{N_x(q)}  F_L^{(-1)}(1/ c_{k+1}(\ve, \wit Y_q(\ve)))& \ge \int_{\wit k_x(q)}^{N_x(q)} F_L^{(-1)}(\exp\{ - b C^x\})\mathrm d x\\
&=\frac{\log \xi}{\log C} \int_{(\wit k_x(q)\log C + \log b)/\log \xi}^{(N_x(q) \log C + \log b)/\log \xi} F_L^{(-1)}(\exp\{ -\xi^x\})\mathrm d x\\
&\ge \Big(1-\frac{\ve}{\log \xi}\Big) \sum_{\lceil (\wit k_x(q) \log C+ \log b)/\log \xi\rceil +1}^{\lfloor(N_x(q) \log C+\log b)/\log \xi\rfloor }  F_L^{(-1)}(\exp\{-\xi^k\}).   \ea  \ee
where in the last step we used \eqref{eq:change-var-2} and the upper bound in \eqref{eq:sum-5} to relate the integral to a sum again. 
We investigate the upper summation boundary on the rhs. Note that $N_x(q)=N_x(\ve, \wit Y_q(\ve))$ as defined in \eqref{eq:nxe}, \eqref{eq:witye} and $b=\wit Y_q(\ve) (1+\ve) (d/\al) C$. Using that $\lfloor y \rfloor \ge y-1$, the upper summation boundary in the last row of \eqref{eq:sum-estimate-2} can be bounded from below as follows, for $x=\lfloor m\underline e\rfloor$
\be\label{eq:upper-boundary}\ba  \left \lfloor \frac{N_x(q) \log C+ \log b}{\log \xi}\right \rfloor&\ge \frac{\log \log (\| x\|/2)  - \log (\wit Y_q(\ve) z C) + \log \big(\wit Y_q(\ve) (1+\ve) (d/\al) C \big) }{\log \xi}-1\\
&=\frac{\log \log ( m/2)}{\log \xi} - \frac{ \log \big(\xi(1+\ve)d/( \al z)\big) }{\log \xi}. \ea\ee
Importantly, the random variable $\wit Y_q(\ve)$ \emph{cancels} and the obtained value is a deterministic constant away from $\log \log \| x\|/\log \xi$. 
Next we investigate the lower summation boundary in \eqref{eq:sum-estimate-2}.
Let us introduce a partial sum and  for any fixed $\de>0$
\be Q(t):=\sum_{k=1}^{\lfloor t\rfloor} F_L^{(-1)}(\exp\{-\xi^k\}), \qquad K_\delta(t):=\max\{k: Q(k)< Q(t) \delta/4\}.\ee
Since $Q(t)\to \infty$ as $t\to \infty$ by the assumed divergence of $\mathbf{I}(L)$ in \eqref{eq:integral-crit}, $K_\delta(t) \to \infty$ as $t\to \infty$.
Combining \eqref{eq:dl-lower-3} with \eqref{eq:sum-estimate-2} and \eqref{eq:upper-boundary}, we have shown that, with $\wit C:=\log (\xi(1+\ve) d/(\al z))/\log \xi$,
\be\label{eq:dl-lower-5} \mathrm{d}_L(0,\lfloor m\underline e\rfloor) \ge 2 (1-\frac{\ve}{\log \xi}) Q\Big(\log\log (m/2)/\log \xi - \wit C\Big)  \\ 
 - \sum_{q\in\{0,x\}}Q\Big(\frac{\wit k_{x}(q) \log C + \log b}{\log \xi}+2\Big).
 \ee
By choosing $\ve$ small enough and $m$ sufficiently large, we can obtain that 
\[ (1-\frac{\ve}{\log \xi}) Q\Big(\log\log (m/2)/\log \xi - \wit C\Big) \ge (1-\de/2) Q(\log \log m/\log \xi).  \]
For the convergence in probability, we would like to show that the last term in \eqref{eq:dl-lower-5} is at least $- Q(\log \log m/\log \xi)\delta/2$ with probability tending to $1$. By the definition of $K_\delta(t)$, the complement of this event is contained in 
\be\label{eq:ptoas} \bigcup_{q\in\{0,x\}}\Big\{\frac{\wit k_{x}(q) \log C + \log b}{\log \xi}+2\ge K_\de(\log \log m/\log \xi) \Big\}.\ee
Thus, recalling that $b=\wit Y_q(\ve)(1+\ve)(d/\al) C$ from before \eqref{eq:sum-estimate-2}, we bound the probability of each of these two events 
by a union bound as follows:
\be\label{eq:as-vs-p} \ba \Pv&\left( \frac{\wit k_{x}(q) \log C + \log b}{\log \xi}+2\ge K_\de(\log \log m/\log \xi) \right) \\
  &\quad \quad \quad \quad \le 
  \Pv\left( \log \big(\wit Y_q(\ve) (1+\ve)(d/\al) C\big) \ge (K_\de(\log \log m/\log \xi)-2) \log \xi/2 \right) \\
  &\quad \quad \quad \quad \quad   +  \Pv\left(\wit k_{x}(q) \ge (K_\de(\log \log m/\log \xi)-2) \log \xi/(2\log C)   \right) 
\ea  
\ee
By  \eqref{eq:kxx-tail} and \eqref{eq:Y-tail}, and since $K_\delta (\log \log m/\log \xi)$ tends to infinity with $m$,  both probabilities on the rhs tend to zero as $m\to \infty$. 
This finishes the proof of the lower bound.

We comment on why this proof could not be strengthened to showing a.s. convergence:  $\wit k_x(x)$ and $\wit Y_x(\ve)$ come from explorations where the root $x=\lfloor m \underline e\rfloor$  is different for each $x$, thus, these are tight sequences of r.v.s, with respective uniform tail bounds given by \eqref{eq:kxx-tail} and \eqref{eq:Y-tail}. 
Thus in principle the proof could be strengthened to a.s.\ convergence if the probability of the event in \eqref{eq:ptoas} were summable, by a Borel-Cantelli type argument. We argue why this is not the case. The lhs of \eqref{eq:as-vs-p} is summable if $\log b$ is comparable to $K_\delta(\log \log m/\log \xi)$ only finitely often. For this, using \eqref{eq:Y-tail}, one needs that 
\be\label{eq:y-summable} \exp\{- \wit C \e^{K_\delta (\log \log m/\log\xi)\log \xi } \}\ee
is summable in $m$, for some $\wit C$. Since the terms in $Q(t)$ are strictly less than $1$ and monotonously decreasing, $K_\delta (t) \le  t \delta /4$, implying that  the expression in \eqref{eq:y-summable} is never summable in $m$. Heuristically speaking, large values do occur frequently enough in the sequence $(\wit Y_{\lfloor m \underline e \rfloor}(\ve))_{m\ge 1}$, and this means that the lower summation boundary in \eqref{eq:sum-estimate-2} starts from a significantly higher value than $1$. 
We do believe that a.s. convergence is actually never possible. However, $\wit Y_x(\ve)$ is \emph{not} independent for different values of $x$, so the second Borel-Cantelli lemma here cannot be used.
\end{proof}

\section{The dominating branching random walks}\label{s:explore-proof}
In this section we describe the three process coupling mentioned in Theorem \ref{thm:thinned-BRW} and prove Theorem \ref{thm:thinned-BRW}. The coupling is developed by coupling the \emph{exploration process} on the three graphs together, that we describe now.  
\subsection{The exploration algorithm.}\label{s:explore}
Our exploration algorithm runs on $\SFPWL,$ $\PoiBRW$, and $\BerBRW$ in Section \ref{s:explore-BP} \emph{at the same time}, providing a three-process coupling of 
the exploration algorithm of $\CB_t^{L, \mathrm{S}}(0)\subseteq \CB_t^{L}(0)$ in $\SFPWL$ to $\CB_t^{L, \mathrm{B}}(0), \CB_t^{L, \mathrm{P}}(0)$ in $\BerBRW, \PoiBRW$, respectively. 

We would like to emphasise the following: 
We describe the coupling of the exploration by describing the exploration algorithm on $\PoiBRW$, and applying two consecutive thinning procedures that yield the exploration on $\BerBRW$ and on $\SFPWL$, respectively. When we thin individuals corresponding to multiple edges and their descendants in the $\PoiBRW$, we obtain $\BerBRW$, while a global thinning - we thin all those individuals and their descendants who are at already visited spatial locations - yields the exploration on $\SFPWL$. Quantities related to the exploration on $\PoiBRW$, $\BerBRW$, $\SFPWL$ get a superscript (or subscript) $\mathrm{P, B, S}$, respectively.

Recall that we write  $\tau_n$ for the time to reach the $n$th new vertex from the origin. The exploration algorithm runs in \emph{discrete stages} $n=0,1, \dots$, where a stage corresponds to exploring one more individual (vertex) in $\PoiBRW$. We keep track of the  true `time' $t$ as well, that is, stage $n$ of the exploration corresponds to time  ($L$-distance from the origin) $\tau_n^{\mathrm{P}}$. At time $\tau_n^{\mathrm{S}}$, the first passage exploration (FPE) has discovered all the individuals that are reachable from the origin on a path with $L$-distance at most $\tau_n^{\mathrm{S}}$, that is, it has discovered the individuals in $B^{L, \mathrm{S}}_{\tau_n^{\mathrm{S}}}(0)$ together with their vertex-weights and the shortest $L$-weighted path leading to them.     
We use the following lists during the exploration:
\begin{enumerate}
\item $\mathcal{W}=\big\{ (W_y)_{y\in \Z^d}\}$, containing the environment.
\item 
$\CE_n=\{ (\emptyset, 0, W_0, \ind^{\mathrm{B}}_E(0), \ind^{\mathrm{S}}_E(0) ), (j_1, M_{j_1}, W_{M_{j_1}}, \ind^{\mathrm{B}}_E(j_1), \ind^{\mathrm{S}}_E(j_1))$, \\$\dots, (j_n, M_{j_n}, W_{M_{j_n}}, \ind^{\mathrm{B}}_E(j_n), \ind^{\mathrm{S}}_E(j_n))\}$, 
with $n+1$ sublists denoting the \emph{name}, \emph{location}, and \emph{weight} of the first $n+1$ explored individuals in $\PoiBRW$, as well as their indicators $\ind^{\mathrm{B}}_E(\cdot), \ind^{\mathrm{S}}_E(\cdot)$, whether they are explored in (and thus belong to) $\BerBRW$ and $\SFPWL$ as well. We set $j_0:=\emptyset$ the root individual.
\item $\CU_n^{\mathrm{q}}=\big\{ ( (0, M_1), L_{j_1}), \dots \big\}$, for $q=$P, B, S, the edge structure and the corresponding edge-weights (the value of $L$) on shortest-length paths between the explored vertices in $\PoiBRW, \BerBRW, \SFPWL$, respectively.
 \item $\CA_n$, with sublists containing the \emph{name, location, vertex-weight, edge-weight to the parent, indicator of being Bernoulli-thinned and scale-free-thinned} of the active individuals, i.e., those that are reachable from an explored individual via a direct edge (in a similar format as that of $\CE_n$). Here  $\ind^{\mathrm{B}}_A(\cdot), \ind^{\mathrm{S}}_A(\cdot)$ stands for the indicator whether the individual belongs to the active set in $\BerBRW, \SFPWL$, respectively.
\item $\CR_n$, the \emph{remaining} edge-weight to the parent for each active individual at time $\tau_n^{\mathrm{P}}$. 
\end{enumerate}
Let $\CF_n:=\sigma( (\CE_k, \CU_k^{\mathrm{P}},\CA_k, \CR_k)_{k\le n})$ be the sigma-algebra generated by the lists by stage $n$ and $\CG_n:=\sigma(\CF_{n-1}\cup\{ \CE_n, \CU_n^{\mathrm{P}}\})$ be an intermediate sigma algebra  ($\CF_{n-1}\subset \CG_n\subset \CF_n$) before determining the active individuals at stage $n$.
For a list of lists $\CC$, let $\CC[i,j]$ denote the $j$th element of the $i$th sublist of $\CC$, while $\CC[\cdot, j]$ denotes the list formed by the $j$th elements of every sublist.  Finally, for a list with elements from $\R$, 
$\CB-x$ denotes a list where we subtract $x$ from each element of $\CB$. Our exploration process is as follows:
 
\emph{(1) (Initialization)} At stage $n=0$,  $\CE_0:=\{(0, \emptyset, W_0,1,1)\}$, that is, the root individual is explored, and its location $M_\emptyset:=0$ and vertex vertex-weight is revealed. 
 We then
 \begin{enumerate}
\item[(i)] Draw a Poisson random variable $D_0^{\text{P}}$ as in \eqref{eq:degree-dist-1} with $x:=\emptyset$, and draw $2d+D_0^{\text{P}}$ many i.i.d.\ variables from distribution $L$, yielding $(L_{\emptyset1}, \dots, L_{\emptyset (2d+D_0^{\text{P}})})$. 
 
\item[(ii)] \emph{(Multiple-edge-thinning step)} Draw the locations of $D_0^{\mathrm{P}}$ many  individuals i.i.d.\ from distribution \eqref{eq:mark-choose} with $x:=\emptyset$. Set the other $2d$ locations to be $0\pm e_i^{(d)}$ for $i\le d$, the where $e_i^{(d)}$ is the $i$th unit vector in $d$ dimensions.
   Mark the second and further edges to any location as \emph{Bernoulli-thinned}, i.e.,  keep only the first occurrence of every location..  Note that the thinning is independent of the realisation of the edge-weights.
\item[(iii)] \emph{(Setting indicator variables)} For each $i\le 2d+D_{j_n}^{\mathrm{P}}$, set 
\be  \ind^{\text{B}}_A(i):= \ind\{ i  \text{ is not Bernoulli-thinned in Step (1ii)}\}=: \ind^{\text{S}}_A( i). 
\ee 
\end{enumerate}
We start the list of \emph{active individuals}. With $\wit D^{\mathrm{P}}_0:=2d+D_0^{\text{P}}$,
\be\label{eq:active-initial} \CA_0:=
\Big\{ \left(1, M_1, W_{M_1}, L_1, \ind^{\text{B}}_A(1), \ind^{\text{S}}_A(1)\right), 
\dots, 
\left(D_0^{\text{P}}, M_{\wit D_0^{\text{P}}}, W_{\wit D_0^{\text{P}}}, L_{\wit D_0^{\text{P}}}, \ind^{\text{B}}_A(\wit D_0^{\text{P}}), \ind^{\text{S}}_A(\wit D_0^{\text{P}})\right) \Big\}. \ee
 We denote by $\sum_{i=1}^{2d+D_0^{\text{P}}}\ind^{\text{B}}_A(i)=:D_0^{\mathrm{B}}$ the degree of the root in the $\BerBRW$. 
 We call this the \emph{Bernoulli thinning}, since we have dropped all multiple edges. 
 We initialize the remaining edge-weight list by  taking the 4th element of each sublist in $\CA_0:$ 
 \be \CR_0:=\{L_1, L_2, \dots, L_{\wit D_0^{\text{P}}}\}.\ee
 For $q=\mathrm{S, B, P}$, we set $\tau_0^{\mathrm{q}}:=0$, and $\CU^{\mathrm{q}}_0=\{\}$ empty.

\emph{(2) (Next-to-explore)} The next (active) individual to explore is the one with minimal remaining edge-weight.
Thus,  let
\be j_n:=\CA_{n-1}[\arg\min \CR_{n-1},1].\ee

\emph{(3) (Time increasing)}
\begin{enumerate}
\item[(i)]
\emph{(Time increasing for $\PoiBRW$)}
Let us set $\tau_{n}^{\mathrm{P}}:=\tau_{n-1}^{\mathrm{P}}+\min \CR_{n-1}$. The \emph{real} time of the process is thus $t=\tau_n^{\mathrm{P}}$ after step $n$.
\item[(ii)] \emph{(Time increasing for $\BerBRW$)}
For $\mathrm{q}=\mathrm{B,S,P}$, let  $f_{\mathrm{q}}(n-1)\le n-1$ be the largest index $k$ for which $\tau^{\mathrm{q}}_k$ is defined after step $n-1$. Thus, there are $f_{\mathrm{B}}(n-1)$ many individuals explored after step $n-1$ in $\BerBRW$.
We \emph{only increase $\tau^\mathrm{B}$} in the exploration of $\BerBRW$ if the individual $j_n$ is part of the Bernoulli-exploration, i.e., it is active in $\BerBRW$.  In this case we increase the index by one and set the last exploration time to $\tau^{\mathrm {B}}_{f_{\mathrm{B}}(n-1)+1}:=\tau_n^{\mathrm{P}}$, the actual time after step $n$. If $j_n$ is Bernoulli-thinned, then we neither increase the number of vertices explored in $\BerBRW$ nor $\tau^{\mathrm{B}}$.
In formulas:
\be\label{eq:tau-b-set}\ba 
\ind^{\text{B}}_E(j_n)&:=\ind^{\text{B}}_A(j_n),\\
\tau_{f_{\mathrm{B}}(n-1)+ \ind_E^{\mathrm{B}}(j_n)}^{\mathrm{B}} &:= \tau_{f_{\mathrm{B}}(n-1) }^{\mathrm{B}}+\ind_E^{\mathrm{B}}(j_n)\left(\tau_n^{\mathrm{P}}-\tau^{\mathrm{B}}_{k_{\mathrm{B}}(n-1)}\right).\ea\ee
\item[(iii)] 
\emph{(Global thinning and time increasing for $\SFPWL$)} 

If $\ind^{\mathrm{S}}_A(j_n)=0$, i.e., the individual $j_n$ is  \emph{not} not active in $\SFPWL$, then we neither increase the last exploration time $\tau^{\mathrm{S}}_{f_{\mathrm{S}}(n)}$, nor add $j_n$ to the explored list in $\SFPWL$. Thus in this case $\ind^{\mathrm{S}}_E(j_n):=\ind^{\mathrm{S}}_A(j_n)=0$.

If $\ind^{\mathrm{S}}_A(j_n)=1$ then $j_n$ is a good possible candidate to explore also in $\SFPWL$. Thus  we check if there is an $i\le n-1$, such that the individual $j_i$ has $\ind^{\text{S}}_E(j_i)=1$ \emph{and} $M_{j_i}=M_{j_n}$. In other words, if we have already explored an individual located at $M_{j_n}\in \Z^d$ in $\SFPWL$. 
If yes, then we call the individual $j_n$ \emph{scale-free thinned} and define $\ind^{\text{S}}_E(j_n):=0$. In this case we do not increase the last exploration time $\tau^{\text{S}}$. 
If there is no such individual, we call the location $M_{j_n}$ \emph{new in $\SFPWL$} and keep the earlier value $\ind^{\text{S}}_E(j_n):=\ind^{\text{S}}_A(j_n)=1$, and increase the last exploration time $\tau^{\text{S}}$ to $\tau_n^{\mathrm{P}}$.
In formulas:
\be\label{eq:tau-s-set} \ba 
\ind^{\text{S}}_E(j_n)&:=\ind^{\text{S}}_A(j_n)\ind\{ M_{j_n} \text{ new in } \SFPWL\}\\
\tau^{\text{S}}_{f_{\text{S}}(n-1) + \ind^{\text{S}}_E(j_n)} &:= \tau^{\text{S}}_{f_{\text{S}}(n-1)} +\ind^{\text{S}}_E(j_n) \left(\tau_n^{\mathrm{P}} - \tau^{\text{S}}_{f_{\text{S}}(n-1)} \right)\\
\ea \ee
\end{enumerate}
We shall see in (5ii) that $\ind_A^{\mathrm{S}}(j_n)\le \ind_A^{\mathrm{B}}(j_n)$, which, combined with comparing \eqref{eq:tau-b-set} to \eqref{eq:tau-s-set} yields that $\ind_E^{\mathrm{S}}(j_n)\le \ind_E^{\mathrm{B}}(j_n)$. Thus, the sequence $(\tau_k^{\mathrm{S}})_{k\ge 1}$ is  a subsequence of $(\tau_k^{\mathrm{B}})_{k\ge 1}$ which, in turn, is a subsequence of $(\tau_k^{\mathrm{P}})_{k\ge 1}$.

\emph{(4) (Renewing the explored list, the used edge-list, and the remaining edge-weight list)}
We refresh
 \[ \CE_n:=\CE_{n-1}\cup \{
 (j_n, M_{j_n}, W_{M_{j_n}}, \ind_E^{\mathrm{B}}(j_n), \ind^{\text{S}}_E(j_n) )\}.\] 
For $\mathrm{q}=\mathrm{P, B, S}$, we add the (location of the endpoints of the) edge between $p(j_n), j_n$ and its length $L_{j_n}$ to the used edges \emph{if} $j_n$ was part of the exploration:
\[ \CU_n^{\mathrm{q}}:=\left\{\ \  \ba &\CU_{n-1}^{\mathrm{q}}\cup \{((M_{p(j_n)}, M_{j_n}), L_{j_n})
\} \mbox{ when }\ind_E^{\mathrm{q}}(j_n)=1,\\
&\CU_{n-1}^{\mathrm{q}} \mbox{ otherwise,} \ea \right. \]
 where $\ind^{\mathrm{P}}_E(j_n):=1$ always.

\emph{(5) (Renewing the active list)}  
To refresh $\CA_{n-1}$, we proceed similarly as in Step (1):
\begin{enumerate}
\item[(i)]Draw the number of children $D_{j_n}^{\mathrm{P}}$ of $j_n$ from the distribution as in \eqref{eq:degree-dist-1} with $x:=j_n$.  
  Draw $\wit D_{j_n}^{\mathrm{P}}:=2d+D_{j_n}^{\text{P}}$ many i.i.d.\ edge-weights from distribution $L$: $(L_{j_n1}, \dots L_{j_n \wit D_{j_n}^{\mathrm{P}}}).$
\item[(ii)] \emph{(Multiple-edge-thinning step)} Draw $D_{j_n}^{\mathrm{P}}$ many  locations i.i.d.\ from the  distribution in \eqref{eq:mark-choose} with $x:=j_n$. Set the other $2d$ locations to be $M_{j_n}\pm e_i^{(d)}$, for $i\le d$. 
Mark second and further occurrences of the same location as Bernoulli-thinned.
Note that the thinning is independent of the realisation of the edge-weights. 
\item[(iii)] \emph{(Setting indicator variables)} For each $i\le \wit D_{j_n}^{\mathrm{P}}$, and $q=B,S$, set 
\be \label{eq:indicator-set} \ind^{\text{q}}_A(j_n i):=\ind^{\text{q}}_E(j_n ) \ind\{ j_ni  \text{ is not Bernoulli-thinned in Step (5ii)}\}. 
\ee 
\end{enumerate}
We then remove $j_n$ from the list of actives and append its children to it:
\be\label{eq:active-refresh} \ba\CA_n:=\CA_{n-1} &\setminus \Big\{\big(j_n, M_{j_n}, W_{M_{j_n}}, L_{j_n}, \ind^{\text{B}}_A(j_n), \ind^{\text{S}}_A(j_n)\big)\Big\} \\
&\quad \cup\Big\{ \big(j_n1, M_{j_n1}, W_{M_{j_n1}}, L_{j_n1}, \ind^{\text{B}}_A(j_n 1), \ind^{\text{B}}_A(j_n 1)\big), \dots \Big.\\\
&\quad\Big.\dots, \big(j_n \wit D_{j_n}^{\mathrm{P}}, M_{j_n \wit D_{j_n}^{\mathrm{P}}}, W_{j_n \wit D_{j_n}}, L_{j_n \wit D_{j_n}^{\mathrm{P}}}, \ind^{\text{B}}_A(j_n \wit D_{j_n}^{\mathrm{P}}), \ind^{\text{B}}_A(j_n \wit D_{j_n}^{\mathrm{P}})\big)\Big\}.\ea\ee

 \emph{(6) (Renewing the remaining edge-weight list)} Finally, we renew the remaining edge-weight list by (a) removing the minimum edge-weight that led to $j_n$, (b) decreasing the other remaining edge-weights in the list by $\min \CR_{n-1}$, (c) appending the new, i.i.d.\ edge-weights to all the newly active children of $j_n$. In formulas,
\[ \CR_n:=\left(\left(\CR_{n-1}- \min \CR_{n-1} \right)\setminus\{0\}\right)\cup \{ L_{j_n 1}, \dots, L_{j_n \wit D_{j_n}^{\text{P}}}\}.\]

\emph{(7) (Repetition)} Increase stage number by $1$ and repeat from (2).
  
Note that the exploration on $\BerBRW$ and on $\SFPWL$ only differs in Step (3ii) versus (3iii). Namely, in $\SFPWL$, an extra thinning is executed by checking that the location of the newly explored vertex has not been visited before in the exploration of $\SFPWL$.

Next we extend the definition of the used edge list and the explored vertex list to the `real time' (that equals $L$-distance from $0$) of the exploration. 
For a time $t\ge 0$  let us define $n(t):=\max\{n: \tau_n^{\mathrm{P}} \le t\}$. 
Let us define  $\wit \CU_t^{\mathrm{q}}:=\CU_{n^{\mathrm{P}}(t)}^{\mathrm{q}}$, and set $\wit \CE_t^{\mathrm{q}}$ be the elements in the list $\CE_{ n^{\mathrm{P}}(t)}$ that have $\ind_E^{\mathrm{q}}(\cdot)=1$.  Similarly, let us denote by $\CE_{n}^{\mathrm{q}}$ those elements in $\CE_n$ that have $\ind_E^{\mathrm{q}}(\cdot)=1$.
\subsection{Coupling SFP to the BRWs}\label{s:coupling-proof}
After having described the joint exploration, we are ready to prove Theorem \ref{thm:thinned-BRW}. First we rephrase Theorem \ref{thm:thinned-BRW} in terms of the exploration. 
\begin{proposition}\label{lemma:thinned-BRW}
Consider $\CB^{L, \mathrm{S}}_t(0)$  as in Theorem \ref{thm:thinned-BRW}. For any $t\ge 0$, the distribution of the location of vertices and their vertex-weights within $\CB^{L, \mathrm{S}}_t(0)$ have the same distribution as   $\wit \CE_{t}^{\mathrm{S}}$, and the edges on the shortest paths towards $0$ in  $\CB^{L, \mathrm{S}}_t(0)$ have the same distribution as $\wit \CU_{t}^{\mathrm{S}}$  in the thinned exploration described above. 
More precisely, a.s.\ under the coupling,
\[ \CB_t^{L, \mathrm{S}}(0)=( \wit \CE_{t}^{\mathrm{S}}, \wit \CU_{t}^{\mathrm{S}})\subseteq  \CB_t^{L,\mathrm{B}}(0)=(\wit \CE_{t}^{\mathrm{B}}, \wit \CU_{t}^{\mathrm{B}} )\subseteq \CB_t^{L,\mathrm{P}}(0)=( \wit \CE_{t}^{\mathrm{P}}, \wit \CU_{t}^{\mathrm{P}}).\]
\end{proposition}
\begin{proof}[Proof of Theorem \ref{lemma:thinned-BRW} subject to Proposition \ref{lemma:thinned-BRW}]
Theorem \ref{lemma:thinned-BRW} an immediate consequence of Proposition \ref{lemma:thinned-BRW}.
\end{proof}
\begin{proof}[Proof of  Proposition \ref{lemma:thinned-BRW}]
Recall from Step (3ii)  that for $\mathrm{q}=\mathrm{S,B, P}$,  $f_{\mathrm{q}}(n):=\sum_{i=1}^n \ind^{\mathrm{q}}_E(j_i)$  denotes the number of vertices explored in each of the processes until step $n$, respectively. Define $f^{(-1)}_{\mathrm{q}}(n)$ as the inverse function of $f_{\mathrm{q}}$.
Naturally, since  $\ind^{\mathrm{P}}_E(\cdot)=1$ always, $f_{\mathrm{P}}(n)=f^{(-1)}_{\mathrm{P}}(n)=n$ while  $f^{(-1)}_{\mathrm{S}}(n)$ gives the \emph{step number} ($\ge n$) when the $n$th vertex is explored in $\SFPWL$.  To show that the distribution of $\CB_t^{L,\mathrm{S}}(0)$ in $\SFPWL$ and $( \wit \CE_{t}^{\mathrm{S}}, \wit \CU_{t}^{\mathrm{S}})$ are the same we argue by induction. 

\emph{Induction hypothesis.} First note that
 $( \wit \CE_{t}^{\mathrm{S}}, \wit \CU_{t}^{\mathrm{S}}) $ does not change between $\tau_{n-1}^{\mathrm{S}}$ and $\tau_n^{\mathrm{S}}$, thus it is enough to check the distributional identity at times $(\tau_n^{\mathrm{S}})_{n\ge 0}$. Since the sequence $(\tau_n^{\mathrm{S}})_{n\ge 0}$ is a subsequence of $(\tau_n^{\mathrm{P}})_{n\ge 0}$, it is enough to check the distributional identity at the latter sequence, equivalently, at each step of the exploration algorithm.   
Thus, our induction hypothesis is that
$( \wit \CE_{\tau_{i}^\mathrm{P}}^{\mathrm{S}}, \wit \CU_{\tau_{i}^\mathrm{P}}^{\mathrm{S}})=( \CE_{i}^{\mathrm{S}},\CU_{i}^{\mathrm{S}})_{i\le n-1}$ has the same distribution as
 $(\CB^{L, \mathrm{S}}_{\tau_{i}^\mathrm{P}}(0))_{i\le n-1}$. 

\emph{Initialisation.} Setting $t=0$, corresponding to $n=0$, yields $\CB^{L, \mathrm{S}}_0(0)\ { \buildrel d \over =}\  ( \wit \CE_{0}^{\mathrm{S}}, \wit \CU_{0}^{\mathrm{S}}) $, since in both sets the vertex set contains $0$, the vertex-weight distribution is the same (a copy of $W$),  while the edge set is empty. This initialises the induction.

 \emph{Advancing the induction.} 
 Note that during Steps (1ii) and (5ii), we thin all multiple edges, while in Step (3iii) we thin those edges that go to locations that have been already allocated to earlier explored individuals in $\SFPWL$, i.e., an application of Step (3iii) corresponds to discovering a (not necessarily edge-disjoint) cycle, i.e., a location $y\in \Z^d$ that is reachable from $0$ on more than one path. Further note that the result of Step (5iii) is that \emph{all the descendants} of a thinned vertex will also be thinned (both in $\BerBRW$ as well as in $\SFPWL$).

Recall that we write $\CE_n^{\mathrm{q}}[\cdot, \ell]$ the list composed by the $\ell$th element of every sublist within $\CE_n^{\mathrm{q}}$, for $\mathrm{q}=\mathrm{S,B,P}$. In particular, $\CE_n^{\mathrm{S}}[\cdot,2]$ gives all the explored locations within the first $n$ steps of the algorithm that are part of $\SFPWL$.

By induction, $( \CE_{n-1}^{\mathrm{S}},\CU_{n-1}^{\mathrm{S}})$ has the same distribution as
 $\CB^{L, \mathrm{S}}_{\tau_{n-1}^\mathrm{P}}(0)$. First we show that 
 
 \emph{($\star$) The edges and their $L$-lengths leading out of $( \CE_{n-1}^{\mathrm{S}},\CU_{n-1}^{\mathrm{S}})=\CB^{L, \mathrm{S}}_{\tau_{n-1}^\mathrm{P}}(0)$ in 
$\Z^d$ that lead to locations outside $\CE_{n-1}^{\mathrm{S}}[\cdot, 2]$ have the same distribution  in the exploration and in $\SFPWL$}. 
\vskip1em

  We thus need to check that the \emph{set of new locations  and their vertex-weights} available from any location $M_{j_i} \in \CE_{n-1}^{\mathrm{S}}[2]$ (corresponding to $j_i\in \CE_{n-1}^{\mathrm{S}}[1]$) in the exploration is the same in the two models, since once this is given, upon the consecutive exploration of these new locations, the thinning is done in a natural way that keeps only the shortest path to the root. 
Recall that the sigma-algebra $\CG_i$ contains all the lists until step $i-1$ but only $(\CE_i, \CU_i)$ but not $\CA_i, \CR_i$, and that we write $\mathcal{N}(x):=\{ y: \| y-x\|=1 \}$ for the nearest-neighbors of $x\in\Z^d$. 

The coupling described between \eqref{eq:poi-vector}-\eqref{eq:mark-choose}
and the fact that the locations of active individuals are determined in Step (5ii) of the algorithm implies  that the number of edges leading to each location $y\in \Z^d, y\not \in \CN(M_{j_i})$ of the individual $j_i\in\CE_{n-1}^{\mathrm{S}}[\cdot,1]\subseteq \CE_{n-1}^{\mathrm{B}}[\cdot, 1]$ at location $M_{j_i}$  in the $\BerBRW$ is distributed as
 \be\label{eq:poi-i} N_y^{\mathrm{B}}(M_{j_i})\mid \CG_i \ {\buildrel d \over =}\ \ind{\Big\{\mathrm{Poi}\Big( \la  W_{M_{j_i}} \frac{W_y}{\|y-M_{j_i}\|^\al}\Big) \ge 1\Big\}}, \ee
 Thus, the probability that the location  $y \notin \mathcal E_{n-1}^{\mathrm{S}}[\cdot,2], y\not\in \CN(M_{j_i})$ appears among the list of children of $j_i\in\CE_{n-1}^{\mathrm{S}}[\cdot,1]$
in the $\BerBRW$ at time $\tau_i^{\mathrm{P}}$ (after step $i$) is given by
\be\label{eq:dist-identity}\Pv(y  \leftrightarrow M_{j_i} \mid \CG_{i})=\Pv\left( \mathrm{Poi}\Big( \frac{W_{M_{j_i}}   W_{y}}{\|y - M_{j_i}\|^\al}\Big)\ge 1 \ \big\lvert \ \CG_{i}\right)= 1- \exp \left\{ - \la\frac{W_{M_{j_i}}   W_{y}}{\|y - M_{j_i}\|^\al}   \right\},\ee 
which is exactly the same as having an edge between location $M_{j_i}$ with vertex-weight $ W_{M_{j_i}}$ and any $y$ with vertex-weight $W_y$ from distribution $W$ in the scale-free percolation model. Since the environment is fixed in advance, and $(W_y)_{y\in \Z^d}$ is independently drawn from everything else, the vertex-weight distribution is matching. The multinomial thinning of the Poisson variable in Steps (1ii) and (5ii) ensures that the edges from $M_{j_i}$ going to different locations $y\in \Z^d$ are conditionally independent given $\CG_{i}$. 
The addition of nearest neighbor edges in Steps (1ii) and (5ii) ensures that the nearest-neighbor edges $M_{j_i}\pm e_k$, for all $k\le d$, are always present.

   Let us write $\tau_{(j_i)}\in \{\tau_{1}^{\mathrm{S}}, \dots,\tau_{f_{\mathrm{S}}(n-1)}^{\mathrm{S}}\}$ for the time when we explored $j_i$. Suppose now  that the edge $(M_{j_i}, y)$ is present, say, it belongs to the child $j_ik$ of $j_i$.
 The fact that we thinned every multiple edge in the active list ensures that each such edge $(M_{j_i}, y)$ is allocated only \emph{one} edge-weight with distribution $L$. Thus, the exploration will explore the location $y$ from $M_{j_i}$ precisely at time $\tau_{(j_i)}+L_{j_ik}$. At that moment, (which is after stage $n-1$ by the assumption that $y\notin \CE_{n-1}^{\mathrm{S}}[\cdot,2]$) the location $y$ might have already been explored via another path that is not contained entirely in $\CE_{n-1}^{\mathrm{S}}$, thus it might be thinned by Step (3iii). Nevertheless, the rate of exploring this edge, given that it is there, is precisely the same in the two models.\footnote{This is why deleting multiple edges in Steps (1ii) and (5ii) was necessary.  If we would not have done this, then whenever there are $k$ multiple edges between $y$ and $M_i$, we explore $y$ from $M_i$ first at time $\tau_{(i)}+\min (L_{i,1}, \dots, L_{i,k} )$, which does not have the right distribution.}
 The distribution of the other edge-weights leaving $(\CE_i^{\mathrm{S}}, \CU_i^{\mathrm{S}})$ continue to match for the following reason: when we determined $j_i$, we took the minimum outgoing remaining edge-weight  from $\CR_{i-1}$, thus increasing time by $\min \CR_{i-1}=\tau_{i}^{\mathrm{P}}-\tau_{i-1}^{\mathrm{P}}$. Then, all the outgoing edge-weights on edges leaving the set  $\CB^{L, \mathrm{S}}_{\tau_{i}^{\mathrm{P}}}(0)$ versus $\CB^{L, \mathrm{S}}_{\tau_{i-1}^{\mathrm{P}}}(0)$ have to be decreased by $\tau_{i}^{\mathrm{P}}-\tau_{i-1}^{\mathrm{P}}$, and this is precisely what the algorithm does in Step (6). This shows $(\star)$.
 
To finish the induction, we argue as follows. By $(\star)$, from every individual in $\CE_{n-1}^{\mathrm{S}}[\cdot, 1]$ and from $\CB^{L, \mathrm{S}}_{\tau_{n-1}^{\mathrm{P}}}(0)$, the number of edges and corresponding lengths to every new location in $\Z^d$ has the same distribution. In other words, the edge-weights within $\CR_{n-1}$ that correspond to leading to new locations, also have the same distribution. Taking the minimum of the remaining edge-weight list $\CR_{n-1}$ yields $j_n$, the individual  to be explored at step $n$. At this point we have to distinguish several cases:

If $j_n$ has $\ind^{\mathrm S}_A (j_n)=0$, then $j_n$ is not active in the $\SFPWL$ exploration, she will not become explored in $\SFPWL$, $\ind^{\mathrm S}_E(j_n)=0$ will be set in Step (3iii), and the last exploration time $\tau_{f_{\mathrm{S}}(n-1)}^\mathrm{S}$ in $\SFPWL$ remains unchanged by \eqref{eq:tau-s-set}, so there is nothing to prove.

If $\ind^{\mathrm S}_A (j_n)=1$, and the location $M_{j_n}$ is not a new location\footnote{Since $\CE^{\mathrm{S}}_{n-1}$ contains those elements of $\CE_{n-1}$ that have $\ind^{\text{S}}_E(j_i)=1$,  this means precisely that there is an $i\le n-1$, such that the individual $j_i\in \CE_{n-1}[1]$ has $\ind^{\text{S}}_E(j_i)=1$ \emph{and} $M_{j_i}=M_{j_n}$.}, i.e., $M_{j_n} \in \CE^{\mathrm{S}}_{n-1}[2]$,
then the location $M_{j_n}$ have been explored earlier. That is, the shortest path to this location is not via $j_n$, i.e., exploring $j_n$ would lead to a longer path.
In this case, Step (3iii) thins $j_n$, that is, $\ind^{\mathrm S}_E(j_n)=0$ and the last exploration time $\tau_{f_{\mathrm{S}}(n-1)}^\mathrm{S}$ remains unchanged again by \eqref{eq:tau-s-set}.

In the previous two cases, all the descendants of $j_n$ will also have $\ind_A^{\mathrm{S}},\ind_E^{\mathrm{S}}=0$ so they will not be added to the exploration either.

If  $\ind^{\mathrm S}_A (j_n)=1$ and $M_{j_n}\notin \CE^{\mathrm{S}}_{n-1}[2]$, then the location $M_{j_n}$ is explored for the first time in $\SFPWL$.  In this case, 
\[ \CE_{n}^{\mathrm{S}}=\CE_{n-1}^{\mathrm{S}}\cup\{(j_n, M_{j_n}, W_{j_n}, 1,1\} \mbox{ and }\CU_{n}^{\mathrm{S}}=\CE_{n-1}^{\mathrm{S}}\cup\{((M_{p(j_n)}, M_{j_n}), L_{j_n})\}.\] 
We argue that these two lists together have the same distribution as $\CB^{L, \mathrm{S}}_{\tau_{n}^{\mathrm{P}}}(0)$.  By induction again, $\CU_{n-1}^{\mathrm S}$ describes the shortest paths to $0$ from all other vertices within $\CE_n^{\mathrm{S}}[1]=\CE_{n-1}^{\mathrm{S}}[1]\cup \{j_n\}$  to $0$ within $\CB^{L, \mathrm{S}}_{\tau_{n-1}^{\mathrm{P}}}(0)$, so, we only have to show that the distribution of the location and vertex-weight of $j_n$ is that of the last discovered vertex in $\CB^{L, \mathrm{S}}_{\tau_{n}^{\mathrm{P}}}(0)$ and that the shortest path from this vertex is contained in $\CU_n^{\mathrm{S}}$. 
These are direct consequences  of $(\star)$: Step (2) of the algorithm determined the next-to-explore vertex $j_n$, based on which individual is closest to the explored vertices in $\CB^{L, \mathrm{S}}_{\tau_{n-1}^{\mathrm{P}}}(0)$ in terms of $L$-distance.  This individual happened to be at a new location in the exploration, thus, by $(\star)$, its vertex-weight was drawn i.i.d.\ from $W$ at stage $p(j_n)$, and we an interpret the length $\min \CR_{n-1}$ as the distance of the vertex $M_{j_n}\in \Z^d$ from the set $\CB^{L, \mathrm{S}}_{\tau_{n-1}^{\mathrm{P}}}(0)$.
  Since the location $M_{j_n}$ is explored via the individual $j_n$ with parent $p(j_n)$ for the first time, any consecutive exploration of the location $M_{j_n}$ will happen later, thus yielding longer paths. Thus, the shortest path from $M_{j_n}$ to the origin starts with  the edge $(M_{j_n}, M_{p(j_n)})$. By induction, the path $M_{p(j_n)}, M_{p(p(j_n))}, \dots, 0$, contained in $\CU_{n-1}^{\mathrm{S}}=\wit \CU_{\tau_{n-1}^{\mathrm{P}}}^{\mathrm{S}}$ is the shortest path from $M_{p(j_n)}$ to the origin. Thus,  the path $M_{j_n}, M_{p(j_n)}, M_{p(p(j_n)}, \dots, 0$ gives the shortest path to the origin from $M_{j_n}$ and is contained in $\CU_{n}^{\mathrm{S}}$. This establishes the statement that $\CU_n^{\mathrm{S}}$ contains the shortest path structure to the origin in $\CB^{L, \mathrm{S}}_{\tau_{n}^{\mathrm{P}}}(0)$. 
 This finishes the induction. \end{proof}

\bibliographystyle{abbrv}
\bibliography{refscompetition}

\end{document}